%% file: FrascaCacciaHydon2020v2.tex
\documentclass[12pt]{article}

\usepackage[english]{babel}
\usepackage[T1]{fontenc}
\usepackage{amsmath}
\usepackage{amsfonts}
\usepackage{amssymb}
\usepackage[utf8]{inputenc}
\usepackage{graphicx,color,url}
\usepackage[margin=2.4cm]{geometry}
\usepackage{booktabs}
\usepackage{cite}
\usepackage{caption}
\usepackage{tikz,pgfplots}
        \pgfplotsset{compat = 1.3}
        \pgfplotsset{minor grid style={dotted}} \pgfplotsset{major grid
        style={dashed}}

        \pgfplotsset{every x tick label/.append style={font=\footnotesize,
        yshift=0.25ex}}

        \pgfplotsset{every y tick label/.append
        style={font=\footnotesize, xshift=0.25ex}}

\usepackage{enumitem}
\newtheorem{cor}{Corollary}
\newtheorem{theo}{Theorem}
\newtheorem{rem}{Remark}

\def\bk{\mathbf{k}}
\def\bx{\mathbf{x}}

\def\bg{\mathbf{g}}
\def\bh{\mathbf{h}}

\def\bu{\mathbf{u}}
\def\bU{\mathbf{U}}
\def\bw{\mathbf{w}}

\def\mca{\mathcal{A}}
\def\mcq{\mathcal{Q}}
\def\mce{\mathcal{E}}

\def\pmatrix{\left(\begin{array}}
\def\endpmatrix{\end{array}\right)}
\def\dd{\mathrm{d}}

\def\sech{\mathrm{sech}}

\def\sech{\mathrm{sech}}

\newcommand*{\QEDB}{\hfill\ensuremath{\square}}
\newcommand{\qed}{\tag*{$\square$}}

\newcommand{\Dc}{D^{(c)}}

\author{G.\,Frasca-Caccia\,\quad P.\,E.\,Hydon\,\\[.5cm]
\small
School of Mathematics, Statistics and Actuarial Science\\
\small
University of Kent, Canterbury, CT2 7NZ\\
\small}

\begin{document}
\title{A new technique for preserving conservation laws}

\author{
{\sc G.\,Frasca-Caccia,\,\quad P.\,E.\,Hydon
} \\[2pt]
School of Mathematics, Statistics and Actuarial Science\\
University of Kent, Canterbury, CT2 7FS, UK}

\maketitle

\begin{abstract} {This paper introduces a new symbolic-numeric strategy for finding semidiscretizations of a given PDE that preserve multiple local conservation laws. We prove that for one spatial dimension, various one-step time integrators from the literature preserve fully-discrete local conservation laws whose densities are either quadratic or a Hamiltonian. The approach generalizes to time integrators with more steps and conservation laws of other kinds; higher-dimensional PDEs can be treated by iterating the new strategy. We use the Boussinesq equation as a benchmark and introduce new families of schemes of order two and four that preserve three conservation laws. We show that the new technique is practicable for PDEs with three dependent variables, introducing as an example new families of second-order schemes for the potential Kadomtsev-Petviashvili equation.


\textbf{AMS class:} \textit{65M06; 37K05; 39A14.}
}
\end{abstract}

\section{Introduction}
\label{intro}
Consider a system of $q$ partial differential equations (PDEs),
\begin{equation}\label{PDE}
\mca(x,t,[\bu])=\mathbf{0},
\end{equation}
where $\mca$ is a row vector, $\bu$ has components $u^\alpha,\ \alpha=1,\dots,q$, and square brackets around a differentiable expression denote the expression and finitely many of its derivatives\footnote{To simplify the presentation, we consider here only one spatial variable, $x$. The extension to PDEs with more spatial variables is outlined in Section~\ref{genericCL} and illustrated in Section~\ref{pKPsec}.}.
We assume that (\ref{PDE}) is totally nondegenerate (see \cite{olverbook}). A local conservation law is a divergence expression,
\begin{equation}\label{CLaw} 
\mbox{Div}\, \mathbf{F}=D_x\{F(x,t,[\bu])\}+D_t\{G(x,t,[\bu])\},
\end{equation}
that vanishes on all solutions of \eqref{PDE}. The functions $F$ and $G$ are the flux and the density of the conservation law, respectively, and $D_x$ and $D_t$ denote the total derivatives with respect to $x$ and $t$, respectively. The conservation law \eqref{CLaw} is in characteristic form if there exists a column vector $\mcq$ such that
\begin{equation}\label{CLchar}
\mbox{Div}\,\mathbf{F}=\mca\mcq,
\end{equation}
in which case $\mcq$ is called the characteristic. The space of total divergences forms the kernel of the Euler operator, $\mce$, whose $\alpha$-th entry is
\begin{equation}\label{Eul}
\mce_\alpha=\sum_{i,j}(-D_x)^i(-D_t)^j\frac{\partial}{\partial(D_x^iD_t^ju^\alpha)}.
\end{equation} 
Hence
\begin{equation}\label{eulcond}
\mce(\mca\mcq)=\mathbf{0}
\end{equation}
if and only if there exists $\mathbf{F}$ such that \eqref{CLchar} holds. These results generalize immediately to PDE systems with more than two independent variables.

The literature on the numerical solution of PDEs is rich in numerical methods that preserve global invariants, but there are relatively few results on the preservation of local conservation laws. Arguably, local conservation laws are more necessary: they hold throughout the domain, apply to the set of all solutions, and provide much stronger constraints than are needed to preserve the corresponding global invariants. Moreover, when the domain and boundary conditions are suitable, conservation of such invariants is automatically achieved.

A new approach for developing finite difference schemes that preserve conservation laws of (\ref{PDE}) was introduced recently in \cite{FCHydon}. It exploits the fact that discrete conservation laws form the kernel of a discrete version of the Euler operator (\ref{Eul}). Discretizations of the PDE (\ref{PDE}) having discrete versions of the desired conservation laws are obtained by requiring that a discrete version of condition (\ref{eulcond}) is satisfied.
This requires the symbolic solution of a large system of nonlinear equations that is impractical in general.
The complexity of the symbolic calculations can be reduced by introducing compactness requirements on the schemes, and this direct approach has been applied to a range of PDEs with different structure in \cite{FCHydon,FCHmkdv,FCHnls}. However, the direct approach is greatly limited by the capacity of symbolic computation; it has been applied only to second-order approximations of PDEs with two independent variables.

In this paper we modify the approach in \cite{FCHydon} by finding semidiscretizations of (\ref{PDE}) that preserve semidiscrete local conservation laws. The reduction to one discrete space dimension significantly reduces the complexity of the computations, to the point that the determining system can be solved easily without introducing any restrictions on the form of the schemes. After this, a suitable integrator in time needs to be chosen to create a fully-discrete scheme; this depends on the form of the conservation laws that one aims to preserve. 

If the PDE is equipped with conservative boundary conditions, it is known that the quadratic invariants of its space discretizations are preserved by symplectic Runge-Kutta methods \cite{Coop,SSC,CIZ}. In this paper we extend this result to prove that if $G$ is quadratic in $[u]$ then any symplectic Runge--Kutta method preserves the conservation law (\ref{CLaw}) locally, regardless of the boundary conditions.

There are various results on local conservation for Hamiltonian PDEs,
\begin{equation}\label{HamPDE}
D_t \bu=\mathcal{D}\left([\bu]_x\right)\mce(H\left([\bu]_x\right)),
\end{equation}
where $[\bu]_x$ denotes $\bu$ and its spatial derivatives only, $\mathcal{D}$ is a skew-adjoint operator that satisfies the Jacobi identity, and $H$ is the Hamiltonian function. 

Multisymplectic schemes \cite{BR} and their generalizations \cite{Sun} can preserve local conservation laws with quadratic flux and density. Requiring the flux to be quadratic is however a strong constraint that is not satisfied by local momentum conservation laws of many important equations in physics such as the nonlinear Schr\"odinger (NLS) equation, the Korteweg-de Vries (KdV) equation, the Benjamin-Bona-Mahony (BBM) equation, the modified Korteweg-de Vries (mKdV) equation, and the Boussinesq equation. The strategy introduced in this paper does not suffer from this restriction, as no assumption is needed about the flux, so it can be applied to the preservation of these conservation laws as well.

Another popular approach is to use a discrete gradient method for the time integration of (\ref{HamPDE}).
These are obtained from a semidiscretization of $H$ and a skew-adjoint discretization of $\mathcal{D}$, and preserve a discrete conservation law of the energy \cite{QMcL}. One widely-used discrete gradient method is the Average Vector Field (AVF) method \cite{Cell,Cell2,McLaren}. We show that the AVF method yields the local conservation law of the Hamiltonian under constraints on the discretization of $\mathcal{D}$ that are milder than skew-adjointness. Consequently, conservation of the local Hamiltonian can be achieved for a larger class of discretizations.

Although the discussion so far has focused on PDEs with two independent variables, the approach of discretizing one dimension at a time works equally well for PDEs on higher-dimensional spaces. We discuss this and give an illustration.

The paper is organised as follows. Section~\ref{spacedis} introduces a method for obtaining conservative spatial semidiscretizations. Section~\ref{timeint} focuses on time integration. In particular, we show the following.
\begin{itemize}
\item[$\bullet$] Conservation laws with quadratic density (without any assumption on the flux) are preserved by any symplectic method in time locally and independently of the boundary conditions. Conservation laws for mass, charge and momentum are typically in this class.
\item[$\bullet$] For Hamiltonian PDEs, the AVF method preserves the local semidiscrete conservation law of the energy for a wide class of semidiscretizations that generalizes the result in \cite{QMcL}.
\item[$\bullet$] For other types of conservation law, fully-discrete methods can be found by introducing relatively few parameters and fixing them by requiring that the conservation law is in the kernel of a fully-discrete Euler operator. This approach can be iterated for dimensions, by using a sequence of semidiscrete and discrete Euler operators.
\end{itemize} 

In Section~\ref{GBsec} we apply this new approach to the Boussinesq equation and introduce methods of order two and four that preserve three conservation laws. Section~\ref{Numsec} describes numerical benchmark tests, including evidence of stability and comparison with other methods from the literature. In Section~\ref{pKPsec}, we apply the new technique to a two-dimensional PDE, the potential Kadomtsev-Petviashvili (pKP) equation and introduce two families of schemes that preserve two conservation laws. Finally, we draw some conclusions in Section~\ref{concl}.

\section{Conservative space discretizations}\label{spacedis}
We begin with a regular spatial grid. The stencil consists of $M=B-A+1$ nodes, \begin{equation}\label{nodesx}
x_{m}=x_0+m\Delta x, \qquad m=A,\ldots,B,
\end{equation}
where $x_0$ is a generic grid point; let $\bx$ denote the vector of the nodes. 
The forward shift operator $S_m$ acts as follows on any semidiscrete function $f$:
\begin{equation*}
S_m:f(x_m,t)\mapsto f(x_{m+1},t);
\end{equation*}
the forward difference, forward average, and centered difference operators are
\begin{equation}\label{DmMm}
D_m=\tfrac{1}{\Delta x}(S_m-I),\qquad \mu_m=\tfrac{1}2(S_m+I),\qquad \Dc_m=\tfrac{1}{2\Delta x}(S_m-S_m^{-1})
\end{equation}
respectively, where $I$ is the identity operator.
The semidiscretizations of $u^\alpha\!(x,\!t)$ are given by the column vector $\bU\in\mathbb{R}^{Mq}$ with the {$(m+\alpha M-B)$-th} entry 
$$U_{m}^\alpha(t)\approx u^\alpha(x_m,t),\qquad m=A,\ldots,B,\quad \alpha=1,\ldots,q.$$
The semidiscrete problem is
\begin{equation}\label{SDPDE}
\widetilde{\mca}(\bx,t,[\bU])=\mathbf{0},
\end{equation}
where here and henceforth tildes represent approximations to the corresponding continuous quantities, and square brackets around a semidiscrete expression denote the expression and a finite number of its  time derivatives.

A semidiscrete conservation law of (\ref{SDPDE}) is a semidiscrete divergence,
\begin{equation}\label{SDCL}
\text{Div}\,\widetilde{\mathbf{F}}=D_m\{\widetilde{F}(\bx,t,[\bU])\}+D_t\{\widetilde{G}(\bx,t,[\bU])\},
\end{equation}
such that $$\text{Div}\,\widetilde{\mathbf{F}}=0,\,\,\,\text{when}\,\,\,[\widetilde{\mca}=\mathbf{0}].$$
The functions $\widetilde{F}$ and $\widetilde{G}$ are the semidiscrete flux and density of the conservation law (\ref{SDCL}), respectively. Similarly, as in the continuous case, we say that (\ref{SDCL}) is in characteristic form if there exists $\widetilde{\mcq}=\widetilde{\mcq}(\bx,t,[\bU])$, called the characteristic, such that $$\text{Div}\,\widetilde{\mathbf{F}}=\widetilde{\mca}\widetilde{\mcq}.$$
The following result is crucial for obtaining semidiscretizations that preserve conservation laws (see \cite{olverbook} and \cite{kuper} for analogous results in the continuous and totally discrete setting, respectively).
\begin{theo}\label{th1}
The kernel of the semidiscrete Euler operator $\mathsf{E}_\bU$, whose $\alpha$-th entry is 
\begin{equation*}
(\mathsf{E}_\bU)_\alpha=\sum_{i,j}S_m^{-i}(-D_t)^j\frac{\partial}{\partial (D_t^jU_{i}^\alpha)},
\end{equation*}
is the space of semidiscrete divergences (\ref{SDCL}).
\end{theo}
\begin{proof}
Let $L=L(\bx,t,[\bU])$ such that $\mathsf{E}_\bU(L)=\mathbf{0}$, and consider the derivative
\begin{equation}\label{dLdeps}
\frac{\dd}{\dd \varepsilon}L(\bx,t,\varepsilon [\bU])=\sum_{\alpha,i,j}(D_t^jU_i^\alpha)\frac{\partial L(\bx,t,\varepsilon[\bU])}{\partial(D_t^jU_i^\alpha)}.
\end{equation}
Integrating by parts yields
\begin{align*}
(D_t^jU_i^\alpha)\frac{\partial L}{\partial(D_t^jU_i^\alpha)}&\,=U_i^\alpha(-D_t)^j\frac{\partial L}{\partial(D_t^jU_i^\alpha)}+D_t \hat G\\
&\,=U_0^\alpha{S_m^{-i}(-D_t)^j\frac{\partial L}{\partial(D_t^jU_i^\alpha)}}+D_m \hat F+D_t \hat G=D_m \hat F+D_t \hat G,
\end{align*}
for some functions $\hat F=\hat F(\bx,t,\varepsilon,[\bU])$ and $\hat G=\hat G(\bx,t,\varepsilon,[\bU])$ whose precise expression is not of importance. Substituting this into (\ref{dLdeps}) and integrating over $\varepsilon\in [0,1]$ shows that $L$ is a semidiscrete divergence.

If $L$ is of the form (\ref{SDCL}), $\mathsf{E}_\bU(L)=\mathbf{0}$ follows from the linearity of the Euler operator and from the fact that for any $k$ (see, e.g., \cite{hydonbook}),
\begin{align*}
\left(\sum_{i}S_m^{-i}\frac{\partial}{\partial (D_t^k\bU_{i}^\alpha)}\right)(D_m\widetilde{F})=0,\,\,\, \left(\sum_{j}(-D_t)^j\frac{\partial}{\partial (D_t^j\bU_{k}^\alpha)}\right)(D_t\widetilde{G})=0. \qed
\end{align*}
\end{proof}
Based on the result in Theorem~\ref{th1}, the approach used in \cite{FCHydon, FCHmkdv,FCHnls} to preserve fully-discrete conservation laws, is adapted here to the preservation of semidiscrete conservation laws of (\ref{PDE}) with characteristics $\mcq_\ell$, as follows:
\begin{enumerate}
\item Select a stencil that is large enough to support generic semidiscretizations for $\mca$ and every $\mcq_\ell$, having the desired order of accuracy. These approximations depend on a number of free parameters to be determined.
\item Find some of the parameters by imposing consistency, up to the desired order of accuracy, $p$.
\item Use symbolic algebra to determine the values of the free parameters that satisfy
\begin{equation}\label{sdeulcond}
\mathsf{E}_\bU(\widetilde{\mca}\widetilde{\mcq}_\ell)=\mathbf{0},
\end{equation}
for $\ell=1$. This guarantees that the first conservation law is locally preserved. As both $\widetilde{\mca}$ and $\widetilde{\mcq}_\ell$ are accurate to order $p$, the discrete conservation law has the same order of accuracy.
\item Iterate the previous step, replacing $\widetilde{\mcq}_1$ with $\widetilde{\mcq}_\ell$, to obtain further constraints on the parameters. If (\ref{sdeulcond}) has no solution for some $\ell$, the corresponding conservation law cannot be preserved without violating one of the previous conservation laws. Typically, the more complicated a conservation law is, the more parameters need to be fixed to preserve it.
\end{enumerate}
\begin{rem}
It might seem appealing to identify a set of conservation laws that one wishes to preserve and use brute force symbolic computation to solve all constraints simultaneously. (This was our approach initially.) However, this takes far longer than the sequential approach and commonly comes up with a null result, with no indication as to which subsets of conservation laws can be preserved. The sequential algorithm above enables the user to decide which conservation laws should be prioritized. At each iteration, the computation is simplified by the fact that some parameters have already been fixed.
\end{rem}
\begin{rem}
If the algorithm above does not produce any schemes for a given stencil, one could try preserving the same conservation laws using a wider stencil. However, the wider the stencil is, the more the computational cost increases. Moreover, if one finds a conservative semidiscretization, a time integrator that preserves all the conservation laws is also needed. For example, in the next section we prove that some known time integrators preserve conservation laws whose density is either quadratic, or is a Hamiltonian, but to the best of our knowledge there are no methods that preserve both of these types.
\end{rem}

\section{Time integration}\label{timeint} We begin by considering one-step time integrators. For fully-discrete schemes the stencil is $$(x_{m},t_n),\quad m=A,\ldots,B,\quad n=0,1, \quad t_1=t_0+\Delta t,$$ 
and the forward shift operators in space and time are
$$S_m:f(x_m,t_n)\mapsto f(x_{m+1},t_n),\qquad S_n:f(x_m,t_0)\mapsto f(x_m,t_{1}),$$
respectively. The forward difference and forward average operators in space are defined by (\ref{DmMm}) and the corresponding operators in time are
$$D_n=\tfrac{1}{\Delta t}(S_n-I),\qquad \mu_n=\tfrac{1}2(S_n+I).$$
Let $\bu_{n}\in\mathbb{R}^{Mq}$ be the column vector whose $(m+\alpha M-B)$-th entry is
$$u_{m,n}^\alpha\approx u^\alpha(x_m,t_n), \qquad m=A,\ldots,B,\quad \alpha=1,\ldots,q,$$
and let  $\bu_{m,n}\in\mathbb{R}^{q}$ be the column vector with entries
$$u_{m,n}^\alpha\approx u^\alpha(x_m,t_n), \qquad \alpha=1,\ldots,q.$$
\subsection{Conservation laws with quadratic density}\label{Clquadden}
Here attention is restricted to PDEs of the form
\begin{equation}\label{PDEtimesec}
D_t\left\{\bg(x,[\bu]_x)\right\}=\bh(x,t,[\bu]_x),
\end{equation}
where $\bg$ is linear homogeneous in $[\bu]_x$; these include Hamiltonian PDEs.
Consider a conservation law of (\ref{PDEtimesec}) of the form
\begin{equation}\label{momentum}
D_x\{F_2(x,t,[u]_x,[u_t]_x)\}+D_t\{G_2(x,[u]_x)\}=0,
\end{equation}
where the density, $G_2$, is a polynomial of degree two in $[u]_x$. (Without loss of generality, assume that no terms in $G_2$ depend on $x$ only.)
For many differential problems of importance in physics (such as KdV, NLS and BBM equations), the conservation laws of mass (or charge) and momentum are of the form (\ref{momentum}) with linear and quadratic density, respectively.  

Let $P(\bx)$ be an invertible operator such that
\begin{equation*}
\widetilde{\bg}(\bx,\bU)=P^{-1}(\bx)\bU
\end{equation*} 
and $\widetilde{\bh}(\bx,t,\bU)$ are two column vectors whose $(m+\alpha M-B)$-th entry is a spatial discretization of the $\alpha$-th component of $\bg(x,[\bu]_x)$ and $\bh(x,t,[\bu]_x)$ at $x_m$, respectively. Let
\begin{equation}\label{SDcon}
D_t\{\widetilde{\bg}(\bx,\bU)\}=\widetilde{\bh}(\bx,t,P(\bx)\widetilde{\bg}(\bx,\bU))
\end{equation}
be a semidiscretization of the PDE (\ref{PDEtimesec}) with the following approximation to the conservation law (\ref{momentum}): 
\begin{equation}\label{SDmom}
D_m\{\widetilde{F}_2(\bx,t,\bU,\bU_t)\}+D_t\{\widetilde{G}_2(\bx,\bU)\}=0.
\end{equation}
Such a semidiscretization can be obtained using the technique in Section~\ref{spacedis}. 
The flux and density of (\ref{SDmom}) have the general form
\begin{align*}
\widetilde{F}_2(\bx,t,\bU,\bU_t)&=\widetilde{F}_2\left(\bx,t,\bU,P(\bx)\widetilde{\bh}(\bx,t,\bU)\right),\\
\widetilde{G}_2(\bx,\bU)&=\tfrac{1}{2}\bU^T S(\bx)\bU+\bw(\bx)^T\bU,
\end{align*}
where $S(\bx)=S(\bx)^T\in\mathbb{R}^{Mq\times Mq}$ and $\bw(\bx)\in\mathbb{R}^{Mq}$ is a column vector. 

The following theorem shows that symplectic Runge--Kutta methods preserve local conservation laws with quadratic density. The proof adapts Calvo, Iserles and Zanna's proof that such methods preserve quadratic invariants of systems of ODEs \cite{CIZ}, to take contributions from the flux into account.
\begin{theo}\label{th2}
The solution of any symplectic Runge--Kutta method applied to (\ref{SDcon}) satisfies a discrete version of (\ref{SDmom}).
\end{theo}
\begin{proof} The conservation law (\ref{SDmom}) amounts to
\begin{align}\label{cond1}\nonumber
D_m&\!\left\{\widetilde{F}_2\left(\bx,t,P(\bx)\widetilde \bg,P(\bx)\widetilde \bh\left(\bx,t,P(\bx)\widetilde \bg\right)\right)\right\}=-D_t\{\widetilde{G}_2(\bx,P(\bx)\widetilde \bg)\}\\ 
&=-\left((P(\bx)\widetilde\bg)^TS(\bx)+\bw(\bx)^T\right)P(\bx)\widetilde\bh(\bx,t,P(\bx)\widetilde\bg).
\end{align} 
Solving (\ref{SDcon}) using a $s$-stage symplectic Runge--Kutta method,
\begin{equation}\label{RK}
\widetilde\bg_{n+1}=\widetilde\bg_{n}+\Delta t\sum_{i=1}^sb_i \widetilde\bh(\bx,t_n+c_i\Delta t,P(\bx)\bk_i)\equiv\widetilde\bg_{n}+\Delta t\sum_{i=1}^sb_i \widetilde\bh_i,
\end{equation}
with internal stages
\begin{equation}\label{kstage}
\bk_i=\widetilde\bg_n+\Delta t\sum_{j=1}^sa_{i,j}\widetilde\bh_j,\quad i=1,\ldots,s,
\end{equation}
we obtain $\bu_n=P(x)\widetilde\bg_n$. Moreover,
\begin{align*}
\widetilde{G}_2&(\bx,\bu_{n+1})=\tfrac{1}2\bu^T_{n+1} S(\bx)\bu_{n+1}+\bw(\bx)^T\bu_{n+1}=\left(\tfrac{1}2 (P(\bx)\widetilde\bg_{n+1})^T S(\bx)+\bw(\bx)^T\right)P(\bx)\widetilde\bg_{n+1}\\
\,=&\left(\tfrac{1}2 \bu_n^T S(\bx)+\bw(\bx)^T\right)\bu_n+\Delta t\sum_{i=1}^sb_i\left((P(\bx)\widetilde\bg_n)^T S(\bx)+\bw(\bx)^T\right)P(\bx)\widetilde\bh_i\\
&+\tfrac{\Delta t^2}2\sum_{i,j=1}^sb_ib_j \left(P(\bx)\widetilde\bh_j\right)^TS(\bx)P(\bx)\widetilde\bh_i.
\end{align*}
Using (\ref{kstage}) to eliminate $\widetilde\bg_n$ from the first sum and rearranging, gives
\begin{align*}
\widetilde{G}_2(\bx,\bu_{n+1})=&\,\widetilde{G}_2(\bx,\bu_{n})+\Delta t\sum_{i=1}^sb_i\left( (P(\bx)\bk_i)^T S(\bx)+\bw(\bx)^T\right)P(\bx)\widetilde\bh_i\\
&+\!\tfrac{\Delta t^2}2\sum_{i,j=1}^s\!(b_jb_i\!-\!b_ia_{i,j}\!-\!b_ja_{j,i})\left(P(\bx)\widetilde\bh_j\right)^T\!\!\!S(\bx)P(\bx)\widetilde\bh_i.
\end{align*}
The condition of symplecticity,
\begin{equation*}
b_ia_{i,j}+b_ja_{j,i}-b_ib_j=0,\quad \forall\, i,j=1,2,\ldots,s,
\end{equation*}
and (\ref{cond1}) give
\begin{equation*}
D_m\left\{\sum_{i=1}^s b_i\widetilde{F}_2\left(\bx,t_n+c_i\Delta t,P(\bx)\bk_i,P(\bx)\widetilde\bh_i\right)\right\}+D_n\{\widetilde{G}_2(\bx,\bu_n)\}=0,
\end{equation*}
which is an approximation of (\ref{momentum}). $\QEDB$
\end{proof}
\begin{rem}
Multisymplectic methods preserve conservation laws whose density and flux are both quadratic. By contrast, Theorem~\ref{th2} applies to all conservation laws that have quadratic density. As no assumption is needed on the flux, a larger class of conservation laws can be preserved. 
\end{rem}
The following results follow directly from the proof of Theorem~\ref{th2}.
\begin{cor}\label{cor1}
Any Runge--Kutta method preserves semidiscrete local conservation laws whose density is linear in $[\bu]_x$.
\end{cor}
\begin{cor}
The symplectic implicit midpoint method (defined by (\ref{RK})-(\ref{kstage}) with $s=1, b_1=1,$ and $a_{1,1}=c_1=1/2$) applied to (\ref{SDcon}) preserves the conservation law 
\[D_m\left\{\widetilde{F}_2\left(\bx,t_n+\tfrac{1}2{\Delta t},\mu_n\bu_n,D_n\bu_n\right)\right\}+D_n\{\widetilde{G}_2(\bx,\bu_n)\}=0.\]
\end{cor}
\subsection{Conservation law for the Hamiltonian}\label{CLHamsect}
We consider here the system of Hamiltonian PDEs (\ref{HamPDE}) defined by a Hamiltonian function $H$ on a domain with periodic boundary conditions. This assumption is introduced only for simplicity: the preservation of conservation laws is local and therefore independent of the specific boundary conditions assigned to the differential problem. The following local conservation law for the energy is satisfied by all solutions of (\ref{HamPDE}):
\begin{align}\nonumber
D_t(H)=&\,\sum_{\alpha,j}\frac{\partial H}{\partial(D_x^ju^\alpha)}(D_tD_x^ju^\alpha)=\sum_{\alpha,j}\frac{\partial H}{\partial(D_x^ju^\alpha)}D_x^j(D_tu^\alpha)\\\label{Hclaw}
=&\,D_x(\psi)+\mathcal{E}(H)^T\mathcal{DE}(H)\equiv D_x(F)
\end{align}
with $$\psi=\sum_{\alpha,i,j} (D_x^iD_tu^\alpha)(-D_x)^j\frac{\partial H}{\partial D_x^{i+j+1}u^\alpha}.$$
Among the best-known energy-conserving discrete gradient methods is the Average Vector Field (AVF) method \cite{McLaren}. We can use this in two different ways, depending on the number of points, $M$, in the stencil in (\ref{nodesx}). If $M$ is odd, let $A=-B$ so that the stencil is centred on $x_0$; we denote a semidiscretization of $H([\bu]_x)$ on such a stencil by $\widetilde{H}(\bU)$. 
The AVF method approximates (\ref{HamPDE}) by
\begin{equation}\label{AVF}
D_n\bu_{0,0}=\widetilde{\mathcal{D}}(\bu_0,\bu_1)\widetilde{\delta}(\mathbf{u}_0,\mathbf{u}_1)\equiv\widetilde{\mathcal{D}}(\bu_0,\bu_1)\int_0^1\mathsf{E}_\bU \left(\widetilde H(\mathbf U)\right)\Big\vert_{\mathbf{U}=(1-\xi)\mathbf{u}_0+\xi\mathbf{u}_1}\mathrm{d}\xi.
\end{equation}
If $M$ is even, let $A=1-B$ so that the stencil is centred at the midpoint of $x_0$ and $x_1$. Denote a semidiscretization of $H([\bu]_x)$ as $\widetilde{H}(\mu_m\bU)$. We define the AVF method on such a stencil to be 
\begin{align}\label{AVF2}
\!\!D_n\mu_m\bu_{0,0}=\widetilde{\mathcal{D}}(\bu_0,\bu_1)\widetilde{\delta}(\mu_m\mathbf{u}_0,\mu_m\mathbf{u}_1):=\widetilde{\mathcal{D}}(\bu_0,\bu_1)\!\!\int_0^1\!\!\mathsf{E}_{\mu_m \bU} \left(\widetilde H(\mu_m\mathbf U)\right)\Big\vert_{\mathbf{U}=(1-\xi)\mathbf{u}_0+\xi\mathbf{u}_1}\!\!\mathrm{d}\xi.
\end{align}
McLachlan and Quispel proved in \cite{QMcL} that discrete gradient methods preserve a discrete version of (\ref{Hclaw}) provided that $\widetilde{\mathcal{D}}$ is a skew-adjoint approximation of $\mathcal{D}$. The following theorem proves that the AVF method (\ref{AVF}) preserves the local conservation law for the energy (\ref{Hclaw}) under a milder assumption.
\begin{theo}\label{theoavf}
The AVF methods (\ref{AVF}) and (\ref{AVF2}) preserve a discrete energy conservation law if there exists a function $f$ defined on the stencil such that
\begin{equation}\label{oddmild}
\widetilde{\delta}(\mathbf{u}_0,\mathbf{u}_1)^T\widetilde{\mathcal{D}}(\mathbf{u}_0,\mathbf{u}_1)\widetilde{\delta}(\mathbf{u}_0,\mathbf{u}_1)=D_m(f),
\end{equation}
or
\begin{equation}\label{evenmild}
\widetilde{\delta}(\mu_m\mathbf{u}_0,\mu_m\mathbf{u}_1)^T\widetilde{\mathcal{D}}(\mathbf{u}_0,\mathbf{u}_1)\widetilde{\delta}(\mu_m\mathbf{u}_0,\mu_m\mathbf{u}_1)=D_m(f),
\end{equation}
respectively.
\end{theo}
\begin{proof}
Equation (\ref{oddmild}) yields a discrete energy conservation law for (\ref{AVF}), namely
\begin{align*}
\!\!D_n\widetilde{H}(\mathbf{u}_0)=&\sum_{i,\alpha}\!(D_n u^\alpha_{i,0})\int_0^1\!\frac{\partial}{\partial U_i^\alpha}\widetilde{H}(\mathbf{U})\Big\vert_{\mathbf{U}=(1-\xi)\mathbf{u}_0+\xi\mathbf{u}_{1}}\!\mathrm{d}\xi=\widetilde{\delta}(\mathbf{u}_0,\mathbf{u}_1)^T(D_n \bu_{0,0})\!\\
&+\!D_m(\widetilde{\psi})=\,\widetilde{\delta}(\mathbf{u}_0,\mathbf{u}_1)^T\widetilde{\mathcal{D}}(\mathbf{u}_0,\mathbf{u}_1)\widetilde{\delta}(\mathbf{u}_0,\mathbf{u}_1)\!+\!D_m(\widetilde\psi)\!=\!D_m(\widetilde{F}),
\end{align*}
where
\begin{equation*}\widetilde{\psi}=\sum_{j\neq 0,\alpha}\frac{j\Delta x}{|j|}\sum_{i=\min\{j,0\}}^{\max\{j,0\}-1}(D_nu_{i,0}^\alpha)S_m^{i-j}\int_0^1\left(\frac{\partial}{\partial U_j^\alpha} \widetilde H(\mathbf U)\right)\Big\vert_{\mathbf{U}=(1-\xi)\mathbf{u_0}+\xi\mathbf{u_1}}\mathrm{d}\xi.
\end{equation*}
Similarly, from (\ref{evenmild}), the conservation law preserved by (\ref{AVF2}) is
\[D_n\widetilde{H}(\mu_m\mathbf{u}_0)=\widetilde{\delta}(\mu_m\mathbf{u}_0,\mu_m\mathbf{u}_1)^T\widetilde{\mathcal{D}}(\mathbf{u}_0,\mathbf{u}_1)\widetilde{\delta}(\mu_m\mathbf{u}_0,\mu_m\mathbf{u}_1)+D_m(\widetilde\phi)=D_m(\widetilde{F}),\]
where
\begin{equation*}\widetilde{\phi}\!=\!\sum_{j\neq 0,\alpha}\!\!\frac{j\Delta x}{|j|}\sum_{i=\min\{j,0\}}^{\max\{j,0\}-1}(D_n\mu_mu_{i,0}^\alpha)S_m^{i-j}\!\!\int_0^1\!\!\left(\!\frac{\partial}{\partial \mu_mU_j^\alpha} \widetilde H(\mu_m\mathbf U)\!\right)\!\!\Big\vert_{\mathbf{U}=(1-\xi)\mathbf{u_0}+\xi\mathbf{u_1}}\!\!\!\!\mathrm{d}\xi.\qed
\end{equation*}
\end{proof}
\begin{rem}
Theorem~\ref{theoavf} holds true in particular when $\widetilde{\mathcal{D}}$ is skew-adjoint.
\end{rem}
\begin{rem}
Condition (\ref{oddmild}) is satisfied if and only if
\begin{equation}\label{testodd}
\mathsf{E}_{\bu_n}\left(\widetilde{\delta}(\mathbf{u}_0,\mathbf{u}_1)^T\widetilde{\mathcal{D}}(\mathbf{u}_0,\mathbf{u}_1)\widetilde{\delta}(\mathbf{u}_0,\mathbf{u}_1)\right)=0,\qquad n=0,1.
\end{equation}
Similarly, condition (\ref{evenmild}) holds true if and only if
\begin{equation}\label{testeven}
\mathsf{E}_{\bu_n}\left(\widetilde{\delta}(\mu_m\bu_0,\mu_m\bu_1)^T\widetilde{\mathcal{D}}(\mathbf{u}_0,\mathbf{u}_1)\widetilde{\delta}(\mu_m\bu_0,\mu_m\bu_1)\right)=0,\qquad n=0,1.
\end{equation} 
Conditions (\ref{testodd}) and (\ref{testeven}) provide simple practical tests for analysing whether the assumptions of Theorem~\ref{theoavf} are satisfied by a given scheme.
\end{rem}
\begin{rem}
The following Hamiltonian-preserving schemes can be obtained from (\ref{AVF}) or (\ref{AVF2}), where the operator $\widetilde{\mathcal D}$ is not skew-adjoint but satisfies (\ref{oddmild}) or (\ref{evenmild}), respectively:
\begin{itemize}
\item[$\bullet$] EC$_8$ and the family of schemes MC$_8$ for KdV in \cite{FCHydon},
\item[$\bullet$] EC$_8$(0) (which preserves a quartic density) and MC$_8$(0) for mKdV in \cite{FCHmkdv},
\item[$\bullet$] EC$_6$ for BBM in \cite{FCHnls}. 
\end{itemize}
\end{rem}

\subsection{Conservation laws of other types and multidimensional domains}\label{genericCL}
Fully-discrete methods that preserve other types of conservation law can be obtained again by using an Euler operator. For simplicity, we restrict the discussion to the case of second-order schemes for a PDE with polynomial nonlinearity. These can be obtained by following the steps below:
\begin{enumerate}
\item Let $P=\prod_{i=1}^r L_i$ be a polynomial of degree $r$ in the semidiscretization, where without loss of generality $L_i$ is a linear approximation of a single monomial factor in the continuous counterpart of $P$. Therefore, $L_i$ can depend on either $\mathbf{U}$ or $\mathbf{U}_t$. Assuming that the stencil has $N$ points in the time dimension, discretize $P$ using 
\begin{equation}\label{fulld}
\sum_{j=0}^{N^r-1} \prod_{i=1}^r \alpha_j L_i(\mathbf{u}_{i_j}), \qquad \alpha_j\in\mathbb{R},
\end{equation}
where $i_j$ is the $i$-th digit of the representation of $N^r-1-j$ in base $N$, ordered from right to left, and setting $i_j=0$ if $N^r-1-j$ has less then $i$ digits. By varying $i_j$ one obtains all possible combinations of $N$ digits of length $r$. At this stage, we assume that the coefficients $\alpha_j$ only satisfy the requirements for consistency.  For example, in a one-step method linear quantities are approximated by
$$L_1(\bU)=L_1(\mu_n\bu_0),\qquad L_1(\bU_t)=L_1(D_n\bu_0),$$
and quadratic quantities by
\begin{align*}L_1L_2\approx&\, \alpha_0L_1(\mathbf{u}_1)L_2(\mathbf{u}_1)+\alpha_1L_1(\mathbf{u}_0)L_2(\mathbf{u}_1)+\alpha_2L_1(\mathbf{u}_1)L_2(\mathbf{u}_0)+\alpha_3L_1(\mathbf{u}_0)L_2(\mathbf{u}_0).
\end{align*}
\item The values of the parameters $\alpha_j$ are obtained by solving
\begin{equation}\label{fulleu}
\mathsf{E}_{\bu}(\widetilde{\mca}\widetilde{\mcq}_\ell)=\mathbf{0}, \qquad \ell=1,2,\ldots,
\end{equation}
where $\widetilde{\mca}$ and $\widetilde{\mcq}_\ell$ are the approximations of the PDE and the characteristic obtained after steps 1 and 2, and $\mathsf{E}_{\bu}$ is the difference Euler operator,
\begin{equation}\label{totdisc}
\mathsf{E}_{\bu}=\sum_{i,j}S_m^{-i}S_n^{-j}\frac{\partial}{\partial u_{i,j}},
\end{equation}
whose kernel consists of all fully-discrete conservation laws \cite{kuper}.
\end{enumerate}
\begin{rem}
In total, net of consistency requirements, for each monomial of degree $r$ one needs:
\begin{itemize}
\item[$\bullet$] $\left(\begin{array}{c} M+r-1\\ r\end{array}\right)$ parameters for the semidiscretization. After solving (\ref{sdeulcond}), only a few of these will still be undetermined.
\item[$\bullet$] $N^r$ new parameters for the full discretization (\ref{fulld}), to be determined by solving (\ref{fulleu}).
\end{itemize}  
By contrast, if one searches directly for all full discretizations without semidscretizing first, the strategy in \cite{FCHydon} introduces $\left(\begin{array}{c} NM+r-1\\ r\end{array}\right)$ variables for each monomial of degree $r$. This in general yields a huge nonlinear system whose solution is impractical.
\end{rem}

The algorithm above can be iterated, to preserve conservation laws for PDEs with more than two independent variables, by discretizing a single variable at each iteration. At the $k$-th iteration the Euler operator in (\ref{totdisc}) is replaced by
$$\mathsf{E}_{\bU}^k=\sum S_{m_1}^{-i_1}S_{m_2}^{-i_2}\ldots S_{m_k}^{-i_k}(-D_{x_{k+1}})^{j_1}(-D_{x_{k+2}})^{j_2}\ldots(-D_{x_d})^{j_k},$$
whose kernel consists in the space of conservation laws of the form
$$D_{m_1}\widetilde{F}_1+D_{m_2}\widetilde{F}_2+\ldots+D_{m_k}\widetilde{F}_k+D_{x_{k+1}}F_{k+1}+\ldots+D_{x_{d}}F_{d}.$$ 
The proof is similar to that of Theorem~\ref{th1}.
\section{The Boussinesq equation}\label{GBsec}
We consider here the (Good) Boussinesq equation
\begin{equation}\label{Bousseq}
{u_{tt}-u_{xx} -(u^2)_{xx}+u_{xxxx}}=0,\quad (x,t)\in [a,b]\times[0,\infty),
\end{equation}
cast as a system of two PDEs
\begin{equation}\label{Bouss}
\mathcal{A}=(u_t-v_x, v_t-u_x-(u^2)_{x}+u_{xxx})=\mathbf{0}.
\end{equation}
This system can be written in Hamiltonian form,
$$D_t\,\bu\equiv D_t\left(\begin{array}{c}u\\v\end{array}\right)=\mathcal{D}(\mathcal{E}(H)),$$
with
\begin{equation*}
\mathcal{D}=\left[\begin{array}{cc}
0&D_x\\D_x&0
\end{array}\right],\qquad H=\tfrac{1}3u^3+\tfrac{1}2(v^2+u^2+u_{x}^2).
\end{equation*} 
System (\ref{Bouss}) has infinitely many independent conservation laws \cite{PC}. The first four are $D_xF_i+D_tG_i=0$, with 
\begin{align}\label{CL1}
F_1=& -v,\qquad G_1=u,\\\label{CL2}
F_2=& u_{xx}-u-u^2,\qquad G_2=v,\\\label{CL3}
F_3=& uu_{xx}-\tfrac{1}2(v^2+u^2+u_x^2)-\tfrac{2}3u^3,\qquad G_3=uv,\\\label{CL4}
F_4=& vu_{xx}-u_xu_t-uv-u^2v,\qquad G_4=H.
\end{align}
with characteristics
\begin{equation}\label{dischar}
\mathcal{Q}_1=(1,0)^T, \quad\mathcal{Q}_2=(0,1)^T, \quad \mathcal{Q}_3=(v,u)^T,\quad \mathcal{Q}_4=(u+u^2-u_{xx},v)^T,
\end{equation}
respectively. When the boundary conditions are conservative (e.g. periodic), integrating in space (\ref{CL1})--(\ref{CL4}) gives the preservation of the following invariants,
\begin{equation}\label{invar}
I_1=\int u\,\mathrm{d}x,\qquad I_2=\int v\,\mathrm{d}x, \qquad I_3=\int uv\,\mathrm{d}x, \qquad I_4=\int H\,\mathrm{d}x;
\end{equation}
here $I_3$ and $I_4$ are the global momentum and the global energy, respectively.
\subsection{Conservative methods for the Boussinesq equation}
We look for semidiscretizations of (\ref{Bouss}) of the form 
\begin{equation}\label{sdgen}
\widetilde{\mathcal{A}}:=(D_m \widetilde{F}_1+D_t \widetilde{G}_1, D_m \widetilde{F}_2+D_t \widetilde{G}_2)=\mathbf{0}.
\end{equation}
Solutions of (\ref{sdgen}) satisfy semidiscrete versions of the conservation laws (\ref{CL1}) and (\ref{CL2}) with $\widetilde{\mcq}_1=\mcq_1$ and $\widetilde{\mcq}_2=\mcq_2$.
The linear and quadratic terms in (\ref{Bouss}) and (\ref{dischar}) are approximated as
\begin{align*}
\sum_{i=A}^B \alpha_i Z_i, \qquad \sum_{i=A}^B\sum_{j=i}^B \beta_{i,j}Z_iZ_j, 
\end{align*}
respectively, where $Z_i\in \{U_i, V_i, D_t U_i, D_t V_i\}$, and the coefficients $\alpha_i$ and $\beta_{i,j}$ are chosen by requiring the desired order of accuracy and the preservation of the conservation law of either the momentum (\ref{CL3}) or the energy (\ref{CL4}), according to the strategy outlined in Section~\ref{spacedis}. We have not found any semidiscrete scheme that preserves both of these conservation laws, as the constraints on the approximations of the nonlinear terms that we have obtained from (\ref{sdeulcond}) are not compatible with each other.

In the following, we present the components of \eqref{sdgen} and the characteristic $\widetilde{\mcq}_{3/4}$ and density $\widetilde{G}_{3/4}$ of the remaining preserved conservation law. The corresponding flux $\widetilde{F}_{3/4}$ does not contribute to our global error estimates; in most cases, it has many terms and gives little insight, so we omit this.

Fully-discrete schemes that preserve the local momentum or the local energy are then obtained by applying Gauss-Legendre method or AVF method, respectively. As these are Runge--Kutta methods, the conservation laws (\ref{CL1}) and (\ref{CL2}) are preserved as a consequence of Corollary~\ref{cor1}.

\subsubsection*{Second-order schemes}\label{IIord}
Here we introduce new families of second-order schemes that preserve three conservation laws. The stencil in space consists of four points and we set $A=-1$ and $B=2$ in (\ref{nodesx}). All the free parameters in the formulae below $(\alpha,\beta,\gamma,\xi)$ are $O(\Delta x^2).$ Free parameters corresponding to higher degree perturbations are set equal to zero as their contribution is negligible.\\[1\baselineskip]
\textit{Momentum-conserving schemes}\\[1\baselineskip]
We have obtained six families of semidiscretizations that preserve the conservation law (\ref{CL3}), split by two different forms of the characteristic and the three parameter values $s\in\{0,1/3,1/2\}$.

The first three families are
\begin{align}\label{MC2}
\widetilde{F}_1=&-(V_0+\alpha D_m^2V_{-1}),\qquad \widetilde{G}_1=\mu_m(U_0+s{\Delta x^2}D_m^2U_{-1}), \\\nonumber
\widetilde{F}_2=&\,(1+\beta)D_m^2 U_{-1}-U_0-(U_0+\tfrac{(1-s)\Delta x^2}{3-5s}D_m^2U_{-1})(U_0+\tfrac{s\Delta x^2}{5s-1}D_m^2U_{-1})\\\nonumber
&+\{\tfrac{s\Delta x^2}{s+1}+\xi(3s-1)(1-2s)\}\{D_m^2(U_{-1}^2)-(D_mU_{-1})(D_mU_{0})\},\\\nonumber
\widetilde{G}_2\!=&\mu_m\!\left(V_0\!+\!\gamma D_m^2V_{-1}\right)\!,\quad\!
\widetilde{G}_3\!=\!\widetilde G_1\widetilde G_2,\quad \! \widetilde{\mcq}_3\!=\!(\widetilde G_2,\widetilde G_1)^T\!.
\end{align}

Three remaining are given by $\widetilde{F}_1, \widetilde{F}_2$ and $\widetilde{G}_2$ as in (\ref{MC2}) together with
\begin{align*}
\widetilde{G}_1=&\,\mu_m\left(U_0+\gamma D_m^2U_{-1}\right),\\ \widetilde{\mcq}_3=&\,\left(\mu_m(V_0+s{\Delta x^2}D_m^2V_{-1}),\,\mu_m(U_0+s{\Delta x^2}D_m^2U_{-1})\right)^T\!\!,\\
\widetilde{G}_3=&\,(\mu_mU_0)(\mu_mV_0)\!-\!(\gamma+s{\Delta x^2})\mu_m\{(\Dc_mU_{0})(\Dc_mV_{0})\}+{s\gamma\Delta x^2}(D_m\Dc_mU_{0})(D_m\Dc_mV_{0}).
\end{align*}

In the numerical tests section we limit our investigations to the semidiscretions obtained from (\ref{sdgen}) with (\ref{MC2}) and $s=\alpha=\beta=\xi=0$, $\gamma=\lambda_1\Delta x^2$; we use MC$_2(\lambda_1)$ to denote the family of finite difference schemes obtained by using the symplectic implicit midpoint method (Gauss-Legendre method of order two) to discretize in time. The iterative technique described in Section~\ref{genericCL} also finds these schemes, but no others.
\bigskip

\noindent\textit{Energy conserving schemes}
\medskip

There is only one family of semidiscretizations of the form (\ref{sdgen}) that preserves the local conservation law of the energy. For this family, $\widetilde{F}_1$ and $\widetilde{G}_2$ are defined as in (\ref{MC2}) and
\begin{align}\nonumber
\widetilde{G}_1=&\,\mu_mU_0,\,\,\, \widetilde{F}_2=(1\!+\!\beta)D_m^2 U_{-1}\!-\!U_0\!-\!(\mu_m^2U_{-1})^2,\,\,\, \widetilde{\mcq}_4=(-\mu_m \widetilde F_2,-\mu_m \widetilde F_1)^T\\
\widetilde{G}_4=&\,\tfrac{1}2\left\{(\mu_m U_0)^2+(1+\beta)\mu_m\big((\Dc_mU_{0})^2\big)\right\}\\\nonumber
&+\!\tfrac{1}2\left\{\mu_m(V_{0}+\alpha D_m^2V_{-1})\mu_m(V_{0}+\gamma D_m^2V_{-1})\right\}\!+\!\tfrac{1}3(\mu_mU_0)\mu_m\big((\mu_m^2 U_{-1})^2\big).
\end{align}
The resulting system of ODEs can be written in the form 
\begin{equation}\label{Hamode}
D_t\left(\begin{array}{c} \mu_m U_0\\ \mu_m V_0\end{array}\right)=
\widetilde{\mathcal D}\left(\begin{array}{c} \mathsf{E}_{\mu_m \bU} (\widetilde{H})\\ \mathsf{E}_{\mu_m \mathbf V} (\widetilde{H})\end{array}\right),
\end{equation}
with 
\[ \widetilde{H}=\widetilde{G}_4,\qquad \widetilde{\mathcal D}=\left(\begin{array}{cc} 0 & \widetilde{D}_x\\ \widetilde{D}_x & 0\end{array}\right),\qquad \widetilde D_x=D_m(\mu_m+\gamma D_m\Dc_mS_m^{-1})^{-1}.\]
The operator $\widetilde{\mathcal D}$ is not skew-adjoint, but satisfies (\ref{evenmild}). Therefore, by applying the AVF method (\ref{AVF2}) to (\ref{Hamode}) we obtain a family of fully-discrete schemes that preserve the local conservation law of the energy. We use EC$_2(\lambda_2)$ to denote the schemes with $\alpha=\gamma=0$ and $\beta =\lambda_2\Delta x^2$. Again, the iterative technique from Section~\ref{genericCL} yields only these schemes.

\subsubsection*{Fourth-order schemes}\label{IVord}
The families of fourth-order schemes introduced here preserve three conservation laws and depend on free parameters ($\alpha,\beta,\gamma,\xi$) that are all $O(\Delta x^4)$. Parameters introducing only perturbations of higher degree are set equal to~zero.

For each of the semidiscretizations introduced in this section, the discrete fluxes $\widetilde{F}_j$ are second-order accurate, but $D_m \widetilde F_j$ and the three preserved conservation laws,
$$\widetilde\mca\widetilde\mcq_j=D_m \widetilde F_j + D_t \widetilde G_j,$$
are approximated with fourth-order accuracy.
\bigskip

\noindent\textit{Momentum-conserving schemes}
\medskip

On a spatial stencil with six points ($A=-2$ and $B=3$ in (\ref{nodesx})) there are two families of semidiscretizations that preserve the local momentum conservation law. Let
\begin{align}\nonumber
\Phi(Z;p):=&\,Z_{-1}+p\Delta x^2D_m^2Z_{-2},\\\nonumber
\!\Theta(Z;p_1,p_2,p_3,p_4,p_5,p_6):=&D_m^2\Phi(Z;p_1)(D_m^2\Phi(Z;p_2))\!+\!p_3(S_m\Phi(Z;p_4))D_m^4Z_{-2}\\\nonumber
&\,+p_5D_m\mu_m(\Phi(Z;p_6))D_m^3\mu_mZ_{-2}.
\end{align}

The first family of semidiscretizations and their preserved conservation laws is given by
\begin{align}\label{MC4}
\widetilde{F}_1=&-\!(V_0\!-\!\tfrac{\Delta x^2}{24}D_m^2V_{-1}\!+\!\alpha D_m^4V_{-2}),\,\,\, \widetilde{G}_1=\!\mu_m(U_0\!-\!\tfrac{\Delta x^2}8\!D_m^2U_{-1}),\\\nonumber
\widetilde{F}_2=&\,D_m^2U_{-1}-U_0-U_0^2+\tfrac{\Delta x^2}{24}D_m^2(U_{-1}+U_{-1}^2-3D_m^2U_{-2})+\beta D_m^4U_{-2}\\\nonumber
&-\tfrac{\gamma}2\Theta(U;-\tfrac{1}{4},-\tfrac{3}{8},-2,-\tfrac{3}{16},-2,-\tfrac{1}{16})+\tfrac{7\Delta x^4}{192}\Theta(U;0,0,\tfrac{8}7,-\tfrac{3}{8},0,0),\\\nonumber
\widetilde{G}_2=&\,\mu_m(V_0-\tfrac{\Delta x^2}8D_m^2 V_{-1}+\xi D_m^4V_{-2}),\quad \widetilde{\mcq}_3=(\widetilde{G}_2,\widetilde{G}_1)^T,\quad \widetilde{G}_3=\widetilde G_1\widetilde G_2.
\end{align}

The second family has $\widetilde{F}_1,\widetilde{F}_2$ and $\widetilde{G}_2$ as in (\ref{MC4}), together with
\begin{align*}
\widetilde{G}_1=&\,\mu_m(U_0\!-\!\tfrac{\Delta x^2}8D_m^2U_{-1}\!+\!\xi D_m^4U_{-2}),\\
\widetilde{\mcq}_3=&\,\big(\mu_m(V_0\!-\!\tfrac{\Delta x^2}8D_m^2V_{-1}),\mu_m(U_0\!-\!\tfrac{\Delta x^2}8D_m^2U_{-1})\big)^T\!,\\
\widetilde{G}_3=&\,\{\mu_m(U_0\!-\!\tfrac{\Delta x^2}8D_m^2U_{-1})\}\{\mu_m(V_0\!-\!\tfrac{\Delta x^2}8D_m^2V_{-1})\}\\
&+\xi(D_m\Dc_mU_{0})(D_m\Dc_mV_{0})+\tfrac{\xi\Delta x^2}8\mu_m\{(D_m^2\Dc_mU_{-1})(D_m^2\Dc_mV_{-1})\}.
\end{align*}

We use MC$_4(\lambda_3)$ to denote the schemes obtained by applying the Gauss-Legendre method of order four to (\ref{MC4}), with $\alpha=\beta=\gamma=0$ and $\xi=\lambda_3\Delta x^4$.

\bigskip
\noindent\textit{Energy conserving schemes}
\medskip

On the most compact six-point stencil, there are no semidiscretizations of the form (\ref{sdgen}) that preserve the local conservation law for energy. However, a seven-point stencil ($B=-A=3$ in (\ref{nodesx})) is more fruitful. For $n\in\mathbb{N}$, let
$$ \varphi_n(k)=\begin{cases}\lceil \tfrac{n}2\rceil, & \text {if}\quad k\geq \tfrac{n}2,\\
k+1, & \text {if}\quad k< \tfrac{n}2, \end{cases}$$
and define the operators
$$\nu_m^\pm=I\pm\tfrac{\Delta x^2}6D_m^2S_m^{-1}$$
and the functions
\begin{align*}
\widehat{F}_1=&\, -(\nu_m^- V_{-1}+\alpha D_m^4V_{-3}),\\
\widehat{F}_2=&\, D_m^2 U_{-2}-\nu_m^-U_{-1}-(\nu_m^+ U_{-1})^2+(\beta-\tfrac{\Delta x^2}4)D_m^4U_{-3}\\
&+\tfrac{\Delta x^2}6\left\{2(\nu_m^+ U_{-1})(D_m^2\nu_m^+ U_{-2})+D_m^2\big((\nu_m^+ U_{-2})^2\big)\right\}.
\end{align*}
The family of semidiscretizations and conservation laws is as follows:
\begin{align}\label{EC4}
\widetilde{F}_1=&\, \mu_m\widehat{F}_1,\qquad \widetilde{G}_1=U_0,\\\nonumber
\widetilde{F}_2=&\, \mu_m\widehat{F}_2,\qquad \widetilde{G}_2=\,V_0+\gamma D_m^4V_{-2},\\\nonumber
\widetilde \mcq_4=&\,(-\nu_m^+S_m\widehat{F}_2,-\nu_m^+S_m\widehat{F}_1)^T,\\\nonumber
\widetilde G_4=&\,\tfrac{1}2\{(V_0+\gamma D_m^4V_{-2})\nu_m^+(\nu_m^-V_0+\alpha D_m^4V_{-2})\\\nonumber
&+\mu_m\{(D_mU_{-1})D_m\nu_m^+(U_{-1}-\tfrac{\Delta x^2}4D_m^2U_{-2})\}+U_0\nu_m^+(\nu_m^-U_0-\beta D_m^4U_{-2})\}\\\nonumber
&+\tfrac{1}3U_0\nu_m^+\left((\nu_m^+U_0)^2-\tfrac{\Delta x^2}6\left(2(\nu_m^+U_0)(D_m^2\nu_m^+U_{-1})+D_m^2((\nu_m^+U_{-1})^2)\right)\right).
\end{align}
The systems of ODEs defined by (\ref{sdgen}) with (\ref{EC4}) amounts to
\begin{equation}\label{Hamode4}
D_t\left(\begin{array}{c} U_0\\ V_0\end{array}\right)=
\widetilde{\mathcal D}\left(\begin{array}{c} \mathsf{E}_{\bU} (\widetilde{H})\\ \mathsf{E}_{ \mathbf V} (\widetilde{H})\end{array}\right),
\end{equation}
with 
\[ \widetilde{H}=\widetilde{G}_4,\qquad \widetilde{\mathcal D}=\left(\begin{array}{cc} 0 & \widetilde{D}_x\\ \widetilde{D}_x & 0\end{array}\right),\qquad \widetilde D_x=D_m\mu_m(S_m\nu_m^++\gamma D_m^4S_m^{-1}\nu_m^+)^{-1}.\]
The operator $\widetilde{\mathcal D}$, although not skew-adjoint, satisfies (\ref{oddmild}). We use EC$_4(\lambda_4)$ to denote the family of schemes obtained by applying the AVF method of order four (see \cite{McLaren}) to (\ref{Hamode4}) with $\alpha =\lambda_4\Delta x^4$ and $\beta=\gamma=0$.
\section{Numerical Tests}\label{Numsec}
In this section we solve a couple of benchmark problems to show the effectiveness and conservation properties of the numerical methods in Section~\ref{GBsec}. 

The results are compared with the following second-order structure-preserving methods:
\begin{itemize}
\item[$\bullet$] The multisymplectic scheme for (\ref{Bousseq}),
\begin{align*}
\text{PS}\equiv &\,D_n^2\mu_m^4u_{-2,-1}-D_m^2\mu_m^2\mu_n^2u_{-2,-1}-D_m^2\mu_m\mu_n(\mu_m\mu_n u_{-2,-1})^2+D_m^4\mu_n^2u_{-2,-1}=0,
\end{align*}
developed in \cite{Zeng} and equivalent to the well-known Preissmann scheme.
\item[$\bullet$] The symplectic scheme for (\ref{Bouss}) in \cite{Chen},
\begin{align*}
\text{MP}:= \big(&\,D_n u_{0,0}-\Dc_m\mu_n v_{0,0},D_n v_{0,0}+\Dc_m(D_m^2\mu_n u_{-1,0}- (\mu_n u_{0,0})^2-\mu_n u_{0,0}   \,\big)=\mathbf{0},
\end{align*}
obtained by applying the midpoint rule to a suitable spatial discretization.
\item[$\bullet$] The energy-conserving scheme for (\ref{Bouss}) in \cite{Matsuo}, 
\begin{align*}
\text{DVD}:=  \big(&\,D_n u_{0,0}-D_m\mu_n v_{0,0},\, D_n v_{0,0}\\&\!+\!D_m(D_m^2\mu_n u_{-2,0}\!-\!\tfrac{1}3 (u_{-1,0}^2\!+\!u_{-1,0}u_{-1,1}\!+\!u_{-1,1}^2) \!-\!\mu_n u_{-1,0} )\big)\!=\!\mathbf{0},
\end{align*}
obtained using a discrete variational derivative method. This scheme can be obtained also by applying the AVF method to the Hamiltonian system of ODEs defined by 
$$ \widetilde{\mathcal D}=\left(\begin{array}{cc} 0 & D_mS_m^{-1}\\ D_m & 0\end{array}\right), \qquad H=\tfrac{1}2(U_{0}^2+V_{0}^2+\mu_m((D_mU_{-1})^2)+\tfrac{1}3U_{0}^3.$$
\end{itemize}
To the best of our knowledge, there are no schemes in the literature
for the Boussinesq equation that are fourth-order accurate in both space and time.
Therefore, we compare the performance of the fourth-order schemes in Section~\ref{GBsec} with the following finite difference scheme for (\ref{Bousseq}) introduced in \cite{Ismail}:
\begin{align*}
\text{FD}_4:=\!(1\!+\!\tfrac{\Delta x^2}{12}D_m^2S_m^{-1})^2D_n^2 u_{0,0}\!-\!D_m^2\mu_n \left\{ (1\!+\!\tfrac{\Delta x^2}{12}D_m^2S_m^{-1})(\mu_nu_{-1,0}\!+\!(\mu_nu_{-1,0})^2)\!-\!D_m^2\mu_nu_{-2,0}\right\}\!.
\end{align*}
The scheme FD$_4$ is fourth-order accurate in space and second-order accurate in time, so to have a fair comparison we will use this scheme with a time step equal to $\Delta t^2$.

We consider $(x,t)\in \Omega\equiv [a,b]\times [0,T]$ and periodic boundary conditions. We introduce on $\Omega$ a grid with $I+1$ nodes, $x_i$, in space and $J+1$ nodes, $t_j$, in time. Henceforth subscripts denote shifts from the point $(x_0,t_0)=(a,0)$ (e.g., $u_{i,j}\simeq u(a+i\Delta x,j\Delta t)$).

As the computational time is similar for all the schemes of the same order of accuracy, our comparisons are based on the error in the solution at the final time $t = T$, evaluated as
\begin{equation}\label{relerr}
\left.\frac{\|u-u_\mathrm{exact}\|}{\|u_\mathrm{exact}\|}\right|_{t=T},
\end{equation}
where $\|\cdot\|$ denotes the Euclidean norm.
We also compare the errors in the global invariants (\ref{invar}) defined as
$$\text{Err}_\alpha\!=\!\Delta x\! \max_{j=1,\ldots,J}\left|\sum_{i=0}^I\left(\widetilde{G}_\alpha\Big\vert_{U_m=u_{m+i,j},V_m=v_{m+i,j}}\!-\widetilde{G}_\alpha\Big\vert_{U_m=u_{m+i,0},V_m=v_{m+i,0}}\right)\right|,$$
where $\alpha=1,2,3,4.$ For the methods introduced in this paper, $\widetilde{G}_\alpha$ is given in Section~\ref{GBsec}. For all the other schemes, we set
$$\widetilde{G}_1\!=\!U_0,\quad\! \widetilde{G}_2\!=\!V_0,\quad\! \widetilde{G}_3\!=\! U_{0}V_{0},\quad\! \widetilde{G}_4\!=\!\tfrac{1}2(U_{0}^2\!+\!V_{0}^2\!+\!\mu_m((D_mU_{-1})^2)\!+\!\tfrac{1}3U_{0}^3\!.$$
\subsubsection*{Single soliton}
For the first problem we set $\Omega=[-60,60]\times[0,25]$ and the initial conditions given by the single soliton solution over $\mathbb{R}$,
\begin{equation*}
u_{\mathrm{exact}}(x,t)\!=\!-\tfrac{3p^2}2\sech^2\!\left(\tfrac{p}2(x\!-\!ct\!+\!d)\right),\quad\! v_{\mathrm{exact}}(x,t)\!=\!\tfrac{3cp^2}2\sech^2\!\left(\tfrac{p}2(x\!-\!ct\!+\!d)\right)\!,
\end{equation*}
where $c=\sqrt{1-p^2}.$ We choose $$p=\frac{1}{\sqrt{3}}, \qquad d=10.$$ \begin{table}[htb]
\caption{Single soliton problem: $\pi(\Delta x)$ for each scheme}\label{tabord}
\small
\centerline{\begin{tabular}{|c|c|c|c|c|c|c|c|c|}
\hline
$\Delta x=\Delta t$ &  0.2 & 0.3  & 0.4   & 0.5 & 0.6 & 0.7 & 0.8 & 0.9   \\
\hline
MC$_2$(0) & 1.99 & 2.00 & 1.97 & 1.92 & 1.96 & 1.94 & 1.70 & 2.00 \\
EC$_2$(0) & 1.99 & 1.97 & 1.94 & 1.94 & 1.91 & 1.88 & 1.66 & 1.88 \\
MC$_4$(0) & 4.00 & 3.99 & 3.98 & 3.96 & 3.95 & 3.99 & 3.80 & 3.94 \\
EC$_4$(0) & 4.01 & 4.00 & 4.05 & 4.13 & 4.06 & 4.21 & 4.32 & 4.26 \\
\hline
\end{tabular}}
\end{table}
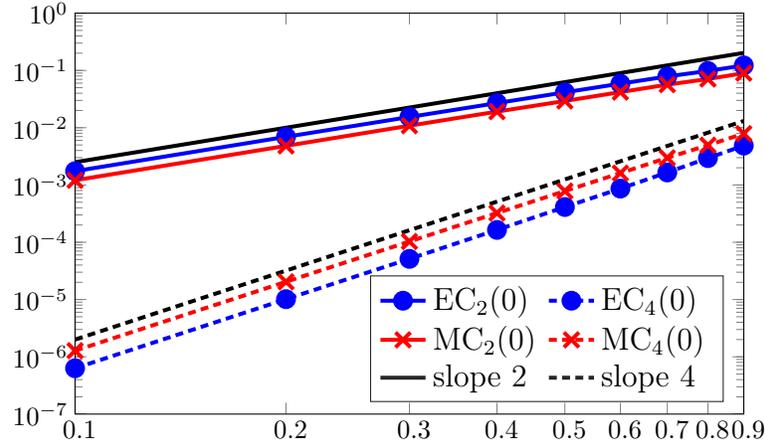
\begin{figure}[htbp]
\begin{center}
	\input{order1sol.tex}
\end{center}
\caption{Single soliton problem: plot of $\pi(\Delta x)$ showing convergence of the schemes.}
	\label{ord1}
\end{figure}
We first examine the convergence and stability of the schemes found in Section~\ref{GBsec}, setting all parameters to zero. The order of convergence at various step sizes is measured by 
$$\pi(\Delta x)=\frac{(k-1)\log(\mathrm{error}_{k}/\mathrm{error}_{k-1})}{k}, \qquad \text{where}\ k=10\Delta x,$$
and error$_k$ denotes the error obtained from (\ref{relerr}) with $\Delta x=\Delta t=k/10$, for $k=1,2,\ldots,9$.
The results in Table~\ref{tabord} and Figure~\ref{ord1} show that all the methods tend to the exact solution with the expected maximum order of convergence as the grid is refined, and are stable also for the largest stepsizes.
\begin{table}
\caption{Single soliton problem with $\Delta x=0.5$ and $\Delta t=0.5$ (except $\text{FD}_4$).}
\label{tab1}       
\centerline{\begin{tabular}{cccccc}
\hline\noalign{\smallskip}
Method &  $\text{Err}_1$ & $\text{Err}_2$ & $\text{Err}_3$  & $\text{Err}_4$  & Sol. Err.  \\
\noalign{\smallskip}\hline\noalign{\smallskip}
$\mbox{MC}_2(0)$	& 6.22e-15 & 6.22e-15 & 1.22e-15 & 7.90e-04 & 0.0293\\
$\mbox{MC}_2(-0.21)$	& 5.33e-15  & 1.33e-15 & 1.22e-15 & 3.77e-04 & 0.0059\\
$\mbox{EC}_2(0)$	& 4.44e-15   & 4.00e-15 & 1.08e-05 & 6.66e-16 & 0.0415 \\
$\mbox{EC}_2(-0.20)$	& 3.55e-15   & 3.55e-15 & 2.63e-06 & 1.22e-15 & 0.0062 \\
PS	& 5.88e-16 & 0.0194 & 5.79e-04 & 0.0013 & 0.0238\\
MP	&  5.77e-15  & 4.89e-15 & 2.77e-04 & 0.0030 & 0.0706 \\
DVD	& 5.77e-15  & 4.00e-15 & 0.0053 &  3.44e-15 & 0.0740\\
$\mbox{MC}_4(0)$	& 5.77e-15  & 5.77e-15 & 2.00e-15 &  3.18e-05 & 7.84e-04\\
$\mbox{MC}_4(0.06)$	& 5.77e-15  & 9.33e-15 & 2.89e-15 &  8.50e-06 & 1.23e-04 \\
$\mbox{EC}_4(0)$	& 6.66e-15  & 1.36e-11 & 4.02e-08 & 5.03e-14 & 4.12e-04\\
$\mbox{EC}_4(0.03)$	& 6.21e-15  & 1.57e-12 & 1.47e-08 &  5.57e-14 & 1.94e-04\\
$\mbox{FD}_4$\,\,($\Delta t=0.5^2$)	& 9.97e-12  & 0.0036 & 1.07e-04 &  2.84e-04 & 0.0038\\
\noalign{\smallskip}\hline
\end{tabular}}
\end{table}

Table~\ref{tab1} shows the error in the conservation laws and the solution for the different methods with $\Delta x=\Delta t=0.5$. For this problem, the values of the free parameters that minimize the error in the solution of MC$_2(\lambda_1)$, EC$_2(\lambda_2)$, MC$_4(\lambda_3)$,  EC$_4(\lambda_4)$ are $\lambda_1=-0.21,$ $\lambda_2=-0.20$, $\lambda_3=0.06$, $\lambda_4=0.03$. Such optimization is easy given that the solution is known, but is not currently feasible more generally. Therefore, the results obtained by setting the above parameters to zero are shown for comparison. This benchmark test illustrates that:
\begin{itemize}
\item[$\bullet$] All schemes preserve the first conservation law, but only those based on the formulation (\ref{Bouss}) preserve the second conservation law;
\item[$\bullet$] The schemes introduced in this paper preserve three conservation laws up to machine accuracy;
\item[$\bullet$] The new second-order schemes compare favourably with the methods from the literature in terms of accuracy in both the solution and the invariants that are not exactly conserved, even without optimising the parameters;
\item[$\bullet$] Choosing the optimal value of the free parameter, the error in the solution is roughly four times smaller (or more) than any other second-order method;
\item[$\bullet$] The fourth-order methods are all more much accurate than FD$_4$ with timestep $\Delta t=0.5^2$, even for non-optimal parameters.
\end{itemize}
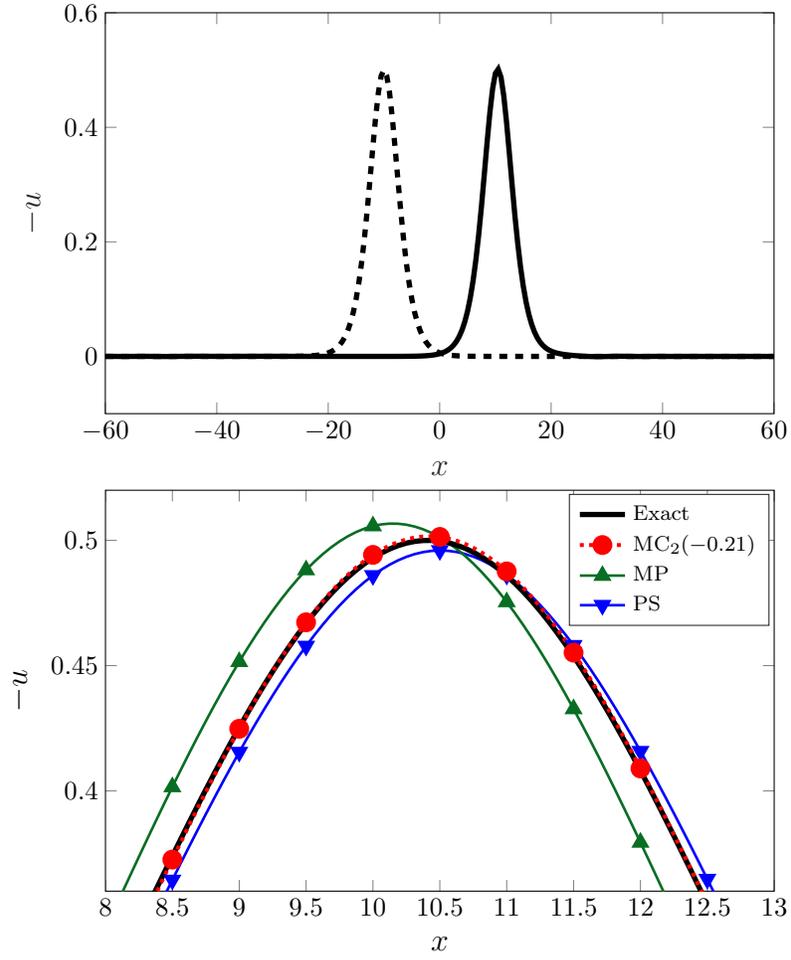
\begin{figure}
\begin{center}
	\input{1sol.tex}
\end{center}
\caption{Top: initial condition (dashed line) and solution of MC$_2(-0.21)$ at $T=25$ (solid line). Bottom: comparison of different schemes around the top of the soliton.}
	\label{1solfig}
\end{figure}
\begin{figure}[htbp]
\begin{center}
	\input{1sol2.tex}
\end{center}
\caption{Single soliton: error of fourth-order schemes.}
	\label{1solerr}
\end{figure}
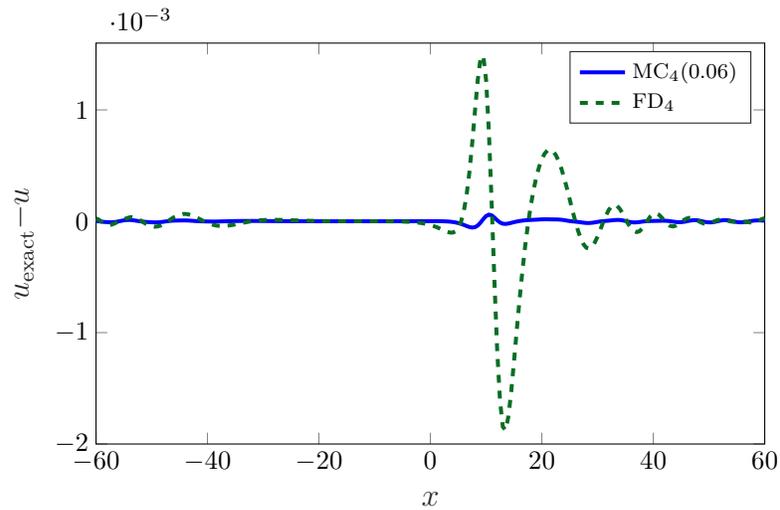

In Figure~\ref{1solfig}, the upper plot shows the solution obtained by the most accurate scheme, MC$_2(-0.21)$. The motion of the soliton does not produce any spurious oscillations. These can be seen (with amplitude of about $10^{-2}$) in the solutions of MP and DVD.

The lower plot shows the exact solution and the solution of MC$_2(-0.21)$ compared to the solutions of MP and PS around the top of the soliton (we omit the solution of DVD as it is the least accurate). The approximate solutions have been reconstructed using cubic spline interpolation of the values at the grid points denoted with markers. This figure shows how well the solution of MC$_2(-0.21)$ matches both the phase and the amplitude of the soliton. 

Figure~\ref{1solerr} shows the difference between the exact solution and the solutions of MC$_4$(0.06) and FD$_4$ (with $\Delta t=0.5^2$). The error in MC$_4$ is roughly 30 times smaller; is located mainly at the peak of the soliton, and can be ascribed to a small phase error. The error of FD$_4$ is more widespread.
\subsubsection*{Interaction of two solitons}
We now study the interaction of two solitons over $\Omega=[-150,150]\times[0,50]$. The exact solution over $\mathbb{R}$ is \cite{MMM},
\begin{equation}\label{ex2sol}
u_{\text{exact}}=-6D_x^2\log \omega(x,t),\quad \omega(x,t)=1+\exp(\eta_1)+\exp(\eta_2)+A\exp(\eta_1+\eta_2),
\end{equation}
where
$$\eta_j=p_j(x-c_jt+d_j),\quad c_j=(-1)^j\sqrt{1-p_j^2},\quad A=\frac{(c_1-c_2)^2-3(p_1-p_2)^2}{(c_1-c_2)^2-3(p_1+p_2)^2}.$$
\begin{figure}[htbp]
\begin{center}
  \includegraphics[width=4.2in,height=5.7cm]{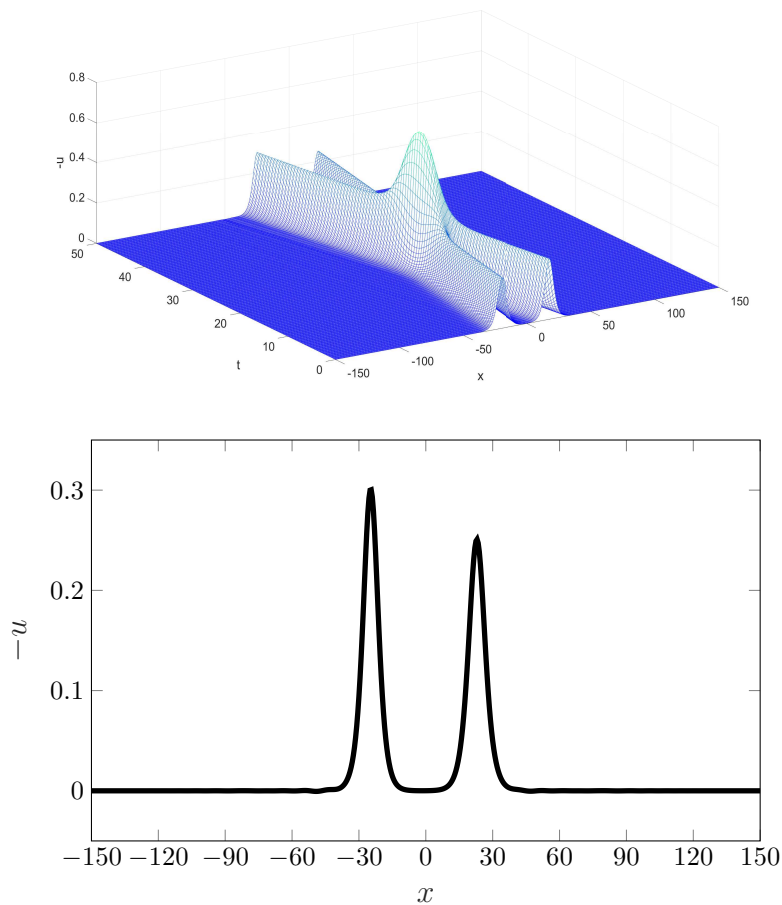}\vspace{0.4cm}
  \input{2sol.tex}
\end{center}
\caption{Solution of EC$_2(-0.18)$ on $\Omega$ and at the time $T=50$. }\label{2sol}
\end{figure}
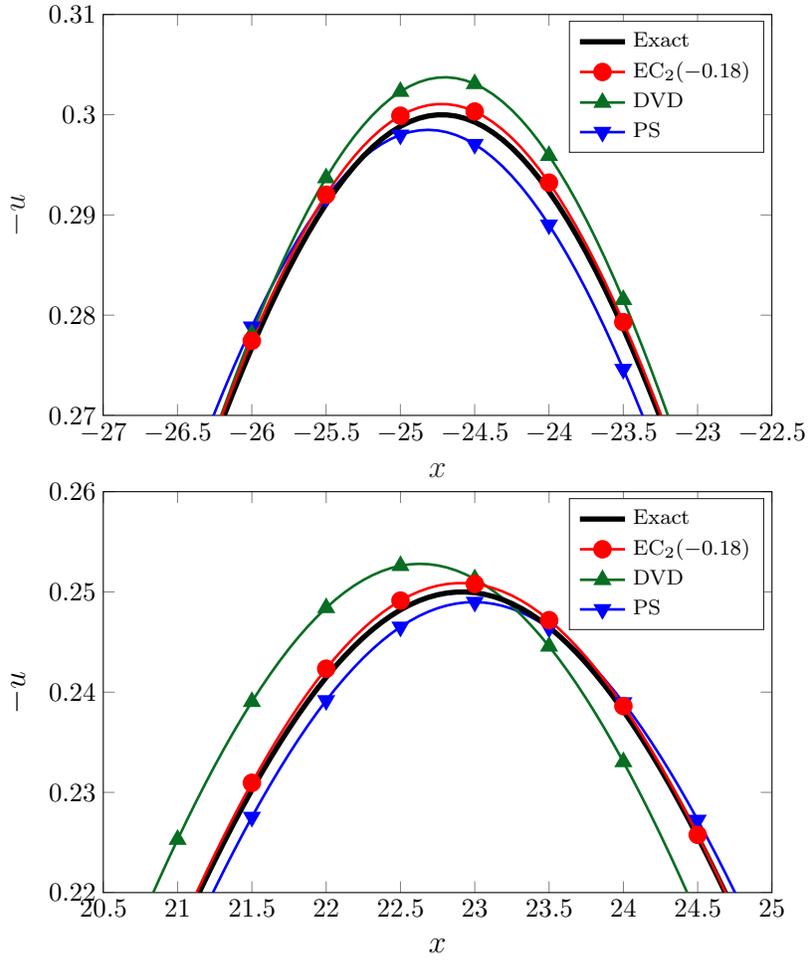
\begin{figure}[htbp]
\begin{center}
	\input{2sol2.tex}
\end{center}
\caption{Comparison of different schemes around the top of the two solitons.}
	\label{2solz}
\end{figure}
\begin{figure}[htbp]
\begin{center}
	\input{2solerr.tex}
\end{center}
\caption{Interaction of two solitons: error of fourth-order methods}
	\label{2solerr}
\end{figure}
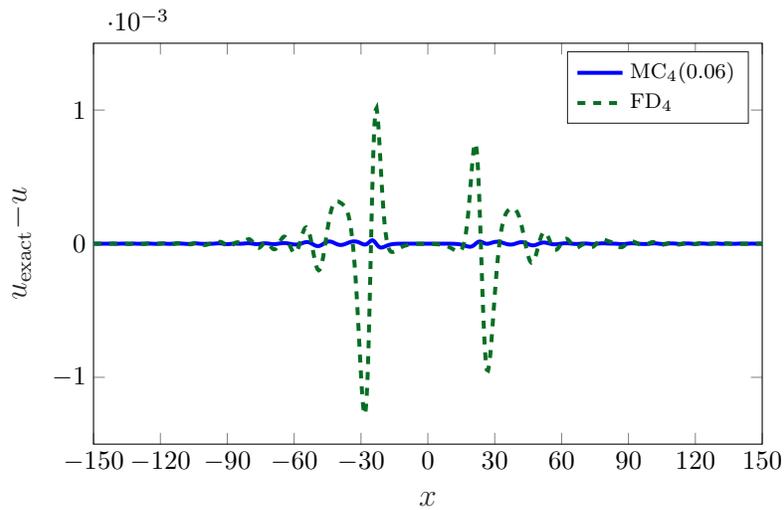
\begin{table}[tb]
\caption{Interaction of two solitons with $\Delta x=0.5$ and $\Delta t=0.5$ (except $\text{FD}_4$).}\label{tab2}
\centerline{\begin{tabular}{cccccc}
\hline\noalign{\smallskip}
Method &  $\text{Err}_1$ & $\text{Err}_2$ & $\text{Err}_3$  & $\text{Err}_4$  & Sol. Err. \\
\noalign{\smallskip}\hline\noalign{\smallskip}
$\mbox{MC}_2(0)$	& 1.69e-14 & 7.83e-15 & 1.62e-15 & 0.0461 & 0.0257\\
$\mbox{MC}_2(-0.19)$	& 1.60e-14  & 9.44e-16 & 9.58e-16 & 0.0479 & 0.0061\\
$\mbox{EC}_2(0)$	& 1.33e-14   & 3.80e-15 & 7.32e-05 & 1.44e-15 & 0.0362 \\
$\mbox{EC}_2(-0.18)$	& 1.60e-14   & 4.86e-15 & 7.88e-05 & 1.55e-15 & 0.0057 \\
PS	& 3.11e-16 & 3.44e-05 & 2.84e-04 & 0.0519 & 0.0176\\
MP	&  1.07e-14  & 3.66e-15 & 1.55e-04 & 0.0532 & 0.0676 \\
DVD	& 1.07e-14  & 3.72e-15 & 0.0300 &  1.44e-15 & 0.0422\\
$\mbox{MC}_4(0)$	& 1.33e-14  & 4.55e-15 & 1.71e-15 &  0.0494 & 5.08e-04 \\
$\mbox{MC}_4(0.06)$	& 1.42e-14  & 4.02e-15 & 1.86e-15 &  0.0495 & 1.12e-04 \\
$\mbox{EC}_4(0)$	& 1.15e-14  & 3.16e-15 & 3.91e-07 & 1.21e-14 & 3.20e-04\\
$\mbox{EC}_4(0.02)$	& 1.15e-14  & 6.16e-15 & 3.89e-07 &  1.25e-14 & 2.56e-04\\
$\mbox{FD}_4$\,\,($\Delta t=0.5^2$)	& 2.98e-12  & 4.10e-04 & 4.23e-05 &  0.0522 & 0.0040\\
\noalign{\smallskip}\hline
\end{tabular}}
\end{table}
We obtain the initial conditions from (\ref{ex2sol}) setting
$$p_1=\frac{1}{\sqrt{6}},\qquad p_2=\frac{1}{\sqrt{5}},\qquad d_2=-d_1=20,$$
and solve this problem with $\Delta x=\Delta t=0.5$. The optimal values of the free parameters for each of the families MC$_2(\lambda_1)$, EC$_2(\lambda_2)$, MC$_4(\lambda_3)$,  EC$_4(\lambda_4)$ are $\lambda_1=-0.19,$ $\lambda_2=-0.18$, $\lambda_3=0.06$, $\lambda_4=0.02$. The results in Table~\ref{tab2} are consistent with those in Table~\ref{tab1}, and analogous remarks apply.

Figure~\ref{2sol} shows the solution of the most accurate second-order scheme, EC$_2(-0.18)$, on the whole domain $\Omega$ (upper plot) and at the final time (lower plot). The schemes MP and DVD produce oscillations (amplitude $\simeq 0.005$), where the exact solution is flat. These do not occur in the solution of EC$_2(-0.18)$.

Figure~\ref{2solz} shows how the different schemes approximate the peak of the two solitons (omitting the least accurate solution of MP). The solution of EC$_2(-0.18)$ best reproduces the speed and the amplitude of the two waves.

Finally, Figure~\ref{2solerr} compares the difference between the exact solution and the approximations given by MC$_4$(0.06) and FD$_4$ (with $\Delta t=0.5^2$). Just as for the single soliton, the error of FD$_4$ has a higher amplitude and spreads far from the final location of the two solitons.

\section{The potential Kadomtsev-–Petviashvili (pKP) equation}\label{pKPsec}
This section briefly demonstrates that the novel strategy described in Section~\ref{genericCL} is practicable for PDEs with more than two independent variables. We seek to preserve two conservation laws,
\[
D_xF_i+D_yG_i+D_tH_i=0,
\]
of the pKP equation,
\begin{equation}\label{pKP}
u_{xt}+\tfrac{3}2u_{x}u_{xx}+\tfrac{1}4u_{xxxx}+\tfrac{3}4 u_{yy}=0.
\end{equation}
\noindent\textbf{Note}:\ Throughout this section, $H_i$ and $\widetilde{H}_i$ are components of conservation laws, not Hamiltonians.

The characteristics
$\mathcal{Q}_1=1, \mathcal{Q}_2=u_x,$
correspond respectively to the conservation laws with components
\begin{align}\label{pKPcl1}
F_1=&\tfrac{3}4 u_x^2+\tfrac{1}4 u_{xxx},\quad G_1=\tfrac{3}{4}u_y,\quad H_1=u_x,\\\label{pKPcl2}
F_2=&\tfrac{1}{2}u_x^3+\tfrac{1}{4}u_xu_{xxx}-\tfrac{1}{8}u_{xx}^2+\tfrac{3}{8}uu_{yy},\quad G_2=\tfrac{3}{8}(u_yu_x-uu_{xy}),\quad H_2=\tfrac{1}{2}u_x^2.
\end{align}

We introduce a uniform grid in space with nodes $(x_m,y_n)$ and use $U_{m,n}(t)$ to denote a semidiscrete approximation of $u(x_m,y_n,t)$. In this section, $D_n$ and $\mu_n$ are the forward difference and forward average operators acting on the second index, respectively.
 
The approach in Section~\ref{genericCL} can be applied to a full 15-point rectangular stencil; this yields a wide range of families of methods. For brevity, we present here only those schemes for which all spatial derivatives are approximated on a one-dimensional spatial stencil consisting of three and five points respectively for the $y$- and $x$-derivatives. There are just two one-parameter families, both of the form
\begin{equation}\label{methpKP}
D_m \widetilde{F}_1+D_n \widetilde{G}_1 + D_t\widetilde{H}_1=0
\end{equation}
(so $\widetilde{\mathcal{Q}}_1=1$), that preserve semidiscrete versions of (\ref{pKPcl1}) and (\ref{pKPcl2}).

The first family is defined by
\begin{equation}\label{MC1pKP}
\widetilde F_1=\tfrac{3}{4}(D_m^{(c)}U_{-1,0})(D_m^{(c)}U_{0,0}) + \tfrac{1}{4}D_m^3U_{-2,0},
\quad\widetilde G_1=\tfrac{3}{4}D_nU_{0,-1},
\quad\widetilde H_1=(I+\alpha D_m^2S_m^{-1})D_m^{(c)}U_{0,0};
\end{equation}
the semidiscrete version of (\ref{pKPcl2}) is
\begin{align*}
\widetilde{\mathcal{Q}}_2=&\,\tfrac{1}{3}\Dc_m(U_{-1,0}+ U_{0,0}+ U_{1,0}),\\
\widetilde F_2=&\, (\tfrac{\alpha}{2}-\tfrac{\Delta x^2}{6})\left\{ (\Dc_m\mu_mU_{-1,0})(D_tD_m\Dc_mU_{-1,0})-(D_m\Dc_mU_{-1,0})(D_t\Dc_m\mu_mU_{-1,0})\right\}\\
& + \tfrac{1}{2}(\Dc_m\mu_mU_{-1,0})(\Dc_mU_{-1,0})(\Dc_mU_{0,0})+\tfrac{1}{12} (D_m^3U_{-2,0})D_m(U_{-2,0}+ U_{-1,0}+ U_{0,0})\\
& - \tfrac{1}{24}\{ (D_m^2U_{-2,0})^2+(D_m^2U_{-2,0})(D_m^2U_{-1,0})+(D_m^2U_{-1,0})^2 \}\\
&+\tfrac{1}{8}\mu_m\{ U_{-2,0}D_n^2U_{0,-1}+U_{0,0}D_n^2U_{-2,-1} \}+\tfrac{1}{16}\{ U_{-1,0}D_n^2U_{0,-1}+U_{0,0}D_n^2U_{-1,-1} \},\\
\widetilde G_2=&\,\tfrac{1}{8}(D_nU_{0,-1})\Dc_m\mu_n(U_{-1,-1}+ U_{0,-1}+ U_{1,-1})-\tfrac{1}{8}(\mu_nU_{0,-1})\Dc_mD_n(U_{-1,-1}+ U_{0,-1}+ U_{1,-1}),\\
\widetilde H_2=&\,\tfrac{1}{2}(\Dc_mU_{0,0})^2+\{\tfrac{\alpha}{6}\Dc_m(U_{-1,0}+ U_{1,0})+\tfrac{\alpha+\Delta x^2}{6}\Dc_mU_{0,0}\} D_m^2\Dc_mU_{-1,0}.
\end{align*}
The second family has $\widetilde{G}_1$ and $\widetilde{H}_1$ as above, together with
\begin{align}\label{MC2pKP}
\widetilde F_1= (\Dc_mU_{-1,0})(\Dc_mU_{0,0})-\tfrac{1}{4}(D_mU_{-1,0})^2 + \tfrac{1}{4}D_m^3U_{-2,0}.
\end{align}
The semidiscrete version of (\ref{pKPcl2}) is
\begin{align*}
\widetilde{\mathcal{Q}}_2=&\,\tfrac{1}{2}\Dc_m(U_{-1,0}+ U_{1,0}),\\
\widetilde F_2=&\, (\tfrac{\alpha}{2}-\tfrac{\Delta x^2}{4})\left\{ (\Dc_m\mu_mU_{-1,0})(D_tD_m\Dc_mU_{-1,0})-(D_m\Dc_mU_{-1,0})(D_t\Dc_m\mu_mU_{-1,0})\right\}\\
& + \tfrac{1}{4}(\Dc_mU_{-1,0})(\Dc_mU_{0,0})D_m(U_{-2,0}+U_{0,0})+\tfrac{1}{8} (D_m^3U_{-2,0})D_m(U_{-2,0}+ U_{0,0})\\
& - \tfrac{1}{16}\{ (D_m^2U_{-2,0})^2+(D_m^2U_{-1,0})^2 \}+\tfrac{3}{16}\mu_m\{ U_{-2,0}D_n^2U_{0,-1}+U_{0,0}D_n^2U_{-2,-1} \},\\
\widetilde G_2=&\,\tfrac{3}{16}(D_nU_{0,-1})\Dc_m\mu_n(U_{-1,-1}+ U_{1,-1})-\tfrac{3}{16}(\mu_nU_{0,-1})\Dc_mD_n(U_{-1,-1}+ U_{1,-1}),\\
\widetilde H_2=&\,\tfrac{1}{2}(\Dc_mU_{0,0})^2+\{\tfrac{\alpha}{4}\Dc_m(U_{-1,0}+ U_{1,0})+\tfrac{\Delta x^2}{4}\Dc_mU_{0,0}\} D_m^2\Dc_mU_{-1,0}.
\end{align*}

For both families, $\alpha=O(\Delta x^2,\Delta y^2).$ Let MC$_1(\alpha)$ and MC$_2(\alpha)$ denote the two families of fully-discrete schemes obtained by applying implicit midpoint in time to (\ref{methpKP}) with (\ref{MC1pKP}) and (\ref{MC2pKP}), respectively.
\subsubsection*{Numerical test}
As a brief test of the above schemes for the pKP equation, we use the following travelling wave solution of (\ref{pKP}) \cite{BK}:
\begin{equation}\label{wave}
u(x,y,t)=2\tanh(x + y - \tfrac{7}4t+5) + 2.
\end{equation}
We apply methods MC$_1(0)$ and MC$_2(0)$ to the pKP equation on the domain $\Omega=[-0.5,0.5]\times [-10,10]\times [0,5]$ with initial and Dirichlet boundary conditions given by (\ref{wave}), using step lengths $\Delta x=0.01, \Delta y=0.2$ and $\Delta t=0.05$.
\begin{figure}
\begin{center}
\includegraphics[scale=0.47]{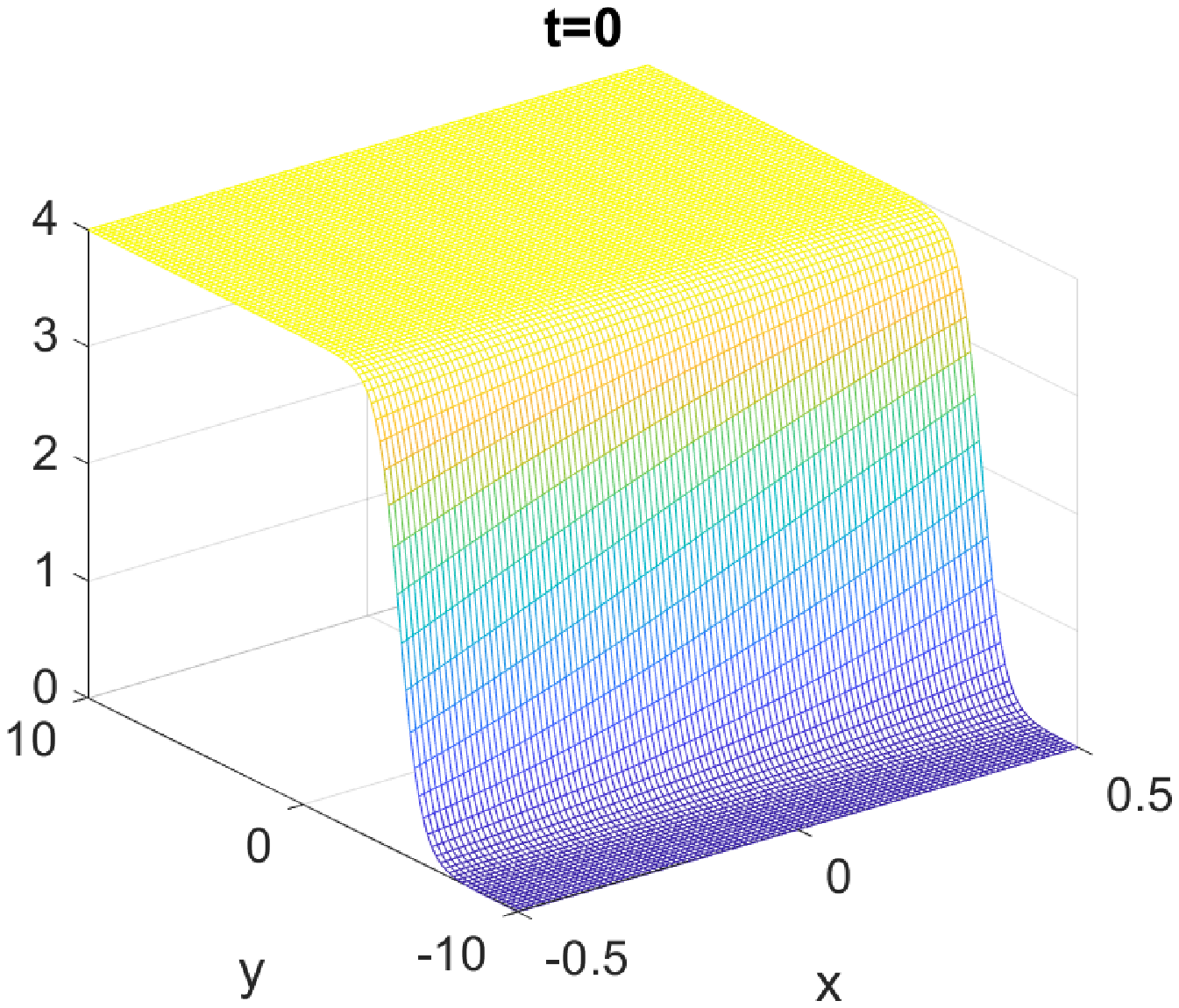}
\includegraphics[scale=0.47]{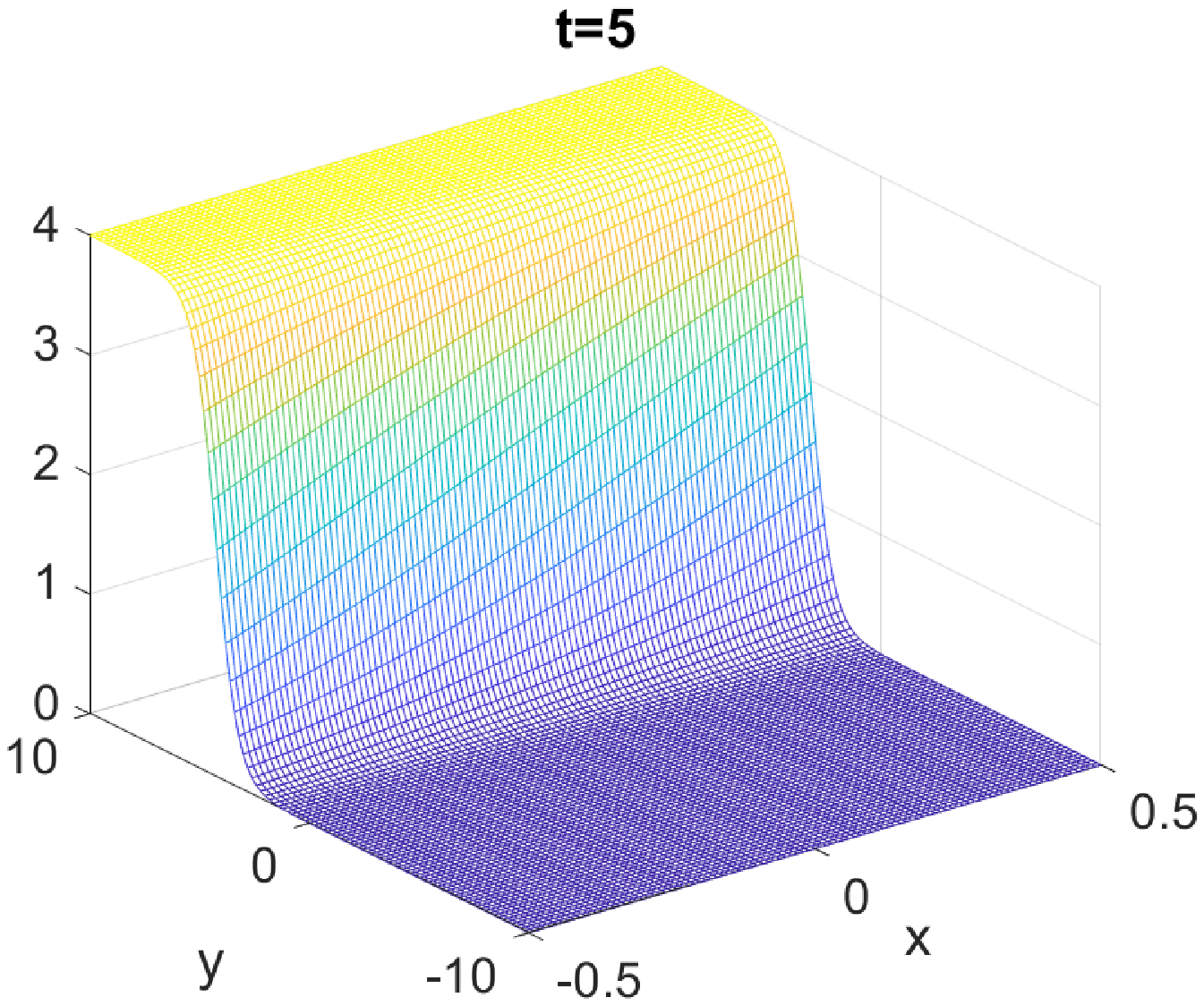}
\end{center}
\caption{Initial condition (left) and solution of MC$_1(0)$ at time $t=5$.}\label{pKPtest}
\end{figure}

Figure~\ref{pKPtest} shows the profile of the wave at the initial time and the numerical solution of MC$_1(0)$ at the final time $t=5$. Both MC$_1(0)$ and MC$_2(0)$ simulate the motion of the wave to the required accuracy, with a maximum absolute error in the solution of $3.40\times 10^{-4}$ and $3.41\times 10^{-4}$, respectively. 

\section{Conclusions}\label{concl}
In this paper we have introduced a new approach to constructing finite difference approximations to a system of PDEs that preserve multiple conservation laws. This is based on discretising one dimension at a time, using semidiscrete Euler operators to find constraints that simplify the remaining symbolic computations. This is much cheaper than the approach introduced in \cite{FCHydon}, and can be iterated to apply to PDEs with more than two independent variables. We have proved that any symplectic Runge--Kutta method preserves local conservation laws with quadratic density and that the AVF method preserves the local conservation law of the energy under milder conditions than the skew-adjointness of the discrete operator $\widetilde{\mathcal{D}}$. 

The new strategy has been applied to obtain methods that preserve either the momentum or the energy of the Boussinesq equation. These are obtained as families that depend on a number of free parameters. Numerical tests have shown that the new schemes are competitive with respect to other methods in the literature and confirmed their conservation properties.
Very accurate solutions can be obtained by selecting optimal parameter values. However, these values depend strongly on the choice of initial condition.

Finally, we have given an example that the new approach is practicable for PDEs with three independent variables, by finding two new families of schemes that preserve two conservation laws for the pKP equation.

\subsection*{Acknowledgements} The authors would like to thank the Isaac Newton Institute for Mathematical Sciences for support and hospitality during the programme {\em Geometry, Compatibility and Structure Preservation in Computational Differential Equations}, when work on this paper was undertaken. This work was supported by EPSRC grant number EP/R014604/1.

The authors are grateful to Dr. Pranav Singh (University of Bath) and the referees for their constructive suggestions which have helped to improve this paper.

\end{document}

%% file: order1sol.tex
\pgfplotsset{
  log x ticks with fixed point/.style={
      xticklabel={
        \pgfkeys{/pgf/fpu=true}
        \pgfmathparse{exp(\tick)}%
        \pgfmathprintnumber[fixed relative, precision=3]{\pgfmathresult}
        \pgfkeys{/pgf/fpu=false}
      }
  },
}
\begin{tikzpicture}

\begin{axis}[%
width=3.5in,
height=2.1in,
at={(1in,10in)},
scale only axis,
xmode=log,
xmin=0.1,
xmax=0.9,
log x ticks with fixed point,
xtick={0.1,0.2,0.3,0.4,0.5,0.6,0.7,0.8,0.9},
xminorticks=true,
ymode=log,
ymin=1e-07,
ymax=1,
yminorticks=true,
axis background/.style={fill=white},
legend columns=2,
legend style={legend pos=south east, legend cell align=left, align=left, draw=white!15!black,text width=3.5em}
]
\addplot [color=blue, line width=1.5pt, mark size=3.0pt, mark=*, mark options={solid, blue}]
  table[row sep=crcr]{%
0.1 0.001744352482318\\
0.2 0.006932809282745\\
0.3	0.015394815532481\\
0.4 0.026901540607160\\
0.5	0.0414870103342257\\
0.6 0.058751742899246\\
0.7 0.078467370171848\\
0.8 0.097984886423440\\
0.9	0.122279591383913\\
};
\addlegendentry{EC$_2$(0)}

\addplot [color=blue, line width=1.5pt, mark size=3.0pt, mark=*, mark options={solid, blue}, densely dashed]
  table[row sep=crcr]{%
0.1 6.310163625023687e-07\\
0.2 1.014844691242605e-05\\
0.3 5.128995782971570e-05\\
0.4 1.642695585517801e-04\\
0.5	0.000412493956780063\\
0.6 8.651670551326446e-04\\
0.7 0.001655553925816\\
0.8 0.002948870766846\\
0.9	0.004870595526870\\
};
\addlegendentry{EC$_4$(0)}

\addplot [color=red, line width=1.5pt, mark size=4.0pt, mark=x, mark options={solid, red}]
  table[row sep=crcr]{%
0.1 0.001209761813698\\
0.2 0.004819513078445\\
0.3 0.010845910281583\\
0.4 0.019103698883414\\
0.5	0.0292950002010354\\
0.6 0.041880187917499\\
0.7 0.056456556421426\\
0.8 0.070841843339690\\
0.9	0.089698515736793\\
};
\addlegendentry{MC$_2$(0)}

\addplot [color=red, line width=1.5pt, mark size=4.0pt, mark=x, mark options={solid, red}, densely dashed]
  table[row sep=crcr]{%
0.1 1.278755235370749e-06\\
0.2 2.041567445049071e-05\\
0.3 1.030454250181083e-04\\
0.4 3.238977725264476e-04\\
0.5	0.000784339380931381\\
0.6 0.001612047955882\\
0.7 0.002980295333747\\
0.8 0.004951959701356\\
0.9	0.007880113942680\\
};
\addlegendentry{MC$_4$(0)}

\addplot [line width=1.5pt]
  table[row sep=crcr]{%
0.1	0.0025\\
0.3	0.0225\\
0.5	0.0625\\
0.7 0.1225\\
0.9	0.2025\\
};
\addlegendentry{slope 2}

\addplot [line width=1.5pt, densely dashed]
  table[row sep=crcr]{%
0.1 0.000002\\
0.3 0.000162\\
0.5	0.00125\\
0.7 0.004802\\
0.9	0.013122\\
};
\addlegendentry{slope 4}

\end{axis}

\end{tikzpicture}%

%% file: 1sol.tex
\definecolor{mycolor1}{rgb}{0.49020,0.18039,0.56078}%
\definecolor{mycolor2}{rgb}{0.03137,0.43137,0.13725}%
\begin{tikzpicture}

\begin{axis}[%
width=3.5in,
height=2.1in,
at={(1in,1in)},
scale only axis,
xmin=-60,
xmax=60,
xlabel style={font=\color{white!15!black}},
xlabel={$x$},
ymin=-0.1,
ymax=0.6,
ylabel style={font=\color{white!15!black}},
ylabel={$-u$},
axis background/.style={fill=white},
legend style={legend cell align=left, align=left, draw=white!15!black}
]
\addplot [color=black, dashed, line width=2.0pt]
  table[row sep=crcr]{%
-60	5.80802121363197e-13\\
-59.5	7.75172240590322e-13\\
-59	1.03458988953317e-12\\
-58.5	1.38082374919545e-12\\
-58	1.8429275654359e-12\\
-57.5	2.45967815474053e-12\\
-57	3.28282930831062e-12\\
-56.5	4.38145464142528e-12\\
-56	5.84774381240645e-12\\
-55.5	7.80473849305841e-12\\
-55	1.04166572440731e-11\\
-54.5	1.39026756933514e-11\\
-54	1.8555318362256e-11\\
-53.5	2.47650054650122e-11\\
-53	3.30528144927091e-11\\
-52.5	4.4114205725893e-11\\
-52	5.88773808429682e-11\\
-51.5	7.85811717080532e-11\\
-51	1.04878995270513e-10\\
-50.5	1.39977597810853e-10\\
-50	1.86822231068565e-10\\
-49.5	2.49343799055987e-10\\
-49	3.327887145484e-10\\
-48.5	4.44159144710939e-10\\
-48	5.92800588477616e-10\\
-47.5	7.91186091441675e-10\\
-47	1.05596290458614e-09\\
-46.5	1.40934941580099e-09\\
-46	1.8809995757968e-09\\
-45.5	2.51049126917785e-09\\
-45	3.35064743893983e-09\\
-44.5	4.47196865201371e-09\\
-44	5.96854905851162e-09\\
-43.5	7.96597217424077e-09\\
-43	1.06318490501248e-08\\
-42.5	1.41898831242864e-08\\
-42	1.89386419889763e-08\\
-41.5	2.52766112782658e-08\\
-41	3.37356330369037e-08\\
-40.5	4.50255345010355e-08\\
-40	6.00936922460179e-08\\
-39.5	8.02045299310253e-08\\
-39	1.07045620559101e-07\\
-38.5	1.42869296632097e-07\\
-38	1.90681651139311e-07\\
-37.5	2.54494788989285e-07\\
-37	3.39663495950285e-07\\
-36.5	4.53334575733367e-07\\
-36	6.05046559868033e-07\\
-35.5	8.0753011267532e-07\\
-35	1.07777629349913e-06\\
-34.5	1.43846231309561e-06\\
-34	1.91985441590629e-06\\
-33.5	2.56234755052194e-06\\
-33	3.41985491376579e-06\\
-32.5	4.56433174826971e-06\\
-32	6.09181291437754e-06\\
-31.5	8.13047071515709e-06\\
-31	1.08513688583097e-05\\
-30.5	1.44828144644922e-05\\
-30	1.9329511570312e-05\\
-29.5	2.57981217968404e-05\\
-29	3.44313742770318e-05\\
-28.5	4.59535821511203e-05\\
-28	6.1331376216207e-05\\
-27.5	8.18547363158006e-05\\
-27	0.000109245091975443\\
-26.5	0.000145799513300267\\
-26	0.000194583002497128\\
-25.5	0.00025968486303756\\
-25	0.000346560376495343\\
-24.5	0.000462485954950725\\
-24	0.000617165066510975\\
-23.5	0.000823534185461705\\
-23	0.00109883356290963\\
-22.5	0.00146602763291767\\
-22	0.00195568530017676\\
-21.5	0.00260846202172367\\
-21	0.00347836384976386\\
-20.5	0.00463701762051948\\
-20	0.00617921792706504\\
-19.5	0.00823006218039572\\
-19	0.0109540026712554\\
-18.5	0.0145661058297424\\
-18	0.0193456535337962\\
-17.5	0.0256518456492617\\
-17	0.0339405985529136\\
-16.5	0.0447800269470783\\
-16	0.058859803290479\\
-15.5	0.0769858421438846\\
-15	0.100046484281684\\
-14.5	0.128930127524099\\
-14	0.164369420824856\\
-13.5	0.206689324119624\\
-13	0.255455413842843\\
-12.5	0.309066612941178\\
-12	0.364416548968853\\
-11.5	0.416834902858224\\
-11	0.460543395050852\\
-10.5	0.489726319629149\\
-10	0.5\\
-9.5	0.489726319629149\\
-9	0.460543395050852\\
-8.5	0.416834902858224\\
-8	0.364416548968853\\
-7.5	0.309066612941178\\
-7	0.255455413842843\\
-6.5	0.206689324119624\\
-6	0.164369420824856\\
-5.5	0.128930127524099\\
-5	0.100046484281684\\
-4.5	0.0769858421438846\\
-4	0.058859803290479\\
-3.5	0.0447800269470783\\
-3	0.0339405985529136\\
-2.5	0.0256518456492617\\
-2	0.0193456535337962\\
-1.5	0.0145661058297424\\
-1	0.0109540026712554\\
-0.5	0.00823006218039572\\
0	0.00617921792706504\\
0.5	0.00463701762051948\\
1	0.00347836384976386\\
1.5	0.00260846202172367\\
2	0.00195568530017676\\
2.5	0.00146602763291767\\
3	0.00109883356290963\\
3.5	0.000823534185461705\\
4	0.000617165066510975\\
4.5	0.000462485954950725\\
5	0.000346560376495343\\
5.5	0.00025968486303756\\
6	0.000194583002497128\\
6.5	0.000145799513300267\\
7	0.000109245091975443\\
7.5	8.18547363158006e-05\\
8	6.1331376216207e-05\\
8.5	4.59535821511203e-05\\
9	3.44313742770318e-05\\
9.5	2.57981217968404e-05\\
10	1.9329511570312e-05\\
10.5	1.44828144644922e-05\\
11	1.08513688583097e-05\\
11.5	8.13047071515709e-06\\
12	6.09181291437754e-06\\
12.5	4.56433174826971e-06\\
13	3.41985491376579e-06\\
13.5	2.56234755052194e-06\\
14	1.91985441590629e-06\\
14.5	1.43846231309561e-06\\
15	1.07777629349913e-06\\
15.5	8.0753011267532e-07\\
16	6.05046559868033e-07\\
16.5	4.53334575733367e-07\\
17	3.39663495950285e-07\\
17.5	2.54494788989285e-07\\
18	1.90681651139311e-07\\
18.5	1.42869296632097e-07\\
19	1.07045620559101e-07\\
19.5	8.02045299310253e-08\\
20	6.00936922460179e-08\\
20.5	4.50255345010355e-08\\
21	3.37356330369037e-08\\
21.5	2.52766112782658e-08\\
22	1.89386419889763e-08\\
22.5	1.41898831242864e-08\\
23	1.06318490501248e-08\\
23.5	7.96597217424077e-09\\
24	5.96854905851162e-09\\
24.5	4.47196865201371e-09\\
25	3.35064743893983e-09\\
25.5	2.51049126917785e-09\\
26	1.8809995757968e-09\\
26.5	1.40934941580099e-09\\
27	1.05596290458614e-09\\
27.5	7.91186091441675e-10\\
28	5.92800588477616e-10\\
28.5	4.44159144710939e-10\\
29	3.327887145484e-10\\
29.5	2.49343799055987e-10\\
30	1.86822231068565e-10\\
30.5	1.39977597810853e-10\\
31	1.04878995270513e-10\\
31.5	7.85811717080532e-11\\
32	5.88773808429682e-11\\
32.5	4.4114205725893e-11\\
33	3.30528144927091e-11\\
33.5	2.47650054650122e-11\\
34	1.8555318362256e-11\\
34.5	1.39026756933514e-11\\
35	1.04166572440731e-11\\
35.5	7.80473849305841e-12\\
36	5.84774381240645e-12\\
36.5	4.38145464142528e-12\\
37	3.28282930831062e-12\\
37.5	2.45967815474053e-12\\
38	1.8429275654359e-12\\
38.5	1.38082374919545e-12\\
39	1.03458988953317e-12\\
39.5	7.75172240590322e-13\\
40	5.80802121363197e-13\\
40.5	4.35169226291027e-13\\
41	3.26052968033674e-13\\
41.5	2.44297003420155e-13\\
42	1.83040891300527e-13\\
42.5	1.37144407909371e-13\\
43	1.02756211943542e-13\\
43.5	7.69906644678043e-14\\
44	5.76856844280205e-14\\
44.5	4.32213205449175e-14\\
45	3.23838152943728e-14\\
45.5	2.42637540870643e-14\\
46	1.81797529736969e-14\\
46.5	1.3621281232851e-14\\
47	1.02058208762717e-14\\
47.5	7.64676816945383e-15\\
48	5.72938366705227e-15\\
48.5	4.29277264288106e-15\\
49	3.21638382666545e-15\\
49.5	2.40989350730955e-15\\
50	1.80562614089303e-15\\
50.5	1.35287544897207e-15\\
51	1.01364946983218e-15\\
51.5	7.59482514426418e-16\\
52	5.69046506594602e-16\\
52.5	4.26361266410555e-16\\
53	3.19453555005684e-16\\
53.5	2.39352356429824e-16\\
54	1.79336086986073e-16\\
54.5	1.34368562629571e-16\\
55	1.00676394397627e-16\\
55.5	7.54323495805247e-17\\
56	5.6518108314104e-17\\
56.5	4.2346507634564e-17\\
57	3.17283568458835e-17\\
57.5	2.37726481916076e-17\\
58	1.7811789144551e-17\\
58.5	1.33455822831697e-17\\
59	9.99925190172935e-18\\
59.5	7.49199521405149e-18\\
60	5.61341916765418e-18\\
};

\addplot [color=black, line width=2.0pt]
  table[row sep=crcr]{%
-60	5.22971461711241e-05\\
-59.5	6.93046641608129e-05\\
-59	6.25983984391501e-05\\
-58.5	3.12129883081544e-05\\
-58	-1.58886751498603e-05\\
-57.5	-6.31931540887579e-05\\
-57	-9.49588418493192e-05\\
-56.5	-0.000101261333689385\\
-56	-8.08666671607681e-05\\
-55.5	-4.00819431040394e-05\\
-55	1.09069167469818e-05\\
-54.5	6.16400114406656e-05\\
-54	0.000103210731246793\\
-53.5	0.000128483817530767\\
-53	0.000132156804229793\\
-52.5	0.000111766776924711\\
-52	6.9234319816758e-05\\
-51.5	1.14160744790625e-05\\
-51	-5.14218067595988e-05\\
-50.5	-0.000108651777286662\\
-50	-0.000152027850829021\\
-49.5	-0.000176654586725724\\
-49	-0.000180337141354967\\
-48.5	-0.000162726050038501\\
-48	-0.000125347183824724\\
-47.5	-7.21789791000105e-05\\
-47	-9.54677649435309e-06\\
-46.5	5.52587391160292e-05\\
-46	0.000115828444697475\\
-45.5	0.00016760826138364\\
-45	0.000207617319103854\\
-44.5	0.000233823997045188\\
-44	0.000245197101396996\\
-43.5	0.000242316699835137\\
-43	0.000227544941682261\\
-42.5	0.000204245981937288\\
-42	0.000175722890221725\\
-41.5	0.000144832231492807\\
-41	0.000114287948399232\\
-40.5	8.68053108520987e-05\\
-40	6.46225402715558e-05\\
-39.5	4.89485595990731e-05\\
-39	4.00422193406309e-05\\
-38.5	3.77493282963129e-05\\
-38	4.17493376929814e-05\\
-37.5	5.13192621777425e-05\\
-37	6.51955422441576e-05\\
-36.5	8.19142965420239e-05\\
-36	0.000100223992016189\\
-35.5	0.00011907660357989\\
-35	0.000137401301881283\\
-34.5	0.000154157645411479\\
-34	0.000168612147514418\\
-33.5	0.000180385269949581\\
-33	0.000189233346572507\\
-32.5	0.00019496478060774\\
-32	0.000197583098078739\\
-31.5	0.000197324263340159\\
-31	0.000194488104518955\\
-30.5	0.000189351065318263\\
-30	0.000182253309535084\\
-29.5	0.000173616699059136\\
-29	0.000163830383586717\\
-28.5	0.00015321439378131\\
-28	0.00014208772243426\\
-27.5	0.000130763563993774\\
-27	0.000119484325012187\\
-26.5	0.000108434850312123\\
-26	9.77852436645591e-05\\
-25.5	8.76688832971461e-05\\
-25	7.8164692914888e-05\\
-24.5	6.93275234390876e-05\\
-24	6.11958261952051e-05\\
-23.5	5.37767023759136e-05\\
-23	4.70581522096349e-05\\
-22.5	4.10211729156727e-05\\
-22	3.5633509289926e-05\\
-21.5	3.08539307588016e-05\\
-21	2.6639775898342e-05\\
-20.5	2.29451163470729e-05\\
-20	1.97227647882846e-05\\
-19.5	1.69275035887047e-05\\
-19	1.4515687808646e-05\\
-18.5	1.24463719979218e-05\\
-18	1.06822889346677e-05\\
-17.5	9.18991474159913e-06\\
-17	7.94001627335327e-06\\
-16.5	6.90797559891898e-06\\
-16	6.07410207424525e-06\\
-15.5	5.42411843039558e-06\\
-15	4.94973437325107e-06\\
-14.5	4.64950941987649e-06\\
-14	4.53003151052197e-06\\
-13.5	4.60758576812671e-06\\
-13	4.9104599172772e-06\\
-12.5	5.48211533719665e-06\\
-12	6.38550167987248e-06\\
-11.5	7.70891314246175e-06\\
-11	9.57389074899536e-06\\
-10.5	1.21458364033597e-05\\
-10	1.56482903649718e-05\\
-9.5	2.03820331345949e-05\\
-9	2.67506823679201e-05\\
-8.5	3.52949307449317e-05\\
-8	4.67383292299421e-05\\
-7.5	6.20484900673165e-05\\
-7	8.25188874497765e-05\\
-6.5	0.000109878157809991\\
-6	0.000146436112084209\\
-5.5	0.000195278787764111\\
-5	0.00026052888184968\\
-4.5	0.000347693431751466\\
-4	0.000464127731522783\\
-3.5	0.000619653965341707\\
-3	0.000827385518487985\\
-2.5	0.00110482409279205\\
-2	0.00147531770186357\\
-1.5	0.00196999438526822\\
-1	0.00263031950397557\\
-0.5	0.00351146516060499\\
0	0.00468672685694199\\
0.5	0.00625327052120669\\
1	0.00833953939258619\\
1.5	0.011114665922191\\
2	0.0148001923243096\\
2.5	0.0196842463360111\\
3	0.0261379022068102\\
3.5	0.0346326415712006\\
4	0.0457563282287319\\
4.5	0.0602224940033035\\
5	0.0788637018889467\\
5.5	0.102594098752708\\
6	0.132319556974194\\
6.5	0.168768872772726\\
7	0.212222632232243\\
7.5	0.262138388883935\\
8	0.316724703199683\\
8.5	0.372604868397825\\
9	0.424802537138947\\
9.5	0.467294010669231\\
10	0.494204206708257\\
10.5	0.501366374845659\\
11	0.48761757977669\\
11.5	0.455185239456332\\
12	0.408956297618658\\
12.5	0.355017404317761\\
13	0.299156894067699\\
13.5	0.245865400104038\\
14	0.19798701160118\\
14.5	0.156866220140936\\
15	0.122734560203621\\
15.5	0.0951305311359915\\
16	0.0732432827082264\\
16.5	0.0561489608687337\\
17	0.0429510863918921\\
17.5	0.0328507896809583\\
18	0.0251722612345276\\
18.5	0.0193629170945025\\
19	0.0149813077771114\\
19.5	0.0116804965727913\\
20	0.00919098244386109\\
20.5	0.00730514575514726\\
21	0.00586411799065117\\
21.5	0.00474728398505778\\
22	0.00386407365656741\\
22.5	0.00314747187619256\\
23	0.0025488624196913\\
23.5	0.00203409473116494\\
24	0.00158062168380313\\
24.5	0.00117527302069542\\
25	0.000812158063873163\\
25.5	0.000490536810392844\\
26	0.000212897359839125\\
26.5	-1.65427491503208e-05\\
27	-0.000193080635964636\\
27.5	-0.000312899500742495\\
28	-0.000374638488613052\\
28.5	-0.000380705058987724\\
29	-0.000337750029892349\\
29.5	-0.000256192331902553\\
30	-0.000149244406770292\\
30.5	-3.19286353639604e-05\\
31	7.99361423075773e-05\\
31.5	0.00017123503490004\\
32	0.000229774622391482\\
32.5	0.000248685749571017\\
33	0.000227935767342588\\
33.5	0.00017415407048983\\
34	9.88983962677388e-05\\
34.5	1.62549209790512e-05\\
35	-5.93651325342568e-05\\
35.5	-0.000115057820733929\\
36	-0.000141296364365238\\
36.5	-0.000134079836111616\\
37	-9.67791779365258e-05\\
37.5	-4.04013496981085e-05\\
38	1.87870786852171e-05\\
38.5	6.40950292463982e-05\\
39	8.39961465727504e-05\\
39.5	7.60160306281882e-05\\
40	4.70747422698885e-05\\
40.5	1.01248710781398e-05\\
41	-2.12948582893752e-05\\
41.5	-3.80903828582609e-05\\
42	-3.81560271785746e-05\\
42.5	-2.57406985968016e-05\\
43	-8.45864897264039e-06\\
43.5	6.35411250523142e-06\\
44	1.43014199672428e-05\\
44.5	1.47582183375389e-05\\
45	1.00065424668062e-05\\
45.5	3.49017584088826e-06\\
46	-1.85432396065082e-06\\
46.5	-4.54265816515618e-06\\
47	-4.61182581646779e-06\\
47.5	-3.06844094724805e-06\\
48	-1.1419546660465e-06\\
48.5	2.9298136850291e-07\\
49	9.48752269780361e-07\\
49.5	1.02811022002341e-06\\
50	9.35566080942224e-07\\
50.5	9.66291253518935e-07\\
51	1.11138858977055e-06\\
51.5	1.05727450067139e-06\\
52	3.84172950607999e-07\\
52.5	-1.10952755626382e-06\\
53	-3.10413567248965e-06\\
53.5	-4.66452700527345e-06\\
54	-4.5070856443948e-06\\
54.5	-1.63815397336257e-06\\
55	3.84926197838365e-06\\
55.5	1.02412246564261e-05\\
56	1.43961812870508e-05\\
56.5	1.28972110195474e-05\\
57	3.89972122836444e-06\\
57.5	-1.11392441320757e-05\\
58	-2.69309995608575e-05\\
58.5	-3.57954820953609e-05\\
59	-3.09323982485441e-05\\
59.5	-1.02446747622839e-05\\
60	2.1276626965401e-05\\
};

\end{axis}

\begin{axis}[%
width=3.5in,
height=2.1in,
at={(1in,-1.5in)},
scale only axis,
xmin=8,
xmax=13,
ymin=0.36,
ymax=0.52,
xlabel={$x$},
ylabel={$-u$},
axis background/.style={fill=white},
legend style={legend cell align=left, align=left, draw=white!15!black}
]

\addplot [color=blue, line width=1.0pt, mark size=3.0pt, mark=triangle*, mark repeat=8, mark options={solid, rotate=180, blue}]
  table[row sep=crcr]{%
8	0.310238716957877\\
8.0625	0.317022269675166\\
8.125	0.323821972865575\\
8.1875	0.330628801788163\\
8.25	0.337433731701986\\
8.3125	0.344227737866103\\
8.375	0.351001795539571\\
8.4375	0.357746879981448\\
8.5	0.364453966450791\\
8.5625	0.37111368493105\\
8.625	0.377715284303247\\
8.6875	0.384247668172796\\
8.75	0.390699740145111\\
8.8125	0.397060403825607\\
8.875	0.403318562819697\\
8.9375	0.409463120732796\\
9	0.415482981170318\\
9.0625	0.421367043437357\\
9.125	0.427104189637724\\
9.1875	0.432683297574909\\
9.25	0.438093245052404\\
9.3125	0.443322909873699\\
9.375	0.448361169842284\\
9.4375	0.453196902761651\\
9.5	0.45781898643529\\
9.5625	0.462216757414815\\
9.625	0.466381387244333\\
9.6875	0.470304506216074\\
9.75	0.473977744622267\\
9.8125	0.477392732755144\\
9.875	0.480541100906934\\
9.9375	0.483414479369867\\
10	0.486004498436173\\
10.0625	0.48830365843435\\
10.125	0.490307939837966\\
10.1875	0.492014193156855\\
10.25	0.493419268900853\\
10.3125	0.494520017579795\\
10.375	0.495313289703516\\
10.4375	0.495795935781853\\
10.5	0.495964806324639\\
10.5625	0.495817788528535\\
10.625	0.495356916337506\\
10.6875	0.494585260382339\\
10.75	0.493505891293824\\
10.8125	0.492121879702748\\
10.875	0.490436296239902\\
10.9375	0.488452211536073\\
11	0.486172696222051\\
11.0625	0.483601694613405\\
11.125	0.480746645764835\\
11.1875	0.477615862415823\\
11.25	0.474217657305847\\
11.3125	0.47056034317439\\
11.375	0.466652232760932\\
11.4375	0.462501638804954\\
11.5	0.458116874045936\\
11.5625	0.453506716053924\\
11.625	0.448681801721221\\
11.6875	0.443653232770692\\
11.75	0.438432110925206\\
11.8125	0.433029537907627\\
11.875	0.427456615440825\\
11.9375	0.421724445247664\\
12	0.415844129051013\\
12.0625	0.409826770111084\\
12.125	0.403683477837478\\
12.1875	0.397425363177142\\
12.25	0.391063537077023\\
12.3125	0.384609110484068\\
12.375	0.378073194345225\\
12.4375	0.371466899607441\\
12.5	0.364801337217662\\
12.5625	0.358087275903681\\
12.625	0.351334115516662\\
12.6875	0.344550913688617\\
12.75	0.337746728051556\\
12.8125	0.330930616237489\\
12.875	0.324111635878427\\
12.9375	0.31729884460638\\
13	0.310501300053358\\
};

\addplot [color=mycolor2, line width=1.0pt, mark size=3.0pt, mark=triangle*, mark repeat=8, mark phase=6 mark options={solid, mycolor2}]
  table[row sep=crcr]{%
8	0.344972493603584\\
8.0625	0.352195808200491\\
8.125	0.359401412744931\\
8.1875	0.366577286247375\\
8.25	0.373711407718297\\
8.3125	0.380791756168167\\
8.375	0.387806310607458\\
8.4375	0.394743050046641\\
8.5	0.401589953496189\\
8.5625	0.408334731760052\\
8.625	0.414964022816095\\
8.6875	0.421464196435662\\
8.75	0.427821622390097\\
8.8125	0.434022670450742\\
8.875	0.440053710388941\\
8.9375	0.445901111976039\\
9	0.451551244983378\\
9.0625	0.456990779141684\\
9.125	0.462207584019217\\
9.1875	0.467189829143616\\
9.25	0.471925684042521\\
9.3125	0.476403318243574\\
9.375	0.480610901274415\\
9.4375	0.484536602662685\\
9.5	0.488168591936023\\
9.5625	0.49149597529875\\
9.625	0.494511605661897\\
9.6875	0.497209272613173\\
9.75	0.499582765740289\\
9.8125	0.501625874630954\\
9.875	0.503332388872878\\
9.9375	0.504696098053771\\
10	0.505710791761343\\
10.0625	0.506371576015365\\
10.125	0.506678822563854\\
10.1875	0.506634219586888\\
10.25	0.506239455264546\\
10.3125	0.505496217776904\\
10.375	0.504406195304041\\
10.4375	0.502971076026036\\
10.5	0.501192548122966\\
10.5625	0.499073510171666\\
10.625	0.496621702335997\\
10.6875	0.493846075176575\\
10.75	0.490755579254019\\
10.8125	0.487359165128945\\
10.875	0.483665783361971\\
10.9375	0.479684384513714\\
11	0.475423919144791\\
11.0625	0.470894022666758\\
11.125	0.466107069894922\\
11.1875	0.46107612049553\\
11.25	0.455814234134827\\
11.3125	0.45033447047906\\
11.375	0.444649889194474\\
11.4375	0.438773549947317\\
11.5	0.432718512403833\\
11.5625	0.426497877553968\\
11.625	0.420124911682466\\
11.6875	0.413612922397769\\
11.75	0.406975217308321\\
11.8125	0.400225104022564\\
11.875	0.393375890148941\\
11.9375	0.386440883295895\\
12	0.379433391071868\\
12.0625	0.372366290730926\\
12.125	0.365250738109621\\
12.1875	0.358097458690126\\
12.25	0.350917177954615\\
12.3125	0.343720621385262\\
12.375	0.336518514464242\\
12.4375	0.329321582673727\\
12.5	0.322140551495893\\
12.5625	0.314985525113975\\
12.625	0.307864122515466\\
12.6875	0.300783341388919\\
12.75	0.293750179422888\\
12.8125	0.286771634305929\\
12.875	0.279854703726594\\
12.9375	0.27300638537344\\
13	0.266233676935019\\
};

\addplot [color=black, line width=2.0pt]
  table[row sep=crcr]{%
8	0.318749627527025\\
8.0625	0.325682378739069\\
8.125	0.332624831122036\\
8.1875	0.339568172584147\\
8.25	0.346503213492094\\
8.3125	0.353420399444631\\
8.375	0.360309826567588\\
8.4375	0.367161259395944\\
8.5	0.37396415139141\\
8.5625	0.380707668124988\\
8.625	0.387380713133038\\
8.6875	0.39397195643288\\
8.75	0.40046986565979\\
8.8125	0.406862739761797\\
8.875	0.413138745162111\\
8.9375	0.419285954271569\\
9	0.425292386205526\\
9.0625	0.431146049531503\\
9.125	0.436834986845895\\
9.1875	0.442347320950736\\
9.25	0.447671302375094\\
9.3125	0.452795357960802\\
9.375	0.457708140209158\\
9.4375	0.462398577064489\\
9.5	0.466855921792508\\
9.5625	0.471069802596466\\
9.625	0.475030271602797\\
9.6875	0.478727852840364\\
9.75	0.482153588834001\\
9.8125	0.485299085433947\\
9.875	0.488156554508147\\
9.9375	0.490718854134347\\
10	0.492979525943445\\
10.0625	0.494932829284629\\
10.125	0.496573771906195\\
10.1875	0.49789813687354\\
10.25	0.498902505477111\\
10.3125	0.499584275917968\\
10.375	0.499941677596336\\
10.4375	0.499973780868851\\
10.5	0.499680502182375\\
10.5625	0.499062604535761\\
10.625	0.498121693265221\\
10.6875	0.496860207193188\\
10.75	0.495281405224298\\
10.8125	0.493389348514625\\
10.875	0.491188878380995\\
10.9375	0.488685590155524\\
11	0.485885803225947\\
11.0625	0.4827965275344\\
11.125	0.479425426835608\\
11.1875	0.475780779039688\\
11.25	0.471871433984699\\
11.3125	0.467706768999508\\
11.375	0.46329664262838\\
11.4375	0.458651346895025\\
11.5	0.45378155848555\\
11.5625	0.448698289227219\\
11.625	0.443412836233155\\
11.6875	0.437936732072467\\
11.75	0.432281695311104\\
11.8125	0.426459581751307\\
11.875	0.420482336677297\\
11.9375	0.414361948392235\\
12	0.408110403306898\\
12.0625	0.401739642814372\\
12.125	0.395261522157931\\
12.1875	0.388687771471307\\
12.25	0.382029959142506\\
12.3125	0.375299457624233\\
12.375	0.368507411786474\\
12.4375	0.361664709879976\\
12.5	0.354781957153702\\
12.5625	0.34786945214487\\
12.625	0.340937165637358\\
12.6875	0.333994722263006\\
12.75	0.327051384700942\\
12.8125	0.320116040412555\\
12.875	0.313197190834136\\
12.9375	0.306302942935636\\
13	0.299441003042321\\
};
\addplot [dotted, color=red, line width=1.5pt, mark size=3.0pt, mark=*, mark repeat=8, mark phase=2, mark options={solid, red}]
  table[row sep=crcr]{%
8	0.316724703199683\\
8.0625	0.323727360526101\\
8.125	0.330744764708751\\
8.1875	0.337767022794282\\
8.25	0.344784241829342\\
8.3125	0.351786528860579\\
8.375	0.35876399093464\\
8.4375	0.365706735098172\\
8.5	0.372604868397825\\
8.5625	0.379448118481507\\
8.625	0.386224695402173\\
8.6875	0.392922429814039\\
8.75	0.399529152371321\\
8.8125	0.406032693728237\\
8.875	0.412420884539002\\
8.9375	0.418681555457833\\
9	0.424802537138947\\
9.0625	0.430771685146975\\
9.125	0.436576954688217\\
9.1875	0.442206325879388\\
9.25	0.447647778837205\\
9.3125	0.452889293678382\\
9.375	0.457918850519635\\
9.4375	0.46272442947768\\
9.5	0.467294010669231\\
9.5625	0.471616144234628\\
9.625	0.475681660408694\\
9.6875	0.479481959449879\\
9.75	0.483008441616631\\
9.8125	0.486252507167396\\
9.875	0.489205556360624\\
9.9375	0.491858989454762\\
10	0.494204206708257\\
10.0625	0.496233639960337\\
10.125	0.497943847373345\\
10.1875	0.499332418690401\\
10.25	0.500396943654627\\
10.3125	0.501135012009144\\
10.375	0.501544213497075\\
10.4375	0.501622137861539\\
10.5	0.501366374845659\\
10.5625	0.500775684672197\\
10.625	0.499853509482476\\
10.6875	0.498604461897459\\
10.75	0.497033154538112\\
10.8125	0.495144200025397\\
10.875	0.492942210980279\\
10.9375	0.490431800023722\\
11	0.48761757977669\\
11.0625	0.484505071144517\\
11.125	0.481103428170015\\
11.1875	0.477422713180367\\
11.25	0.473472988502755\\
11.3125	0.469264316464364\\
11.375	0.464806759392374\\
11.4375	0.460110379613969\\
11.5	0.455185239456332\\
11.5625	0.450041792015812\\
11.625	0.444692053465431\\
11.6875	0.439148430747379\\
11.75	0.433423330803844\\
11.8125	0.427529160577015\\
11.875	0.421478327009082\\
11.9375	0.415283237042233\\
12	0.408956297618658\\
12.0625	0.402509787495911\\
12.125	0.39595547269301\\
12.1875	0.389304991044336\\
12.25	0.382569980384271\\
12.3125	0.375762078547199\\
12.375	0.368892923367502\\
12.4375	0.361974152679562\\
12.5	0.355017404317761\\
12.5625	0.348033855438313\\
12.625	0.341032840484755\\
12.6875	0.334023233222455\\
12.75	0.327013907416781\\
12.8125	0.3200137368331\\
12.875	0.313031595236781\\
12.9375	0.306076356393191\\
13	0.299156894067699\\
};
\end{axis}

\begin{axis}[%
width=1.2in,
height=1in,
at={(3.3in,-0.4in)},
scale only axis,
xmin=0,
xmax=1,
ymin=0,
ymax=1,
axis line style={draw=none},
ticks=none,
axis x line*=bottom,
axis y line*=left,
legend style={legend cell align=left, align=left, draw=white!15!black, font = {\fontsize{8.3 pt}{12 pt}\selectfont}}
]

\addplot [color=black, line width=2.0pt, mark=none, mark options={solid, black}]
  table[row sep=crcr]{%
1	1\\
};
\addlegendentry{Exact}

\addplot [dotted, color=red, line width=1.5pt, mark size=3.0pt, mark=*, mark options={solid, red}]
  table[row sep=crcr]{%
1	2\\
};
\addlegendentry{$\text{MC}_2(-0.21)$}

\addplot [color=mycolor2, line width=1.0pt, mark size=3.0pt, mark=triangle*, mark options={solid, mycolor2}]
  table[row sep=crcr]{%
1	3\\
};
\addlegendentry{MP}

\addplot [color=blue, line width=1.0pt, mark size=3.0pt, mark=triangle*, mark options={solid, rotate=180, blue}]
  table[row sep=crcr]{%
1	4\\
};
\addlegendentry{PS}
\end{axis}

\end{tikzpicture}%

%% file: 1sol2.tex
\definecolor{mycolor1}{rgb}{0.03137,0.43137,0.13725}%
\begin{tikzpicture}

\begin{axis}[%
width=3.5in,
height=2.1in,
at={(1in,10in)},
scale only axis,
xmin=-60,
xmax=60,
xlabel style={font=\color{white!15!black}},
xlabel={$x$},
ymin=-0.002,
ymax=0.0016,
ylabel style={font=\color{white!15!black}},
ylabel={${u}_{\text{exact}}{-u}$},
axis background/.style={fill=white},
legend style={legend cell align=left, align=left, draw=white!15!black, font = {\fontsize{8.3 pt}{12 pt}\selectfont}}
]
\addplot [color=blue, line width=1.5pt]
  table[row sep=crcr]{%
-60	2.09779040594895e-06\\
-59.5	-6.36477292414493e-07\\
-59	-3.4298757570314e-06\\
-58.5	-6.06115749400015e-06\\
-58	-8.16601504121951e-06\\
-57.5	-9.18412982853208e-06\\
-57	-8.51545993290154e-06\\
-56.5	-5.86739059280536e-06\\
-56	-1.53933173042252e-06\\
-55.5	3.56444796437514e-06\\
-55	8.19430308569233e-06\\
-54.5	1.11704266121013e-05\\
-54	1.17884928733276e-05\\
-53.5	1.00535804866774e-05\\
-53	6.66170008477805e-06\\
-52.5	2.6959682736934e-06\\
-52	-8.98571710040865e-07\\
-51.5	-3.72826558178627e-06\\
-51	-5.90515606564596e-06\\
-50.5	-7.66650287565512e-06\\
-50	-9.0584822053933e-06\\
-49.5	-9.84161960555616e-06\\
-49	-9.59753107368637e-06\\
-48.5	-8.00745220611232e-06\\
-48	-5.13225123104986e-06\\
-47.5	-1.50713647003176e-06\\
-47	1.9907007255287e-06\\
-46.5	4.540308273444e-06\\
-46	5.88609408303029e-06\\
-45.5	6.44621677695277e-06\\
-45	6.82730702719191e-06\\
-44.5	7.24784631918162e-06\\
-44	7.39073677343779e-06\\
-43.5	6.75319470682545e-06\\
-43	5.2007653482819e-06\\
-42.5	3.19092165613543e-06\\
-42	1.39545421073113e-06\\
-41.5	1.63462792398685e-07\\
-41	-6.11700292309235e-07\\
-40.5	-1.23260109062504e-06\\
-40	-1.86926121770837e-06\\
-39.5	-2.40655620787714e-06\\
-39	-2.59505975367806e-06\\
-38.5	-2.34353851505582e-06\\
-38	-1.79112061210723e-06\\
-37.5	-1.13820186665338e-06\\
-37	-5.1214592024746e-07\\
-36.5	6.18447491079692e-08\\
-36	6.34923044511746e-07\\
-35.5	1.23379759866907e-06\\
-35	1.79607323597123e-06\\
-34.5	2.2486625678003e-06\\
-34	2.57526024261184e-06\\
-33.5	2.78152510289834e-06\\
-33	2.88423912087836e-06\\
-32.5	2.91945962769304e-06\\
-32	2.89643979258121e-06\\
-31.5	2.79838620843439e-06\\
-31	2.64060790021145e-06\\
-30.5	2.45185871268769e-06\\
-30	2.2349733948044e-06\\
-29.5	2.00209158009609e-06\\
-29	1.77787848227139e-06\\
-28.5	1.55760325954094e-06\\
-28	1.33556923362468e-06\\
-27.5	1.13328421301534e-06\\
-27	9.54104491028094e-07\\
-26.5	7.84024382315618e-07\\
-26	6.34742480500219e-07\\
-25.5	5.13976870613091e-07\\
-25	4.03632552381795e-07\\
-24.5	3.05585739254452e-07\\
-24	2.33618176522938e-07\\
-23.5	1.73079960893761e-07\\
-23	1.15875782826244e-07\\
-22.5	7.75775978566479e-08\\
-22	5.13452715373949e-08\\
-21.5	2.19374779600964e-08\\
-21	2.55126995921824e-09\\
-20.5	-3.82991226791426e-09\\
-20	-1.54558221709649e-08\\
-19.5	-2.57932696859744e-08\\
-19	-2.32732847177264e-08\\
-18.5	-2.41028991759184e-08\\
-18	-3.06021424020655e-08\\
-17.5	-2.63040613924943e-08\\
-17	-2.09399028293447e-08\\
-16.5	-2.52878045049107e-08\\
-16	-2.30409557368142e-08\\
-15.5	-1.4803312381233e-08\\
-15	-1.68216171140959e-08\\
-14.5	-1.78616256170469e-08\\
-14	-9.33977703650377e-09\\
-13.5	-8.30949996140848e-09\\
-13	-1.19662637126353e-08\\
-12.5	-5.09778060116274e-09\\
-12	-6.00801072856149e-10\\
-11.5	-4.93354322903344e-09\\
-11	-7.04427576918288e-10\\
-10.5	7.14042582998461e-09\\
-10	4.8071516638186e-09\\
-9.5	6.83052349724768e-09\\
-9	1.77341343092519e-08\\
-8.5	2.05181392777428e-08\\
-8	2.27239235127364e-08\\
-7.5	3.71639945687249e-08\\
-7	4.87775229548846e-08\\
-6.5	5.61038751080873e-08\\
-6	7.67221724071654e-08\\
-5.5	1.02168752252278e-07\\
-5	1.22668409683203e-07\\
-4.5	1.55668706554161e-07\\
-4	2.01972888851044e-07\\
-3.5	2.46423302696576e-07\\
-3	3.01072731414827e-07\\
-2.5	3.7520680717636e-07\\
-2	4.51533202678895e-07\\
-1.5	5.30984144431635e-07\\
-1	6.24843215055789e-07\\
-0.5	7.12609814102361e-07\\
0	7.74218727204144e-07\\
0.5	8.05044372648341e-07\\
1	7.65672286880473e-07\\
1.5	5.89361422701287e-07\\
2	2.12145415722301e-07\\
2.5	-4.74486991575845e-07\\
3	-1.64593844933711e-06\\
3.5	-3.50785563174583e-06\\
4	-6.31880441932103e-06\\
4.5	-1.04113362874164e-05\\
5	-1.61028305627059e-05\\
5.5	-2.3576534693398e-05\\
6	-3.27110640217509e-05\\
6.5	-4.27175923102263e-05\\
7	-5.17110214039995e-05\\
7.5	-5.65301281650332e-05\\
8	-5.31632042901631e-05\\
8.5	-3.83010716374921e-05\\
9	-1.20300611244351e-05\\
9.5	1.99506702325825e-05\\
10	4.72208920968553e-05\\
10.5	5.95855467508399e-05\\
11	5.31896683062527e-05\\
11.5	3.28197987814471e-05\\
12	8.3699568116935e-06\\
12.5	-1.12361099618985e-05\\
13	-2.18399996510521e-05\\
13.5	-2.39417890439686e-05\\
14	-2.04305167059715e-05\\
14.5	-1.44002516193897e-05\\
15	-8.02852119033415e-06\\
15.5	-2.39710234063717e-06\\
16	2.19690463253153e-06\\
16.5	5.81529632232708e-06\\
17	8.57682331553006e-06\\
17.5	1.05830777500718e-05\\
18	1.20112721112564e-05\\
18.5	1.31793977161052e-05\\
19	1.44200843436026e-05\\
19.5	1.58180098332808e-05\\
20	1.70792307567988e-05\\
20.5	1.77484113949326e-05\\
21	1.76467084090845e-05\\
21.5	1.7097793223821e-05\\
22	1.66283693722436e-05\\
22.5	1.63557095235543e-05\\
23	1.56903062425101e-05\\
23.5	1.37669249483921e-05\\
24	1.02854612820407e-05\\
24.5	5.93493135552545e-06\\
25	1.90207880832827e-06\\
25.5	-1.1219906480969e-06\\
26	-3.47555141228182e-06\\
26.5	-6.0692984711981e-06\\
27	-9.38519195406583e-06\\
27.5	-1.28856426536505e-05\\
28	-1.53338717998345e-05\\
28.5	-1.56591631242561e-05\\
29	-1.362385714584e-05\\
29.5	-9.88730684297109e-06\\
30	-5.58583575887616e-06\\
30.5	-1.75845729255633e-06\\
31	1.11971577629114e-06\\
31.5	3.29581162190343e-06\\
32	5.4358162280051e-06\\
32.5	7.96411515407704e-06\\
33	1.05196418985071e-05\\
33.5	1.20156951297831e-05\\
34	1.13069917823657e-05\\
34.5	7.96085820503339e-06\\
35	2.60002662565738e-06\\
35.5	-3.32171161174827e-06\\
36	-8.09416773871756e-06\\
36.5	-1.04046838930253e-05\\
37	-9.85626026083763e-06\\
37.5	-7.11115819994418e-06\\
38	-3.51000436044407e-06\\
38.5	-3.6105031875371e-07\\
39	1.63906077762705e-06\\
39.5	2.54005636050216e-06\\
40	2.87726779290449e-06\\
40.5	3.24771380105778e-06\\
41	3.92401893344737e-06\\
41.5	4.65555814167262e-06\\
42	4.77468259131696e-06\\
42.5	3.60200830870617e-06\\
43	9.32177697620035e-07\\
43.5	-2.71495806706506e-06\\
44	-6.22795119723191e-06\\
44.5	-8.31524552363782e-06\\
45	-8.00193078427597e-06\\
45.5	-5.06480342992061e-06\\
46	-2.40886933049273e-07\\
46.5	4.95713324012919e-06\\
47	8.78554671351912e-06\\
47.5	9.90394772984286e-06\\
48	7.83789925097397e-06\\
48.5	3.16471481609429e-06\\
49	-2.62952107576107e-06\\
49.5	-7.63915648398114e-06\\
50	-1.02133816079424e-05\\
50.5	-9.5149193367668e-06\\
51	-5.76742937602554e-06\\
51.5	-1.74386259875578e-07\\
52	5.46029228410379e-06\\
52.5	9.33653659163835e-06\\
53	1.02706125470895e-05\\
53.5	8.05342650061434e-06\\
54	3.45641818911187e-06\\
54.5	-2.05008956444843e-06\\
55	-6.77577420936358e-06\\
55.5	-9.3666763184365e-06\\
56	-9.23012664390787e-06\\
56.5	-6.64340404475424e-06\\
57	-2.54683072560065e-06\\
57.5	1.82921102307264e-06\\
58	5.35919961367664e-06\\
58.5	7.34175324696245e-06\\
59	7.63884371951142e-06\\
59.5	6.56148992628823e-06\\
60	4.58302570545745e-06\\
};
\addlegendentry{MC$_4(0.06)$}

\addplot [color=mycolor1, dashed, line width=1.5pt]
  table[row sep=crcr]{%
-60	2.56433574008302e-05\\
-59.5	1.44816272254625e-05\\
-59	-3.18647780855796e-06\\
-58.5	-2.09776300600851e-05\\
-58	-3.23216760840891e-05\\
-57.5	-3.38853418801697e-05\\
-57	-2.69211833379762e-05\\
-56.5	-1.56011951238526e-05\\
-56	-3.7978410034586e-06\\
-55.5	7.02145777233306e-06\\
-55	1.72706706493568e-05\\
-54.5	2.70748485240007e-05\\
-54	3.48169386291123e-05\\
-53.5	3.77287574650444e-05\\
-53	3.37621586373028e-05\\
-52.5	2.30188465580635e-05\\
-52	7.71558347451718e-06\\
-51.5	-9.02986472181148e-06\\
-51	-2.45090567176085e-05\\
-50.5	-3.69544452480756e-05\\
-50	-4.53054211350448e-05\\
-49.5	-4.88333273347926e-05\\
-49	-4.70670879379026e-05\\
-48.5	-3.99800160628795e-05\\
-48	-2.81881212105937e-05\\
-47.5	-1.29488211378549e-05\\
-47	4.08897994480181e-06\\
-46.5	2.12207318429999e-05\\
-46	3.69190390834258e-05\\
-45.5	4.99468024084875e-05\\
-45	5.94110673609397e-05\\
-44.5	6.47765790355021e-05\\
-44	6.58651737789657e-05\\
-43.5	6.28540861849002e-05\\
-43	5.62326791661554e-05\\
-42.5	4.67147019381919e-05\\
-42	3.51494981417773e-05\\
-41.5	2.2435759023129e-05\\
-41	9.43785058718212e-06\\
-40.5	-3.06956481406239e-06\\
-40	-1.4437765221085e-05\\
-39.5	-2.41752665460834e-05\\
-39	-3.19566558897651e-05\\
-38.5	-3.76145811717989e-05\\
-38	-4.11293712957288e-05\\
-37.5	-4.26050693282467e-05\\
-37	-4.22385743973995e-05\\
-36.5	-4.02948909852345e-05\\
-36	-3.70788351851236e-05\\
-35.5	-3.29080651163526e-05\\
-35	-2.80956079850243e-05\\
-34.5	-2.29325232692291e-05\\
-34	-1.76749732026745e-05\\
-33.5	-1.253958596259e-05\\
-33	-7.69873327357772e-06\\
-32.5	-3.28007772006034e-06\\
-32	6.29252298945534e-07\\
-31.5	3.98001255436667e-06\\
-31	6.75514311440894e-06\\
-30.5	8.96310279035165e-06\\
-30	1.06332660075058e-05\\
-29.5	1.18094523484727e-05\\
-29	1.25448769253038e-05\\
-28.5	1.28984618363861e-05\\
-28	1.29303674653184e-05\\
-27.5	1.2699685273047e-05\\
-27	1.22623117157273e-05\\
-26.5	1.16690734376388e-05\\
-26	1.09653731189941e-05\\
-25.5	1.01903604839844e-05\\
-25	9.37698383400909e-06\\
-24.5	8.55236155866344e-06\\
-24	7.73789942053101e-06\\
-23.5	6.95012171360404e-06\\
-23	6.20109701422261e-06\\
-22.5	5.49910695454073e-06\\
-22	4.84934821483866e-06\\
-21.5	4.25437175802184e-06\\
-21	3.71476425891382e-06\\
-20.5	3.22951041201515e-06\\
-20	2.79645542510423e-06\\
-19.5	2.41264833590283e-06\\
-19	2.07458872820397e-06\\
-18.5	1.77851047467805e-06\\
-18	1.52050128480917e-06\\
-17.5	1.29668548249597e-06\\
-17	1.10328711316387e-06\\
-16.5	9.3671378726885e-07\\
-16	7.93590473960812e-07\\
-15.5	6.7077387497214e-07\\
-15	5.65362539216076e-07\\
-14.5	4.74670479126165e-07\\
-14	3.9621355412116e-07\\
-13.5	3.27660273093728e-07\\
-13	2.66793990766445e-07\\
-12.5	2.11444076058706e-07\\
-12	1.59421349236027e-07\\
-11.5	1.08423371042025e-07\\
-11	5.59335062949326e-08\\
-10.5	-9.16950760524769e-10\\
-10	-6.55049953757761e-08\\
-9.5	-1.41925697191821e-07\\
-9	-2.35237270088695e-07\\
-8.5	-3.51787127908551e-07\\
-8	-4.9959601949428e-07\\
-7.5	-6.8886534060995e-07\\
-7	-9.32572859847555e-07\\
-6.5	-1.24724756056731e-06\\
-6	-1.65386661313995e-06\\
-5.5	-2.17899580910651e-06\\
-5	-2.85607670078562e-06\\
-4.5	-3.72699774454373e-06\\
-4	-4.84378761261075e-06\\
-3.5	-6.27052368004069e-06\\
-3	-8.08518784354705e-06\\
-2.5	-1.03813475468295e-05\\
-2	-1.32692010562697e-05\\
-1.5	-1.68752761261198e-05\\
-1	-2.13398857547634e-05\\
-0.5	-2.68103368493344e-05\\
0	-3.34278911571671e-05\\
0.5	-4.13040261338484e-05\\
1	-5.04813555347091e-05\\
1.5	-6.08707241138203e-05\\
2	-7.21547273876376e-05\\
2.5	-8.36441441502203e-05\\
3	-9.40719023657943e-05\\
3.5	-0.000101311799068138\\
4	-0.000102016749262482\\
4.5	-9.12020753774806e-05\\
5	-6.18559670201996e-05\\
5.5	-4.77482867489165e-06\\
6	9.09982270820564e-05\\
6.5	0.000236519549458164\\
7	0.000439282246863792\\
7.5	0.000696934218409717\\
8	0.000988828504350669\\
8.5	0.00126831932800436\\
9	0.00146210820578047\\
9.5	0.00148443799969722\\
10	0.00126862404768524\\
10.5	0.00080462452287472\\
11	0.000158595515896753\\
11.5	-0.000544482276218339\\
12	-0.0011684921664683\\
12.5	-0.00161480323243968\\
13	-0.00184534641047479\\
13.5	-0.00187592081119695\\
14	-0.00175344818072964\\
14.5	-0.00153304199768137\\
15	-0.00126317743184544\\
15.5	-0.000979638627691584\\
16	-0.000705365158404928\\
16.5	-0.000453040914214067\\
17	-0.000228299680165793\\
17.5	-3.2513881575838e-05\\
18	0.000135151344002452\\
18.5	0.000276394015196309\\
19	0.000392988879761798\\
19.5	0.000486390892095073\\
20	0.000557602665079057\\
20.5	0.000607211178397142\\
21	0.00063552082862323\\
21.5	0.000642746158237388\\
22	0.000629235001348323\\
22.5	0.000595685493322622\\
23	0.00054334406246707\\
23.5	0.000474172922401587\\
24	0.000390954259512186\\
24.5	0.000297319043738596\\
25	0.000197695393261457\\
25.5	9.71443826813915e-05\\
26	1.08335305614647e-06\\
26.5	-8.50602798357639e-05\\
27	-0.000156245777828205\\
27.5	-0.000208245769015565\\
28	-0.000238056313276106\\
28.5	-0.000244282628788615\\
29	-0.000227481038796978\\
29.5	-0.000190324126503473\\
30	-0.00013744939304129\\
30.5	-7.49871905545221e-05\\
31	-9.87340670492686e-06\\
31.5	5.08707849304674e-05\\
32	0.000100713056220183\\
32.5	0.000134101807780614\\
33	0.000147169674702789\\
33.5	0.000138706196263388\\
34	0.00011093543323751\\
34.5	6.94342765947686e-05\\
35	2.18844348575515e-05\\
35.5	-2.39105429760534e-05\\
36	-6.16651702245523e-05\\
36.5	-8.70318484471183e-05\\
37	-9.71022677492352e-05\\
37.5	-9.02334384943161e-05\\
38	-6.69438018451009e-05\\
38.5	-3.12045050641921e-05\\
39	9.33626491007968e-06\\
39.5	4.50660650500317e-05\\
40	6.75021936542687e-05\\
40.5	7.21182236031706e-05\\
41	5.94555840489395e-05\\
41.5	3.43706145884019e-05\\
42	4.27221476844924e-06\\
42.5	-2.29357052509677e-05\\
43	-4.09360948467131e-05\\
43.5	-4.66154816029205e-05\\
44	-4.04675619494339e-05\\
44.5	-2.56878204354828e-05\\
45	-6.78644122911673e-06\\
45.5	1.15036919519276e-05\\
46	2.50225195759173e-05\\
46.5	3.09547268435317e-05\\
47	2.84500200220467e-05\\
47.5	1.87637970345901e-05\\
48	4.87720832115749e-06\\
48.5	-9.23460940551388e-06\\
49	-1.96013337177325e-05\\
49.5	-2.34091375263575e-05\\
50	-1.99293539312556e-05\\
50.5	-1.07709552444458e-05\\
51	6.94270989968105e-07\\
51.5	1.0523126466855e-05\\
52	1.56233383906877e-05\\
52.5	1.49137226248261e-05\\
53	9.66705316672623e-06\\
53.5	2.75590788506645e-06\\
54	-2.82562605306844e-06\\
54.5	-5.32419875498495e-06\\
55	-4.83405928063981e-06\\
55.5	-3.04551790264784e-06\\
56	-2.28257100928835e-06\\
56.5	-4.0570291782429e-06\\
57	-7.78858223433644e-06\\
57.5	-1.07408534905624e-05\\
58	-9.60953888324655e-06\\
58.5	-2.76902336647584e-06\\
59	8.44188161721047e-06\\
59.5	2.00876316390996e-05\\
60	2.7167621029929e-05\\
};
\addlegendentry{FD$_4$}

\end{axis}

\end{tikzpicture}%

%% file: 2sol.tex
\begin{tikzpicture}

\begin{axis}[%
width=3.5in,
height=2.1in,
at={(1in,1in)},
scale only axis,
xmin=-150,
xmax=150,
xlabel style={font=\color{white!15!black}},
xlabel={$x$},
xtick={-150,-120,-90,-60,-30,0,30,60,90,120,150},
xticklabels={$-150$,$-120$,$-90$,$-60$,$-30$,$0$,$30$,$60$,$90$,$120$,$150$},
ymin=-0.05,
ymax=0.35,
ylabel style={font=\color{white!15!black}},
ylabel={$-u$},
axis background/.style={fill=white},
legend style={legend cell align=left, align=left, draw=white!15!black}
]
\addplot [color=black, line width=2.0pt]
  table[row sep=crcr]{%
-150	-1.1877006855763e-10\\
-149.5	-1.78441586538433e-10\\
-149	-1.36520199715553e-10\\
-148.5	2.42199433256465e-11\\
-148	2.64575529402646e-10\\
-147.5	4.76435055857926e-10\\
-147	5.04430663859706e-10\\
-146.5	2.13927172106598e-10\\
-146	-4.08144363758729e-10\\
-145.5	-1.16644159250574e-09\\
-145	-1.65287141419562e-09\\
-144.5	-1.3807217743253e-09\\
-144	-5.41589607464629e-11\\
-143.5	2.12310935185839e-09\\
-143	4.28012511782619e-09\\
-142.5	5.04468104464862e-09\\
-142	3.12766389058853e-09\\
-141.5	-1.81464261199206e-09\\
-141	-8.43843056122197e-09\\
-140.5	-1.36274635396188e-08\\
-140	-1.34687289651904e-08\\
-139.5	-5.23048031681285e-09\\
-139	1.03746495003555e-08\\
-138.5	2.78929059862144e-08\\
-138	3.80603841534436e-08\\
-137.5	3.14544645461621e-08\\
-137	3.93224064459391e-09\\
-136.5	-3.85732476539972e-08\\
-136	-7.9004877655577e-08\\
-135.5	-9.40859073855757e-08\\
-135	-6.47984684021538e-08\\
-134.5	1.13290627384357e-08\\
-134	1.12380047706456e-07\\
-133.5	1.9470237722517e-07\\
-133	2.08364287176978e-07\\
-132.5	1.2138818960487e-07\\
-132	-5.74241798779722e-08\\
-131.5	-2.69716724667669e-07\\
-131	-4.21760187740658e-07\\
-130.5	-4.20872372904837e-07\\
-130	-2.21519006949141e-07\\
-129.5	1.39695425642812e-07\\
-129	5.40316063012292e-07\\
-128.5	8.08662103469947e-07\\
-128	7.91847794350688e-07\\
-127.5	4.29275538816114e-07\\
-127	-2.01581081052033e-07\\
-126.5	-8.92729846543208e-07\\
-126	-1.37145984987082e-06\\
-125.5	-1.40376036937042e-06\\
-125	-8.95825218293209e-07\\
-124.5	4.8972015681007e-08\\
-124	1.14513996482851e-06\\
-123.5	2.01217324077557e-06\\
-123	2.30515311405474e-06\\
-122.5	1.84244751806674e-06\\
-122	6.83955820015097e-07\\
-121.5	-8.69536511462469e-07\\
-121	-2.35799667116637e-06\\
-120.5	-3.29961059472474e-06\\
-120	-3.34097767710088e-06\\
-119.5	-2.37349570252464e-06\\
-119	-5.80407842308411e-07\\
-118.5	1.59996724990383e-06\\
-118	3.58503345258069e-06\\
-117.5	4.80246720933252e-06\\
-117	4.84682901247584e-06\\
-116.5	3.5966109023032e-06\\
-116	1.26087059960957e-06\\
-115.5	-1.65727272539702e-06\\
-115	-4.47231909257073e-06\\
-114.5	-6.4769566396706e-06\\
-114	-7.11333230737121e-06\\
-113.5	-6.11397023191847e-06\\
-113	-3.57818560922003e-06\\
-112.5	3.49342771867181e-08\\
-112	3.99602307533557e-06\\
-111.5	7.45706285353984e-06\\
-111	9.63105016253701e-06\\
-110.5	9.96069159639847e-06\\
-110	8.24128872642807e-06\\
-109.5	4.67388505430152e-06\\
-109	-1.60951967716763e-07\\
-108.5	-5.40350912743702e-06\\
-108	-1.00708335214522e-05\\
-107.5	-1.32391492127069e-05\\
-107	-1.42183496093075e-05\\
-106.5	-1.26854378895288e-05\\
-106	-8.7521179382989e-06\\
-105.5	-2.9550889228244e-06\\
-105	3.8276588987479e-06\\
-104.5	1.05185346746157e-05\\
-104	1.60151310976561e-05\\
-103.5	1.9368533230182e-05\\
-103	1.99353282289309e-05\\
-102.5	1.74783924538279e-05\\
-102	1.22018382238401e-05\\
-101.5	4.71926464366022e-06\\
-101	-4.032759953898e-06\\
-100.5	-1.29169809509129e-05\\
-100	-2.07349204336559e-05\\
-99.5	-2.63701970405026e-05\\
-99	-2.89271512936695e-05\\
-98.5	-2.78562081496713e-05\\
-98	-2.30522347060269e-05\\
-97.5	-1.49056343894554e-05\\
-97	-4.28402121863059e-06\\
-96.5	7.56974721297118e-06\\
-96	1.92180034886806e-05\\
-95.5	2.92554160053183e-05\\
-95	3.65280195185894e-05\\
-94.5	4.02834563462672e-05\\
-94	4.02121630429718e-05\\
-93.5	3.63820044791188e-05\\
-93	2.91150831408647e-05\\
-92.5	1.88795461955703e-05\\
-92	6.25432068763482e-06\\
-91.5	-8.02465020381799e-06\\
-91	-2.29857194822499e-05\\
-90.5	-3.73912737853553e-05\\
-90	-4.98210878458531e-05\\
-89.5	-5.88763066848288e-05\\
-89	-6.34415464789067e-05\\
-88.5	-6.29025520506933e-05\\
-88	-5.72324776688013e-05\\
-87.5	-4.69247687509374e-05\\
-87	-3.28266857678173e-05\\
-86.5	-1.59672951099307e-05\\
-86	2.54781500063431e-06\\
-85.5	2.15693281868323e-05\\
-85	3.99216225452879e-05\\
-84.5	5.64744513693482e-05\\
-84	7.02902546163891e-05\\
-83.5	8.07944415566643e-05\\
-83	8.78720053281119e-05\\
-82.5	9.18159413553674e-05\\
-82	9.31305456787171e-05\\
-81.5	9.22752986932232e-05\\
-81	8.94680429540214e-05\\
-80.5	8.46295321476723e-05\\
-80	7.74732676466595e-05\\
-79.5	6.76788079399669e-05\\
-79	5.50719333741516e-05\\
-78.5	3.97664092639391e-05\\
-78	2.22601101634964e-05\\
-77.5	3.48620628140402e-06\\
-77	-1.52026380669508e-05\\
-76.5	-3.21579659395825e-05\\
-76	-4.56876836796599e-05\\
-75.5	-5.43975623466219e-05\\
-75	-5.75230879150547e-05\\
-74.5	-5.51385564553355e-05\\
-74	-4.81754455367608e-05\\
-73.5	-3.82559208991278e-05\\
-73	-2.74016709870495e-05\\
-72.5	-1.76875905548796e-05\\
-72	-1.08894656410436e-05\\
-71.5	-8.16099262546475e-06\\
-71	-9.78694611744349e-06\\
-70.5	-1.50779318348689e-05\\
-70	-2.24621685716008e-05\\
-69.5	-2.97769477692111e-05\\
-69	-3.46930865889316e-05\\
-68.5	-3.51647978205069e-05\\
-68	-2.98076323362035e-05\\
-67.5	-1.81500593926833e-05\\
-67	-7.41735742117074e-07\\
-66.5	2.08859489055708e-05\\
-66	4.44095106973968e-05\\
-65.5	6.70378809897102e-05\\
-65	8.59184053018602e-05\\
-64.5	9.85389040895975e-05\\
-64	0.0001030449955354\\
-63.5	9.84345381876041e-05\\
-63	8.46398178922054e-05\\
-62.5	6.25288071818922e-05\\
-62	3.38457602547927e-05\\
-61.5	1.09427536850612e-06\\
-61	-3.26342056053403e-05\\
-60.5	-6.38715950853317e-05\\
-60	-8.90080075949092e-05\\
-59.5	-0.000104552717585326\\
-59	-0.000107442684360602\\
-58.5	-9.54126486719536e-05\\
-58	-6.74035488049917e-05\\
-57.5	-2.39446820309208e-05\\
-57	3.25721187869367e-05\\
-56.5	9.77977247112874e-05\\
-56	0.000165562450041379\\
-55.5	0.00022828031817673\\
-55	0.000277634122794666\\
-54.5	0.000305508694842885\\
-54	0.000305075287500516\\
-53.5	0.000271875182173534\\
-53	0.000204731757136937\\
-52.5	0.000106338393737581\\
-52	-1.65926769146365e-05\\
-51.5	-0.000153688884345834\\
-51	-0.000291790436208116\\
-50.5	-0.000416265414438122\\
-50	-0.000512611552569917\\
-49.5	-0.000568133340177547\\
-49	-0.000573470448126557\\
-48.5	-0.000523776494377651\\
-48	-0.000419389576710825\\
-47.5	-0.000265893207600668\\
-47	-7.35427570836873e-05\\
-46.5	0.000143877531718663\\
-46	0.000370620790754233\\
-45.5	0.000590770525898232\\
-45	0.000790047955744909\\
-44.5	0.000957429961339764\\
-44	0.00108640184403183\\
-43.5	0.0011757202516255\\
-43	0.0012296352014175\\
-42.5	0.00125760014932885\\
-42	0.0012735647604238\\
-41.5	0.00129498795862197\\
-41	0.00134173349734063\\
-40.5	0.00143502002006243\\
-40	0.00159658797609525\\
-39.5	0.00184821481932837\\
-39	0.00221166741131388\\
-38.5	0.00270914086714741\\
-38	0.00336420178731775\\
-37.5	0.00420322739491203\\
-37	0.00525730928788025\\
-36.5	0.00656457624479608\\
-36	0.00817288658523641\\
-35.5	0.0101428385619194\\
-35	0.0125510344379917\\
-34.5	0.0154935014378554\\
-34	0.0190891124393116\\
-33.5	0.0234827444729313\\
-33	0.0288477387758058\\
-32.5	0.0353869585017643\\
-32	0.0433313647480362\\
-31.5	0.0529345495216037\\
-31	0.0644611120941086\\
-30.5	0.078166257271036\\
-30	0.0942637795335473\\
-29.5	0.112880107986795\\
-29	0.133993940188461\\
-28.5	0.157364865861379\\
-28	0.182460645707215\\
-27.5	0.208400845872598\\
-27	0.233941888628096\\
-26.5	0.257530647878347\\
-26	0.277444919973009\\
-25.5	0.292016395703432\\
-25	0.299899288208998\\
-24.5	0.300318831728892\\
-24	0.293225909739882\\
-23.5	0.279307206880882\\
-23	0.25984746107879\\
-22.5	0.236489264528199\\
-22	0.210962944860889\\
-21.5	0.184854883338419\\
-21	0.159455456009208\\
-20.5	0.135694839842703\\
-20	0.114150511024253\\
-19.5	0.0950997152016292\\
-19	0.0785911507114413\\
-18.5	0.064517083986599\\
-18	0.0526752546319329\\
-17.5	0.042816496114215\\
-17	0.0346781821206598\\
-16.5	0.0280056926436221\\
-16	0.0225647453323315\\
-15.5	0.0181472951532291\\
-15	0.0145732219677977\\
-14.5	0.0116894673955882\\
-14	0.00936778193612228\\
-13.5	0.00750184780481367\\
-13	0.00600425214615002\\
-12.5	0.00480358346872159\\
-12	0.00384179234489006\\
-11.5	0.00307187482391782\\
-11	0.00245588707421954\\
-10.5	0.00196327264783382\\
-10	0.00156947107578136\\
-9.5	0.001254771418465\\
-9	0.0010033742523545\\
-8.5	0.000802628570175132\\
-8	0.000642413767438829\\
-7.5	0.000514640737553337\\
-7	0.000412850286013226\\
-6.5	0.000331890752536291\\
-6	0.000267659728397582\\
-5.5	0.000216897587975584\\
-5	0.00017702293187292\\
-4.5	0.000146001835323518\\
-4	0.000122244424523243\\
-3.5	0.000104523666402155\\
-3	9.19122560847185e-05\\
-2.5	8.37343888540713e-05\\
-2	7.9529965239436e-05\\
-1.5	7.90293432228999e-05\\
-1	8.21372795082777e-05\\
-0.5	8.89251599551755e-05\\
0	9.96309913803947e-05\\
0.5	0.000114667001150853\\
1	0.000134635042823635\\
1.5	0.000160350345370401\\
2	0.000192874516762773\\
2.5	0.000233559108310227\\
3	0.000284101485194197\\
3.5	0.000346615260573106\\
4	0.00042371813717283\\
4.5	0.00051864068330803\\
5	0.000635360380408122\\
5.5	0.000778766222724389\\
6	0.000954860238775997\\
6.5	0.00117100357043674\\
7	0.00143621617490201\\
7.5	0.0017615407689983\\
8	0.00216048329584383\\
8.5	0.00264954384014369\\
9	0.00324885331165478\\
9.5	0.0039829320715633\\
10	0.00488158647109333\\
10.5	0.00598095709412403\\
11	0.00732472712081341\\
11.5	0.00896548885852059\\
12	0.010966248340637\\
12.5	0.0134020182491007\\
13	0.0163614034744732\\
13.5	0.0199480151126576\\
14	0.0242814503900965\\
14.5	0.0294974412339403\\
15	0.0357465990226464\\
15.5	0.0431909711251595\\
16	0.0519973964007091\\
16.5	0.0623264511087913\\
17	0.0743157071642883\\
17.5	0.0880562339325163\\
18	0.103561975895461\\
18.5	0.120733070677555\\
19	0.139316502000696\\
19.5	0.158870621347699\\
20	0.178743427931546\\
20.5	0.198076778625661\\
21	0.215848020314579\\
21.5	0.230955079812086\\
22	0.24234033567146\\
22.5	0.249134830528053\\
23	0.250792684807742\\
23.5	0.247182139489869\\
24	0.238608172556774\\
24.5	0.22576020721555\\
25	0.209599728759833\\
25.5	0.191217767679983\\
26	0.171695770564234\\
26.5	0.151995937653167\\
27	0.132893821214308\\
27.5	0.114952979950027\\
28	0.0985328046607036\\
28.5	0.0838172669492239\\
29	0.0708531010640283\\
29.5	0.0595888408891514\\
30	0.0499095110657432\\
30.5	0.0416646161509107\\
31	0.0346890258293532\\
31.5	0.0288174600819781\\
32	0.0238937501012979\\
32.5	0.0197761252519527\\
33	0.0163396436096458\\
33.5	0.0134766712103296\\
34	0.0110960971668768\\
34.5	0.00912178399760061\\
35	0.00749060624484647\\
35.5	0.00615032404220342\\
36	0.00505746489335554\\
36.5	0.00417533851655393\\
37	0.00347227783348465\\
37.5	0.00292017601270422\\
38	0.00249336882156904\\
38.5	0.00216789015114016\\
39	0.0019211050565228\\
39.5	0.00173169819993316\\
40	0.00157996647750036\\
40.5	0.0014483356150771\\
41	0.00132199653319757\\
41.5	0.001189542831532\\
42	0.00104348835552166\\
42.5	0.000880554897700208\\
43	0.000701645722128586\\
43.5	0.000511460075575113\\
44	0.000317752748297521\\
44.5	0.000130293284071335\\
45	-4.03764007707196e-05\\
45.5	-0.000184254923723619\\
46	-0.00029311710373584\\
46.5	-0.000361527299031871\\
47	-0.000387509443182492\\
47.5	-0.000372795680470262\\
48	-0.00032262978398067\\
48.5	-0.000245158292252709\\
49	-0.000150493993065195\\
49.5	-4.95775097810008e-05\\
50	4.70122148538097e-05\\
50.5	0.000130144847914268\\
51	0.000192990811532342\\
51.5	0.000231547011099378\\
52	0.000244794520774588\\
52.5	0.000234497846460937\\
53	0.000204702022573121\\
53.5	0.000161020231847538\\
54	0.00010982518771819\\
54.5	5.74581366614485e-05\\
55	9.55050139647223e-06\\
55.5	-2.9478023779278e-05\\
56	-5.67144536380601e-05\\
56.5	-7.08395906082737e-05\\
57	-7.20316095066861e-05\\
57.5	-6.17307649560389e-05\\
58	-4.23199416379534e-05\\
58.5	-1.67727088702461e-05\\
59	1.16920276270449e-05\\
59.5	3.99243416780582e-05\\
60	6.51129501095805e-05\\
60.5	8.50015141247375e-05\\
61	9.80441366143068e-05\\
61.5	0.000103494563006583\\
62	0.000101427184895404\\
62.5	9.26894183980565e-05\\
63	7.87830943676976e-05\\
63.5	6.16735928389839e-05\\
64	4.35365296089612e-05\\
64.5	2.64698393746133e-05\\
65	1.22131625910701e-05\\
65.5	1.91911434376934e-06\\
66	-3.98673070805808e-06\\
66.5	-5.83402857775581e-06\\
67	-4.617308127437e-06\\
67.5	-1.79734697521281e-06\\
68	9.78444243447308e-07\\
68.5	2.19207724500328e-06\\
69	7.4869007375849e-07\\
69.5	-3.81269287211754e-06\\
70	-1.12339652046114e-05\\
70.5	-2.05889951616492e-05\\
71	-3.04518968808686e-05\\
71.5	-3.91466724433656e-05\\
72	-4.50353759328362e-05\\
72.5	-4.67971423687603e-05\\
73	-4.36480758256053e-05\\
73.5	-3.54590192908687e-05\\
74	-2.27481411971277e-05\\
74.5	-6.55397010427661e-06\\
75	1.17761618486938e-05\\
75.5	3.08265795342986e-05\\
76	4.93236411909607e-05\\
76.5	6.62739485298605e-05\\
77	8.101619032796e-05\\
77.5	9.31952032799177e-05\\
78	0.000102676624108303\\
78.5	0.000109424396307016\\
79	0.000113371677873339\\
79.5	0.000114325999653205\\
80	0.000111947964640826\\
80.5	0.000105820873816917\\
81	9.55944757683564e-05\\
81.5	8.11571456712811e-05\\
82	6.2780017902093e-05\\
82.5	4.11863733800265e-05\\
83	1.75252731625424e-05\\
83.5	-6.73935519896594e-06\\
84	-2.99881556829456e-05\\
84.5	-5.06134973876903e-05\\
85	-6.71689907000969e-05\\
85.5	-7.8496356572949e-05\\
86	-8.38320258567312e-05\\
86.5	-8.2887357575809e-05\\
87	-7.58859724913913e-05\\
87.5	-6.35441573302778e-05\\
88	-4.69956742075273e-05\\
88.5	-2.76780337975859e-05\\
89	-7.20277634440693e-06\\
89.5	1.27721892237556e-05\\
90	3.06583810677628e-05\\
90.5	4.5039119941576e-05\\
91	5.47641894233065e-05\\
91.5	5.90451859914422e-05\\
92	5.75524265418022e-05\\
92.5	5.04970379029945e-05\\
93	3.8666439449162e-05\\
93.5	2.33823118424244e-05\\
94	6.37053729193453e-06\\
94.5	-1.04398357595006e-05\\
95	-2.51534421040868e-05\\
95.5	-3.61436448623691e-05\\
96	-4.22545369671315e-05\\
96.5	-4.2944784153279e-05\\
97	-3.83530155453624e-05\\
97.5	-2.92733519558166e-05\\
98	-1.70446112410591e-05\\
98.5	-3.37166618038469e-06\\
99	9.89546760684428e-06\\
99.5	2.10002701552591e-05\\
100	2.85130864836275e-05\\
100.5	3.15321457508776e-05\\
101	2.98189535357515e-05\\
101.5	2.3836793880918e-05\\
102	1.46772893326089e-05\\
102.5	3.88488687387102e-06\\
103	-6.78959999797813e-06\\
103.5	-1.56656813021493e-05\\
104	-2.1399333047925e-05\\
104.5	-2.32000146319619e-05\\
105	-2.09617893938046e-05\\
105.5	-1.52784035827699e-05\\
106	-7.33021578095176e-06\\
106.5	1.34078471589045e-06\\
107	9.12351003183606e-06\\
107.5	1.46379640366843e-05\\
108	1.6985047062907e-05\\
108.5	1.59021523508043e-05\\
109	1.17988583840346e-05\\
109.5	5.66497918826505e-06\\
110	-1.13714542554341e-06\\
110.5	-7.16087641236914e-06\\
111	-1.11856685185785e-05\\
111.5	-1.24785277971086e-05\\
112	-1.09423981373396e-05\\
112.5	-7.11499236074667e-06\\
113	-2.02073780541416e-06\\
113.5	3.08473182549427e-06\\
114	7.01453339879681e-06\\
114.5	8.92512201224697e-06\\
115	8.50357668498347e-06\\
115.5	6.02479207248636e-06\\
116	2.26368437518021e-06\\
116.5	-1.71871480946853e-06\\
117	-4.86281897872235e-06\\
117.5	-6.39550662933236e-06\\
118	-6.0324096235698e-06\\
118.5	-4.03780563952991e-06\\
119	-1.12350882341502e-06\\
119.5	1.78253044689045e-06\\
120	3.8258275644338e-06\\
120.5	4.48184921601991e-06\\
121	3.69825590445259e-06\\
121.5	1.87915806017239e-06\\
122	-2.79284984724417e-07\\
122.5	-2.04158753870154e-06\\
123	-2.8831908410039e-06\\
123.5	-2.64988512970709e-06\\
124	-1.57374720727848e-06\\
124.5	-1.48758345629852e-07\\
125	1.07715656474358e-06\\
125.5	1.70778145736083e-06\\
126	1.62318848913059e-06\\
126.5	9.83576130716708e-07\\
127	1.25289479519685e-07\\
127.5	-5.95675147455378e-07\\
128	-9.42789081377304e-07\\
128.5	-8.67599150250725e-07\\
129	-4.91804662835253e-07\\
129.5	-2.53023263876574e-08\\
130	3.34511165253011e-07\\
130.5	4.77377961607424e-07\\
131	4.03568413028788e-07\\
131.5	1.9747294071813e-07\\
132	-2.65434242700138e-08\\
132.5	-1.77224010378982e-07\\
133	-2.16777783101206e-07\\
133.5	-1.61640779653136e-07\\
134	-6.09602448963675e-08\\
134.5	3.23659083081672e-08\\
135	8.40486618950201e-08\\
135.5	8.67303864782542e-08\\
136	5.48369526988723e-08\\
136.5	1.20507046598491e-08\\
137	-2.11733499338234e-08\\
137.5	-3.48661006188867e-08\\
138	-3.01853736834748e-08\\
138.5	-1.52748502096193e-08\\
139	3.87659965927088e-10\\
139.5	1.03005993595707e-08\\
140	1.25591184547135e-08\\
140.5	9.02262045495092e-09\\
141	3.22108988148824e-09\\
141.5	-1.68612599598376e-09\\
142	-4.06611601437037e-09\\
142.5	-3.91804379194738e-09\\
143	-2.26586345565235e-09\\
143.5	-3.535677978099e-10\\
144	9.47737808417334e-10\\
144.5	1.35316649304123e-09\\
145	1.05318492395293e-09\\
145.5	4.49558836647411e-10\\
146	-9.21425665308425e-11\\
146.5	-3.7690599015663e-10\\
147	-3.908515531215e-10\\
147.5	-2.33983344151307e-10\\
148	-3.52185429291744e-11\\
148.5	1.07972247045076e-10\\
149	1.52117846132329e-10\\
149.5	1.01988212636503e-10\\
150	-5.98941723407052e-12\\
};

\end{axis}

\end{tikzpicture}%

%% file: 2sol2.tex
\definecolor{mycolor1}{rgb}{0.49020,0.18039,0.56078}%
\definecolor{mycolor2}{rgb}{0.03137,0.43137,0.13725}%
\begin{tikzpicture}

\begin{axis}[%
width=3.5in,
height=2.1in,
at={(1in,1in)},
scale only axis,
xmin=-27,
xmax=-22.5,
xlabel style={font=\color{white!15!black}},
xlabel={$x$},
ymin=0.27,
ymax=0.31,
ylabel style={font=\color{white!15!black}},
ylabel={$-u$},
axis background/.style={fill=white},
legend style={legend cell align=left, align=left, draw=white!15!black}
]

\addplot [color=blue, line width=1.0pt, mark size=3.0pt, mark=triangle*, mark repeat=8, mark phase=5, mark options={solid, rotate=180, blue}]
  table[row sep=crcr]{%
-28	0.188195626137131\\
-27.9375	0.191371276172775\\
-27.875	0.194552449235616\\
-27.8125	0.197736648721148\\
-27.75	0.200921378024866\\
-27.6875	0.204104140542263\\
-27.625	0.207282439668834\\
-27.5625	0.210453778800074\\
-27.5	0.213615661331476\\
-27.4375	0.216765504126087\\
-27.375	0.21990037791716\\
-27.3125	0.2230172669055\\
-27.25	0.226113155291912\\
-27.1875	0.229185027277201\\
-27.125	0.232229867062172\\
-27.0625	0.235244658847629\\
-27	0.238226386834378\\
-26.9375	0.241171998068406\\
-26.875	0.244078290976422\\
-26.8125	0.246942026830318\\
-26.75	0.249759966901987\\
-26.6875	0.252528872463321\\
-26.625	0.255245504786212\\
-26.5625	0.257906625142551\\
-26.5	0.26050899480423\\
-26.4375	0.263049410329936\\
-26.375	0.265524809425525\\
-26.3125	0.26793216508365\\
-26.25	0.270268450296963\\
-26.1875	0.272530638058116\\
-26.125	0.274715701359759\\
-26.0625	0.276820613194545\\
-26	0.278842346555125\\
-25.9375	0.280777991988577\\
-25.875	0.282625110259679\\
-25.8125	0.284381379687631\\
-25.75	0.286044478591639\\
-25.6875	0.287612085290905\\
-25.625	0.28908187810463\\
-25.5625	0.29045153535202\\
-25.5	0.291718735352275\\
-25.4375	0.292881345263264\\
-25.375	0.293937987597515\\
-25.3125	0.294887473706219\\
-25.25	0.295728614940569\\
-25.1875	0.296460222651758\\
-25.125	0.297081108190978\\
-25.0625	0.29759008290942\\
-25	0.297985958158278\\
-24.9375	0.298267773132266\\
-24.875	0.29843547840019\\
-24.8125	0.298489252374378\\
-24.75	0.298429273467157\\
-24.6875	0.298255720090857\\
-24.625	0.297968770657804\\
-24.5625	0.297568603580328\\
-24.5	0.297055397270755\\
-24.4375	0.296429552048378\\
-24.375	0.295692355860349\\
-24.3125	0.294845318560779\\
-24.25	0.293889950003785\\
-24.1875	0.29282776004348\\
-24.125	0.291660258533979\\
-24.0625	0.290388955329397\\
-24	0.289015360283846\\
-23.9375	0.287541156248127\\
-23.875	0.285968718059771\\
-23.8125	0.284300593552997\\
-23.75	0.282539330562022\\
-23.6875	0.280687476921063\\
-23.625	0.278747580464337\\
-23.5625	0.276722189026063\\
-23.5	0.274613850440456\\
-23.4375	0.272425209301171\\
-23.375	0.270159297239607\\
-23.3125	0.267819242646597\\
-23.25	0.265408173912976\\
-23.1875	0.26292921942958\\
-23.125	0.260385507587242\\
-23.0625	0.257780166776797\\
-23	0.255116325389081\\
-22.9375	0.252397127053805\\
-22.875	0.249625776356192\\
-22.8125	0.246805493120344\\
-22.75	0.243939497170362\\
-22.6875	0.241031008330346\\
-22.625	0.238083246424397\\
-22.5625	0.235099431276617\\
-22.5	0.232082782711107\\
-22.4375	0.229036468365671\\
-22.375	0.22596344713293\\
-22.3125	0.222866625719212\\
-22.25	0.21974891083084\\
-22.1875	0.21661320917414\\
-22.125	0.213462427455438\\
-22.0625	0.210299472381059\\
-22	0.207127250657329\\
-21.9375	0.203948574223716\\
-21.875	0.20076587595226\\
-21.8125	0.197581493948146\\
-21.75	0.194397766316557\\
-21.6875	0.191217031162676\\
-21.625	0.188041626591688\\
-21.5625	0.184873890708775\\
-21.5	0.181716161619121\\
-21.4375	0.178570666144105\\
-21.375	0.175439185969881\\
-21.3125	0.172323391498798\\
-21.25	0.169224953133205\\
-21.1875	0.166145541275451\\
-21.125	0.163086826327886\\
-21.0625	0.160050478692858\\
-21	0.157038168772717\\
};

\addplot [color=mycolor2, line width=1.0pt, mark size=3.0pt, mark=triangle*, mark repeat=8, mark phase=4, mark options={solid, mycolor2}]
  table[row sep=crcr]{%
-28	0.179675738401624\\
-27.9375	0.182947452081795\\
-27.875	0.186235270798699\\
-27.8125	0.189536735976433\\
-27.75	0.192849389039094\\
-27.6875	0.19617077141078\\
-27.625	0.19949842451559\\
-27.5625	0.202829889777619\\
-27.5	0.206162708620967\\
-27.4375	0.209494298139229\\
-27.375	0.212821578104001\\
-27.3125	0.216141343956378\\
-27.25	0.219450391137453\\
-27.1875	0.222745515088321\\
-27.125	0.226023511250075\\
-27.0625	0.229281175063811\\
-27	0.232515301970623\\
-26.9375	0.235722607291502\\
-26.875	0.238899485867033\\
-26.8125	0.242042252417699\\
-26.75	0.245147221663983\\
-26.6875	0.248210708326366\\
-26.625	0.251229027125331\\
-26.5625	0.25419849278136\\
-26.5	0.257115420014937\\
-26.4375	0.259976123422553\\
-26.375	0.262776917104741\\
-26.3125	0.265514115038044\\
-26.25	0.268184031199003\\
-26.1875	0.270782979564161\\
-26.125	0.273307274110061\\
-26.0625	0.275753228813244\\
-26	0.278117157650254\\
-25.9375	0.280395478656102\\
-25.875	0.282585026099675\\
-25.8125	0.28468273830833\\
-25.75	0.286685553609423\\
-25.6875	0.288590410330311\\
-25.625	0.290394246798351\\
-25.5625	0.292094001340899\\
-25.5	0.293686612285312\\
-25.4375	0.29516922400841\\
-25.375	0.296539805084873\\
-25.3125	0.297796530138841\\
-25.25	0.298937573794459\\
-25.1875	0.299961110675867\\
-25.125	0.30086531540721\\
-25.0625	0.301648362612628\\
-25	0.302308426916265\\
-24.9375	0.302843955778832\\
-24.875	0.303254488007315\\
-24.8125	0.30353983524527\\
-24.75	0.303699809136252\\
-24.6875	0.303734221323817\\
-24.625	0.303642883451519\\
-24.5625	0.303425607162915\\
-24.5	0.303082204101559\\
-24.4375	0.302612765882021\\
-24.375	0.302018504002926\\
-24.3125	0.301300909933913\\
-24.25	0.300461475144621\\
-24.1875	0.299501691104689\\
-24.125	0.298423049283756\\
-24.0625	0.297227041151461\\
-24	0.295915158177444\\
-23.9375	0.294489116561129\\
-23.875	0.292951531421088\\
-23.8125	0.291305242605678\\
-23.75	0.289553089963258\\
-23.6875	0.287697913342185\\
-23.625	0.285742552590816\\
-23.5625	0.283689847557509\\
-23.5	0.281542638090622\\
-23.4375	0.279303891694541\\
-23.375	0.276977086497765\\
-23.3125	0.274565828284819\\
-23.25	0.272073722840232\\
-23.1875	0.26950437594853\\
-23.125	0.26686139339424\\
-23.0625	0.264148380961889\\
-23	0.261368944436004\\
-22.9375	0.258526710895062\\
-22.875	0.255625392593343\\
-22.8125	0.252668723079078\\
-22.75	0.249660435900497\\
-22.6875	0.246604264605829\\
-22.625	0.243503942743306\\
-22.5625	0.240363203861157\\
-22.5	0.237185781507614\\
-22.4375	0.233975343728501\\
-22.375	0.230735296560028\\
-22.3125	0.227468980535996\\
-22.25	0.224179736190211\\
-22.1875	0.220870904056474\\
-22.125	0.217545824668591\\
-22.0625	0.214207838560364\\
-22	0.210860286265596\\
-21.9375	0.207506390610129\\
-21.875	0.204148903587951\\
-21.8125	0.200790459485086\\
-21.75	0.19743369258756\\
-21.6875	0.194081237181398\\
-21.625	0.190735727552625\\
-21.5625	0.187399797987266\\
-21.5	0.184076082771346\\
-21.4375	0.180767081215464\\
-21.375	0.177474752728512\\
-21.3125	0.174200921743953\\
-21.25	0.170947412695251\\
-21.1875	0.167716050015873\\
-21.125	0.164508658139282\\
-21.0625	0.161327061498942\\
-21	0.158173084528318\\
};

\addplot [color=black, line width=2.0pt, forget plot]
  table[row sep=crcr]{%
-28	0.182822153813194\\
-27.9375	0.18602157959226\\
-27.875	0.189231155492022\\
-27.8125	0.192448684366439\\
-27.75	0.195671882869757\\
-27.6875	0.198898382564386\\
-27.625	0.202125731323439\\
-27.5625	0.20535139503836\\
-27.5	0.20857275964115\\
-27.4375	0.211787133449728\\
-27.375	0.214991749843798\\
-27.3125	0.218183770277408\\
-27.25	0.221360287632975\\
-27.1875	0.22451832992013\\
-27.125	0.227654864321104\\
-27.0625	0.230766801582721\\
-27	0.233851000753268\\
-26.9375	0.23690427426061\\
-26.875	0.239923393325986\\
-26.8125	0.242905093705859\\
-26.75	0.24584608175211\\
-26.6875	0.248743040778723\\
-26.625	0.251592637720921\\
-26.5625	0.254391530070545\\
-26.5	0.257136373069252\\
-26.4375	0.259823827138954\\
-26.375	0.262450565526793\\
-26.3125	0.265013282139847\\
-26.25	0.267508699542795\\
-26.1875	0.269933577089863\\
-26.125	0.272284719160565\\
-26.0625	0.274558983467167\\
-26	0.276753289400272\\
-25.9375	0.27886462637764\\
-25.875	0.280890062160261\\
-25.8125	0.282826751098777\\
-25.75	0.284671942272717\\
-25.6875	0.28642298748455\\
-25.625	0.288077349070429\\
-25.5625	0.289632607489561\\
-25.5	0.291086468654519\\
-25.4375	0.292436770965455\\
-25.375	0.293681492012078\\
-25.3125	0.294818754908461\\
-25.25	0.295846834227203\\
-25.1875	0.296764161501212\\
-25.125	0.29756933026334\\
-25.0625	0.298261100596348\\
-25	0.298838403168129\\
-24.9375	0.29930034272978\\
-24.875	0.299646201056963\\
-24.8125	0.299875439318042\\
-24.75	0.299987699855592\\
-24.6875	0.299982807371227\\
-24.625	0.299860769506978\\
-24.5625	0.299621776819939\\
-24.5	0.299266202150324\\
-24.4375	0.29879459938651\\
-24.375	0.298207701634106\\
-24.3125	0.297506418799395\\
-24.25	0.29669183460081\\
-24.1875	0.295765203025214\\
-24.125	0.294727944248804\\
-24.0625	0.293581640045257\\
-24	0.292328028706414\\
-23.9375	0.290968999503234\\
-23.875	0.289506586716925\\
-23.8125	0.287942963272179\\
-23.75	0.286280434006093\\
-23.6875	0.284521428607838\\
-23.625	0.282668494265278\\
-23.5625	0.28072428805567\\
-23.5	0.278691569118156\\
-23.4375	0.276573190646122\\
-23.375	0.274372091737568\\
-23.3125	0.272091289141436\\
-23.25	0.269733868937406\\
-23.1875	0.267302978185981\\
-23.125	0.264801816584767\\
-23.0625	0.262233628165734\\
-23	0.259601693066932\\
-22.9375	0.256909319410622\\
-22.875	0.254159835318143\\
-22.8125	0.251356581090053\\
-22.75	0.248502901578147\\
-22.6875	0.245602138773993\\
-22.625	0.242657624636501\\
-22.5625	0.23967267417894\\
-22.5	0.236650578833634\\
-22.4375	0.23359460011036\\
-22.375	0.230507963562318\\
-22.3125	0.227393853071331\\
-22.25	0.224255405461819\\
-22.1875	0.221095705450991\\
-22.125	0.217917780940665\\
-22.0625	0.214724598654179\\
-22	0.211519060119962\\
-21.9375	0.208303998001583\\
-21.875	0.205082172772381\\
-21.8125	0.201856269731232\\
-21.75	0.198628896354525\\
-21.6875	0.195402579978093\\
-21.625	0.192179765801594\\
-21.5625	0.188962815206739\\
-21.5	0.185754004379766\\
-21.4375	0.182555523227681\\
-21.375	0.179369474577034\\
-21.3125	0.176197873643354\\
-21.25	0.173042647758814\\
-21.1875	0.169905636345282\\
-21.125	0.166788591119588\\
-21.0625	0.163693176517573\\
-21	0.160620970323386\\
};
\addplot [color=red, line width=1.0pt, mark size=3.0pt, mark=*, mark repeat=8, mark phase=4, mark options={solid, red}]
  table[row sep=crcr]{%
-28	0.182460645707215\\
-27.9375	0.185679249762656\\
-27.875	0.188909578734653\\
-27.8125	0.192149187445675\\
-27.75	0.195395630718192\\
-27.6875	0.198646463374673\\
-27.625	0.201899240237588\\
-27.5625	0.205151516129407\\
-27.5	0.208400845872598\\
-27.4375	0.211644678601085\\
-27.375	0.214880040694606\\
-27.3125	0.21810385284435\\
-27.25	0.221313035741508\\
-27.1875	0.224504510077271\\
-27.125	0.22767519654283\\
-27.0625	0.230822015829375\\
-27	0.233941888628096\\
-26.9375	0.237031675801042\\
-26.875	0.240087998893694\\
-26.8125	0.243107419622389\\
-26.75	0.246086499703465\\
-26.6875	0.249021800853261\\
-26.625	0.251909884788115\\
-26.5625	0.254747313224364\\
-26.5	0.257530647878347\\
-26.4375	0.260256465241917\\
-26.375	0.262921400908989\\
-26.3125	0.265522105248991\\
-26.25	0.268055228631355\\
-26.1875	0.270517421425509\\
-26.125	0.272905334000883\\
-26.0625	0.275215616726906\\
-26	0.277444919973009\\
-25.9375	0.279590000098283\\
-25.875	0.281648037420471\\
-25.8125	0.283616318246976\\
-25.75	0.285492128885204\\
-25.6875	0.287272755642559\\
-25.625	0.288955484826446\\
-25.5625	0.290537602744268\\
-25.5	0.292016395703432\\
-25.4375	0.293389341203636\\
-25.375	0.294654681513765\\
-25.3125	0.295810850094998\\
-25.25	0.296856280408515\\
-25.1875	0.297789405915493\\
-25.125	0.298608660077114\\
-25.0625	0.299312476354556\\
-25	0.299899288208998\\
-24.9375	0.300367773618361\\
-24.875	0.300717588627522\\
-24.8125	0.300948633798099\\
-24.75	0.301060809691713\\
-24.6875	0.30105401686998\\
-24.625	0.30092815589452\\
-24.5625	0.300683127326951\\
-24.5	0.300318831728892\\
-24.4375	0.299835417686373\\
-24.375	0.299234025883073\\
-24.3125	0.29851604502708\\
-24.25	0.297682863826485\\
-24.1875	0.296735870989378\\
-24.125	0.295676455223849\\
-24.0625	0.294506005237987\\
-24	0.293225909739882\\
-23.9375	0.291837757934883\\
-23.875	0.290343941017364\\
-23.8125	0.288747050678962\\
-23.75	0.287049678611309\\
-23.6875	0.285254416506041\\
-23.625	0.283363856054791\\
-23.5625	0.281380588949193\\
-23.5	0.279307206880882\\
-23.4375	0.277146419600879\\
-23.375	0.274901409097751\\
-23.3125	0.272575475419452\\
-23.25	0.270171918613937\\
-23.1875	0.267694038729158\\
-23.125	0.265145135813072\\
-23.0625	0.262528509913631\\
-23	0.25984746107879\\
-22.9375	0.257105315555815\\
-22.875	0.254305504389224\\
-22.8125	0.251451484822846\\
-22.75	0.248546714100512\\
-22.6875	0.24559464946605\\
-22.625	0.242598748163291\\
-22.5625	0.239562467436064\\
-22.5	0.236489264528199\\
-22.4375	0.233382545185165\\
-22.375	0.230245509158984\\
-22.3125	0.227081304703319\\
-22.25	0.223893080071831\\
-22.1875	0.220683983518182\\
-22.125	0.217457163296034\\
-22.0625	0.214215767659049\\
-22	0.210962944860889\\
-21.9375	0.207701741804094\\
-21.875	0.204434799986717\\
-21.8125	0.20116465955569\\
-21.75	0.197893860657945\\
-21.6875	0.194624943440413\\
-21.625	0.191360448050028\\
-21.5625	0.188102914633719\\
-21.5	0.184854883338419\\
-21.4375	0.181618773078727\\
-21.375	0.178396517839903\\
-21.3125	0.175189930374878\\
-21.25	0.172000823436578\\
-21.1875	0.168831009777933\\
-21.125	0.16568230215187\\
-21.0625	0.16255651331132\\
-21	0.159455456009208\\
};

\end{axis}

\begin{axis}[%
width=2in,
height=1.7in,
at={(2.5in,1.4in)},
scale only axis,
xmin=0,
xmax=1,
ymin=0,
ymax=1,
axis line style={draw=none},
ticks=none,
axis x line*=bottom,
axis y line*=left,
legend style={legend cell align=left, align=left, draw=white!15!black, font = {\fontsize{8.3 pt}{12 pt}\selectfont}}
]

\addplot [color=black, line width=2.0pt, mark=none, mark options={solid, black}]
  table[row sep=crcr]{%
1	1\\
};
\addlegendentry{Exact}

\addplot [color=red, line width=1.0pt, mark size=3.0pt, mark=*, mark options={red}]
  table[row sep=crcr]{%
1	2\\
};
\addlegendentry{$\text{EC}_2(-0.18)$}

\addplot [color=mycolor2, line width=1.0pt, mark size=3.0pt, mark=triangle*, mark options={solid, mycolor2}]
  table[row sep=crcr]{%
1	3\\
};
\addlegendentry{DVD}

\addplot [color=blue, line width=1.0pt, mark size=3.0pt, mark=triangle*, mark options={solid, rotate=180, blue}]
  table[row sep=crcr]{%
1	4\\
};
\addlegendentry{PS}
\end{axis}
\begin{axis}[%
width=3.5in,
height=2.1in,
at={(1in,-1.5in)},
scale only axis,
xmin=20.5,
xmax=25,
ymin=0.22,
ymax=0.26,
xlabel={$x$},
ylabel={$-u$},
axis background/.style={fill=white},
legend style={legend cell align=left, align=left, draw=white!15!black}
]

\addplot [color=blue, line width=1.0pt, mark size=3.0pt, mark=triangle*, mark repeat=8, mark phase=5, mark options={solid, rotate=180, blue}]
  table[row sep=crcr]{%
20.5	0.194929609878331\\
20.5625	0.197217994614262\\
20.625	0.199482530007154\\
20.6875	0.201721162546693\\
20.75	0.203931838722564\\
20.8125	0.206112505024454\\
20.875	0.208261107942048\\
20.9375	0.210375593965032\\
21	0.212453909583092\\
21.0625	0.214494009130008\\
21.125	0.216493878315936\\
21.1875	0.218451510695125\\
21.25	0.220364899821826\\
21.3125	0.22223203925029\\
21.375	0.224050922534765\\
21.4375	0.225819543229503\\
21.5	0.227535894888753\\
21.5625	0.229198022317874\\
21.625	0.230804175326653\\
21.6875	0.232352654975986\\
21.75	0.233841762326771\\
21.8125	0.235269798439902\\
21.875	0.236635064376277\\
21.9375	0.237935861196791\\
22	0.23917048996234\\
22.0625	0.24033734501066\\
22.125	0.241435193786847\\
22.1875	0.242462897012836\\
22.25	0.24341931541056\\
22.3125	0.244303309701957\\
22.375	0.245113740608959\\
22.4375	0.245849468853502\\
22.5	0.246509355157522\\
22.5625	0.247092384130733\\
22.625	0.24759803593397\\
22.6875	0.248025914615848\\
22.75	0.248375624224981\\
22.8125	0.248646768809984\\
22.875	0.248838952419473\\
22.9375	0.248951779102062\\
23	0.248984852906365\\
23.0625	0.248937912826519\\
23.125	0.248811237638745\\
23.1875	0.248605241064785\\
23.25	0.24832033682638\\
23.3125	0.247956938645274\\
23.375	0.247515460243207\\
23.4375	0.246996315341922\\
23.5	0.246399917663162\\
23.5625	0.245726804291032\\
23.625	0.244978005759094\\
23.6875	0.244154675963274\\
23.75	0.243257968799499\\
23.8125	0.242289038163694\\
23.875	0.241249037951784\\
23.9375	0.240139122059697\\
24	0.238960444383357\\
24.0625	0.237714251188103\\
24.125	0.236402158216919\\
24.1875	0.235025873582199\\
24.25	0.23358710539634\\
24.3125	0.232087561771738\\
24.375	0.230528950820787\\
24.4375	0.228912980655883\\
24.5	0.227241359389423\\
24.5625	0.225515845324492\\
24.625	0.223738397526943\\
24.6875	0.221911025253318\\
24.75	0.220035737760161\\
24.8125	0.218114544304014\\
24.875	0.216149454141421\\
24.9375	0.214142476528925\\
25	0.212095620723067\\
};

\addplot [color=mycolor2, line width=1.0pt, mark size=3.0pt, mark=triangle*, mark repeat=8, mark phase=3, mark options={solid, mycolor2}]
  table[row sep=crcr]{%
20.5	0.208209948594319\\
20.5625	0.210487268999022\\
20.625	0.212728754577411\\
20.6875	0.214932049152671\\
20.75	0.217094796547988\\
20.8125	0.219214640586548\\
20.875	0.221289225091534\\
20.9375	0.223316193886133\\
21	0.225293190793531\\
21.0625	0.227217898064038\\
21.125	0.229088151656471\\
21.1875	0.230901825956774\\
21.25	0.232656795350889\\
21.3125	0.23435093422476\\
21.375	0.235982116964329\\
21.4375	0.23754821795554\\
21.5	0.239047111584335\\
21.5625	0.240476766421824\\
21.625	0.24183552777978\\
21.6875	0.243121835155146\\
21.75	0.244334128044859\\
21.8125	0.245470845945863\\
21.875	0.246530428355096\\
21.9375	0.247511314769499\\
22	0.248411944686013\\
22.0625	0.249230897531226\\
22.125	0.249967312450322\\
22.1875	0.250620468518131\\
22.25	0.251189644809484\\
22.3125	0.251674120399213\\
22.375	0.252073174362148\\
22.4375	0.252386085773121\\
22.5	0.252612133706963\\
22.5625	0.252750759570403\\
22.625	0.252802054097766\\
22.6875	0.252766270355277\\
22.75	0.252643661409159\\
22.8125	0.252434480325636\\
22.875	0.252138980170932\\
22.9375	0.25175741401127\\
23	0.251290034912875\\
23.0625	0.250737250296314\\
23.125	0.25010008499953\\
23.1875	0.249379718214811\\
23.25	0.248577329134445\\
23.3125	0.247694096950718\\
23.375	0.246731200855919\\
23.4375	0.245689820042335\\
23.5	0.244571133702253\\
23.5625	0.243376439563425\\
23.625	0.24210750949546\\
23.6875	0.240766233903431\\
23.75	0.239354503192411\\
23.8125	0.237874207767474\\
23.875	0.236327238033692\\
23.9375	0.23471548439614\\
24	0.233040837259889\\
24.0625	0.231305252815909\\
24.125	0.229510950398747\\
24.1875	0.227660215128845\\
24.25	0.225755332126646\\
24.3125	0.223798586512591\\
24.375	0.221792263407125\\
24.4375	0.219738647930688\\
24.5	0.217640025203723\\
24.5625	0.215498690824997\\
24.625	0.213316982306566\\
24.6875	0.211097247638815\\
24.75	0.208841834812123\\
24.8125	0.206553091816873\\
24.875	0.204233366643446\\
24.9375	0.201885007282225\\
25	0.199510361723591\\
};
\addplot [color=black, line width=2.0pt, forget plot]
  table[row sep=crcr]{%
20.5	0.197875402432771\\
20.5625	0.200168403079678\\
20.625	0.202435207067818\\
20.6875	0.204673725526582\\
20.75	0.206881856270377\\
20.8125	0.209057487311146\\
20.875	0.211198500504805\\
20.9375	0.213302775323649\\
21	0.215368192745825\\
21.0625	0.217392639252198\\
21.125	0.219374010919998\\
21.1875	0.221310217601896\\
21.25	0.223199187178343\\
21.3125	0.225038869870222\\
21.375	0.226827242598278\\
21.4375	0.228562313374961\\
21.5	0.230242125713952\\
21.5625	0.231864763041913\\
21.625	0.233428353096722\\
21.6875	0.23493107229594\\
21.75	0.236371150059103\\
21.8125	0.237746873067123\\
21.875	0.239056589442033\\
21.9375	0.240298712830279\\
22	0.24147172637286\\
22.0625	0.242574186545783\\
22.125	0.243604726854582\\
22.1875	0.244562061367108\\
22.25	0.245444988069131\\
22.3125	0.246252392028105\\
22.375	0.246983248350899\\
22.4375	0.247636624922278\\
22.5	0.248211684911638\\
22.5625	0.248707689036565\\
22.625	0.249123997572741\\
22.6875	0.249460072100904\\
22.75	0.249715476982669\\
22.8125	0.24988988055831\\
22.875	0.249983056060864\\
22.9375	0.249994882242209\\
23	0.249925343708205\\
23.0625	0.249774530961247\\
23.125	0.249542640150077\\
23.1875	0.249229972528021\\
23.25	0.24883693362222\\
23.3125	0.248364032117769\\
23.375	0.247811878462061\\
23.4375	0.247181183195848\\
23.5	0.246472755018844\\
23.5625	0.245687498598874\\
23.625	0.244826412134669\\
23.6875	0.24389058468352\\
23.75	0.242881193265914\\
23.8125	0.241799499760274\\
23.875	0.240646847601572\\
23.9375	0.239424658298531\\
24	0.238134427784457\\
24.0625	0.236777722617577\\
24.125	0.235356176046871\\
24.1875	0.233871483959982\\
24.25	0.232325400729749\\
24.3125	0.230719734976242\\
24.375	0.229056345261026\\
24.4375	0.227337135730392\\
24.5	0.225564051724111\\
24.5625	0.223739075365943\\
24.625	0.221864221151882\\
24.6875	0.219941531551601\\
24.75	0.217973072638106\\
24.8125	0.215960929760068\\
24.875	0.213907203270569\\
24.9375	0.21181400432547\\
25	0.209683450763696\\
};
\addplot [color=red, line width=1.0pt, mark size=3.0pt, mark=*, mark repeat=8, mark phase=6, mark options={solid, red}]
  table[row sep=crcr]{%
20.5	0.198076778625661\\
20.5625	0.200405860343322\\
20.625	0.202708564998032\\
20.6875	0.204982698763189\\
20.75	0.207226067812191\\
20.8125	0.209436478318433\\
20.875	0.211611736455314\\
20.9375	0.21374964839623\\
21	0.215848020314579\\
21.0625	0.217904671634062\\
21.125	0.219917474779595\\
21.1875	0.221884315426395\\
21.25	0.223803079249683\\
21.3125	0.225671651924676\\
21.375	0.227487919126596\\
21.4375	0.229249766530659\\
21.5	0.230955079812086\\
21.5625	0.232601806554292\\
21.625	0.234188141973478\\
21.6875	0.235712343194042\\
21.75	0.237172667340382\\
21.8125	0.238567371536896\\
21.875	0.239894712907982\\
21.9375	0.241152948578037\\
22	0.24234033567146\\
22.0625	0.243455238914093\\
22.125	0.244496453437565\\
22.1875	0.245462881974947\\
22.25	0.246353427259312\\
22.3125	0.247166992023733\\
22.375	0.247902479001281\\
22.4375	0.24855879092503\\
22.5	0.249134830528053\\
22.5625	0.249629639240751\\
22.625	0.250042813282848\\
22.6875	0.250374087571396\\
22.75	0.250623197023447\\
22.8125	0.250789876556055\\
22.875	0.250873861086271\\
22.9375	0.250874885531149\\
23	0.250792684807742\\
23.0625	0.250627140271541\\
23.125	0.250378719031801\\
23.1875	0.250048034636214\\
23.25	0.249635700632475\\
23.3125	0.249142330568278\\
23.375	0.248568537991315\\
23.4375	0.247914936449281\\
23.5	0.247182139489869\\
23.5625	0.246370889211641\\
23.625	0.245482441916632\\
23.6875	0.244518182457745\\
23.75	0.243479495687885\\
23.8125	0.242367766459954\\
23.875	0.241184379626856\\
23.9375	0.239930720041495\\
24	0.238608172556774\\
24.0625	0.237218212253262\\
24.125	0.235762675122195\\
24.1875	0.234243487382472\\
24.25	0.232662575252994\\
24.3125	0.231021864952663\\
24.375	0.229323282700378\\
24.4375	0.22756875471504\\
24.5	0.22576020721555\\
24.5625	0.223899608330538\\
24.625	0.221989093827554\\
24.6875	0.220030841383875\\
24.75	0.218027028676782\\
24.8125	0.215979833383552\\
24.875	0.213891433181464\\
24.9375	0.211764005747798\\
25	0.209599728759833\\
};
\end{axis}
\begin{axis}[%
width=2in,
height=1.7in,
at={(2.5in,-1.1in)},
scale only axis,
xmin=0,
xmax=1,
ymin=0,
ymax=1,
axis line style={draw=none},
ticks=none,
axis x line*=bottom,
axis y line*=left,
legend style={legend cell align=left, align=left, draw=white!15!black, font = {\fontsize{8.3 pt}{12 pt}\selectfont}}
]

\addplot [color=black, line width=2.0pt, mark=none, mark options={solid, black}]
  table[row sep=crcr]{%
1	1\\
};
\addlegendentry{Exact}

\addplot [color=red, line width=1.0pt, mark size=3.0pt, mark=*, mark options={red}]
  table[row sep=crcr]{%
1	2\\
};
\addlegendentry{$\text{EC}_2(-0.18)$}

\addplot [color=mycolor2, line width=1.0pt, mark size=3.0pt, mark=triangle*, mark options={solid, mycolor2}]
  table[row sep=crcr]{%
1	3\\
};
\addlegendentry{DVD}

\addplot [color=blue, line width=1.0pt, mark size=3.0pt, mark=triangle*, mark options={solid, rotate=180, blue}]
  table[row sep=crcr]{%
1	4\\
};
\addlegendentry{PS}
\end{axis}
\end{tikzpicture}%

%% file: 2solerr.tex
\definecolor{mycolor1}{rgb}{0.03137,0.43137,0.13725}%
\begin{tikzpicture}

\begin{axis}[%
width=3.5in,
height=2.1in,
at={(1in,1in)},
scale only axis,
xmin=-150,
xmax=150,
xlabel style={font=\color{white!15!black}},
xlabel={$x$},
ymin=-0.0015,
ymax=0.0015,
ylabel style={font=\color{white!15!black}},
ylabel={${u}_{\text{exact}}{-u}$},
xtick={-150,-120,-90,-60,-30,0,30,60,90,120,150},
xticklabels={$-150$,$-120$,$-90$,$-60$,$-30$,$0$,$30$,$60$,$90$,$120$,$150$},
axis background/.style={fill=white},
legend style={legend cell align=left, align=left, draw=white!15!black, font = {\fontsize{8.3 pt}{12 pt}\selectfont}}
]
\addplot [color=blue, line width=1.5pt]
  table[row sep=crcr]{%
-150	2.11336392670747e-07\\
-149.5	1.68800370765466e-07\\
-149	8.7044027207525e-08\\
-148.5	-4.17439626413952e-09\\
-148	-8.17616689411874e-08\\
-147.5	-1.28034162606389e-07\\
-147	-1.30059214547459e-07\\
-146.5	-8.55503658119029e-08\\
-146	-1.13171027672324e-08\\
-145.5	5.73495264692261e-08\\
-145	8.34132494897076e-08\\
-144.5	5.29856931431427e-08\\
-144	-1.07843979847022e-08\\
-143.5	-5.87431538142175e-08\\
-143	-4.71803203124628e-08\\
-142.5	3.0492317175859e-08\\
-142	1.32883834820275e-07\\
-141.5	1.89178694263442e-07\\
-141	1.39276880883707e-07\\
-140.5	-2.48722212756466e-08\\
-140	-2.41073887783409e-07\\
-139.5	-3.97671117380418e-07\\
-139	-3.8877693476444e-07\\
-138.5	-1.77063951242869e-07\\
-138	1.71978797064419e-07\\
-137.5	5.08207772883189e-07\\
-137	6.63722907824009e-07\\
-136.5	5.37478958471388e-07\\
-136	1.54303365679314e-07\\
-135.5	-3.35176484411356e-07\\
-135	-7.18596820364461e-07\\
-134.5	-8.20924107554961e-07\\
-134	-5.89053927661222e-07\\
-133.5	-1.18405143669083e-07\\
-133	3.92094491386195e-07\\
-132.5	7.33307145702325e-07\\
-132	7.80444430317277e-07\\
-131.5	5.42884652571228e-07\\
-131	1.47286984157644e-07\\
-130.5	-2.33395179551076e-07\\
-130	-4.65067138046553e-07\\
-129.5	-5.06745460290449e-07\\
-129	-4.0897378178443e-07\\
-128.5	-2.64486039289197e-07\\
-128	-1.44968186548435e-07\\
-127.5	-6.21342941045003e-08\\
-127	2.59043402147011e-08\\
-126.5	1.69782961312148e-07\\
-126	3.7569649366141e-07\\
-125.5	5.8012144646666e-07\\
-125	6.72518726778166e-07\\
-124.5	5.55812851410942e-07\\
-124	2.09313748094316e-07\\
-123.5	-2.83927376026931e-07\\
-123	-7.62324779218992e-07\\
-122.5	-1.05107713201955e-06\\
-122	-1.03532675508132e-06\\
-121.5	-7.0703653300881e-07\\
-121	-1.65550722352879e-07\\
-120.5	4.25352745158646e-07\\
-120	9.00073162115641e-07\\
-119.5	1.14262772432905e-06\\
-119	1.10841033290333e-06\\
-118.5	8.19401148055346e-07\\
-118	3.48494305669051e-07\\
-117.5	-1.96353183114087e-07\\
-117	-6.89995166042724e-07\\
-116.5	-1.01624321094574e-06\\
-116	-1.10056239825269e-06\\
-115.5	-9.3657177953692e-07\\
-115	-5.90629414623429e-07\\
-114.5	-1.76954884409195e-07\\
-114	1.89465317619599e-07\\
-113.5	4.41047096649491e-07\\
-113	5.8070969758026e-07\\
-112.5	6.63836081412133e-07\\
-112	7.46822087683146e-07\\
-111.5	8.34242418297535e-07\\
-111	8.62306817055975e-07\\
-110.5	7.34169559683282e-07\\
-110	3.85781856459051e-07\\
-109.5	-1.62152041233885e-07\\
-109	-7.99876134812781e-07\\
-108.5	-1.36999356056276e-06\\
-108	-1.72329911649748e-06\\
-107.5	-1.75840925179984e-06\\
-107	-1.43817572513051e-06\\
-106.5	-7.9486926899547e-07\\
-106	6.69667520913256e-08\\
-105.5	9.77407708687768e-07\\
-105	1.7323451211635e-06\\
-104.5	2.15028990524228e-06\\
-104	2.13352027794839e-06\\
-103.5	1.70224643950871e-06\\
-103	9.83532505462247e-07\\
-102.5	1.6246476675824e-07\\
-102	-5.78856471423258e-07\\
-101.5	-1.10948513165328e-06\\
-101	-1.37328330713214e-06\\
-100.5	-1.38904917261556e-06\\
-100	-1.23282667518672e-06\\
-99.5	-1.0078318999091e-06\\
-99	-8.08207461869206e-07\\
-98.5	-6.86052714183028e-07\\
-98	-6.33229617611811e-07\\
-97.5	-5.86215554807085e-07\\
-97	-4.53477261126456e-07\\
-96.5	-1.55051808821213e-07\\
-96	3.40768445407347e-07\\
-95.5	9.96485934500129e-07\\
-95	1.70654616589979e-06\\
-94.5	2.3199194348755e-06\\
-94	2.67766613711204e-06\\
-93.5	2.65535499727212e-06\\
-93	2.19917512310467e-06\\
-92.5	1.3460214641272e-06\\
-92	2.21457929994521e-07\\
-91.5	-9.84873427119996e-07\\
-91	-2.06047898250404e-06\\
-90.5	-2.81738438306319e-06\\
-90	-3.13380657619642e-06\\
-89.5	-2.97999499300075e-06\\
-89	-2.42226306215655e-06\\
-88.5	-1.60482413517164e-06\\
-88	-7.14407563194465e-07\\
-87.5	6.3604090771641e-08\\
-87	5.87109830717763e-07\\
-86.5	7.87594541950457e-07\\
-86	6.82166378907329e-07\\
-85.5	3.65236701909683e-07\\
-85	-1.75206281968938e-08\\
-84.5	-3.0638231452609e-07\\
-84	-3.6710914681863e-07\\
-83.5	-1.23798915646581e-07\\
-83	4.22066455363694e-07\\
-82.5	1.19024392005894e-06\\
-82	2.03795945912694e-06\\
-81.5	2.78991592825744e-06\\
-81	3.27549257966035e-06\\
-80.5	3.36528664409981e-06\\
-80	2.99948965776792e-06\\
-79.5	2.20243792753609e-06\\
-79	1.08058740221288e-06\\
-78.5	-1.95550574436449e-07\\
-78	-1.42191578893296e-06\\
-77.5	-2.39793284550389e-06\\
-77	-2.96305211545116e-06\\
-76.5	-3.02603720923772e-06\\
-76	-2.58167611151051e-06\\
-75.5	-1.71196350764522e-06\\
-75	-5.71712087678431e-07\\
-74.5	6.38632465898858e-07\\
-74	1.70799062813349e-06\\
-73.5	2.45115129478223e-06\\
-73	2.74028306197272e-06\\
-72.5	2.52602645200235e-06\\
-72	1.84473011532576e-06\\
-71.5	8.10884879865226e-07\\
-71	-4.03527479461042e-07\\
-70.5	-1.59873848668982e-06\\
-70	-2.58050004427546e-06\\
-69.5	-3.19115727378482e-06\\
-69	-3.33325254962087e-06\\
-68.5	-2.98203211403891e-06\\
-68	-2.18549242072003e-06\\
-67.5	-1.05289626859145e-06\\
-67	2.65338544533499e-07\\
-66.5	1.60207816015585e-06\\
-66	2.79755468652074e-06\\
-65.5	3.71908079089028e-06\\
-65	4.27370797488252e-06\\
-64.5	4.41288809672962e-06\\
-64	4.13010722226228e-06\\
-63.5	3.45386396218614e-06\\
-63	2.43892207964142e-06\\
-62.5	1.15862859147506e-06\\
-62	-2.99812251024985e-07\\
-61.5	-1.8367887169092e-06\\
-61	-3.3396631638794e-06\\
-60.5	-4.68233513910269e-06\\
-60	-5.72775869213497e-06\\
-59.5	-6.33580277933518e-06\\
-59	-6.37741929630703e-06\\
-58.5	-5.75426754821412e-06\\
-58	-4.42112653337872e-06\\
-57.5	-2.40675900119809e-06\\
-57	1.71869225213337e-07\\
-56.5	3.1065535209863e-06\\
-56	6.10908849658977e-06\\
-55.5	8.83758339338576e-06\\
-55	1.09356399101392e-05\\
-54.5	1.20794671708588e-05\\
-54	1.20261820611228e-05\\
-53.5	1.06555267654863e-05\\
-53	7.99766682731796e-06\\
-52.5	4.2417964594134e-06\\
-52	-2.76977293800633e-07\\
-51.5	-5.11177074593587e-06\\
-51	-9.75480383732992e-06\\
-50.5	-1.36981416193405e-05\\
-50	-1.64960686639106e-05\\
-49.5	-1.78203436714906e-05\\
-49	-1.75006949009234e-05\\
-48.5	-1.55454430537403e-05\\
-48	-1.21397684860501e-05\\
-47.5	-7.62228027292172e-06\\
-47	-2.44383271292326e-06\\
-46.5	2.88560369325371e-06\\
-46	7.85501483317113e-06\\
-45.5	1.20077068274124e-05\\
-45	1.49875766075941e-05\\
-44.5	1.65716015242877e-05\\
-44	1.66854176842039e-05\\
-43.5	1.54017021200178e-05\\
-43	1.29232291919818e-05\\
-42.5	9.55346605323738e-06\\
-42	5.65936533121047e-06\\
-41.5	1.6314029260238e-06\\
-41	-2.15509268891481e-06\\
-40.5	-5.37313888889944e-06\\
-40	-7.76885877508881e-06\\
-39.5	-9.17689090942139e-06\\
-39	-9.5258096609028e-06\\
-38.5	-8.83425462828514e-06\\
-38	-7.19997676929462e-06\\
-37.5	-4.78358464227856e-06\\
-37	-1.78931878247381e-06\\
-36.5	1.55447589408365e-06\\
-36	5.01355343042378e-06\\
-35.5	8.36314248668646e-06\\
-35	1.13994123367694e-05\\
-34.5	1.39469322646083e-05\\
-34	1.58623531567037e-05\\
-33.5	1.70351579317149e-05\\
-33	1.7386924891255e-05\\
-32.5	1.68708866835343e-05\\
-32	1.54741464522545e-05\\
-31.5	1.32250441661208e-05\\
-31	1.02079034915159e-05\\
-30.5	6.58640247847009e-06\\
-30	2.63371720814198e-06\\
-29.5	-1.23810916850819e-06\\
-29	-4.46576825249845e-06\\
-28.5	-6.36555023719731e-06\\
-28	-6.23887958534075e-06\\
-27.5	-3.57225536182226e-06\\
-27	1.68975657607873e-06\\
-26.5	8.88950694355461e-06\\
-26	1.65753724243323e-05\\
-25.5	2.27447974296968e-05\\
-25	2.54289720592049e-05\\
-24.5	2.34172458264403e-05\\
-24	1.67657363344675e-05\\
-23.5	6.80256948187852e-06\\
-23	-4.37935547015211e-06\\
-22.5	-1.46638039862645e-05\\
-22	-2.25225079102109e-05\\
-21.5	-2.72760637868297e-05\\
-21	-2.90108919243637e-05\\
-20.5	-2.83049737855312e-05\\
-20	-2.59291845465109e-05\\
-19.5	-2.26251958296425e-05\\
-19	-1.89869376230611e-05\\
-18.5	-1.54265404059273e-05\\
-18	-1.21905941789166e-05\\
-17.5	-9.39758869217266e-06\\
-17	-7.07847611447487e-06\\
-16.5	-5.21107212473823e-06\\
-16	-3.74524040648236e-06\\
-15.5	-2.61984672039056e-06\\
-15	-1.77329134483373e-06\\
-14.5	-1.14903674398176e-06\\
-14	-6.97989976568628e-07\\
-13.5	-3.79326502217628e-07\\
-13	-1.60148077267379e-07\\
-12.5	-1.43927567511706e-08\\
-12	7.8167409192989e-08\\
-11.5	1.32829227333844e-07\\
-11	1.61083616714939e-07\\
-10.5	1.71519316344466e-07\\
-10	1.70314336787577e-07\\
-9.5	1.61835159447499e-07\\
-9	1.49253790344061e-07\\
-8.5	1.34750151904945e-07\\
-8	1.1968023505744e-07\\
-7.5	1.04971694322889e-07\\
-7	9.12581700016762e-08\\
-6.5	7.88030462686782e-08\\
-6	6.77024140739247e-08\\
-5.5	5.806210894008e-08\\
-5	4.98505977913791e-08\\
-4.5	4.2917445248797e-08\\
-4	3.72093557857861e-08\\
-3.5	3.26839968159626e-08\\
-3	2.91799601088901e-08\\
-2.5	2.66016772827743e-08\\
-2	2.49604602772541e-08\\
-1.5	2.41754727080582e-08\\
-1	2.41506122187129e-08\\
-0.5	2.49320082927655e-08\\
0	2.65409472331095e-08\\
0.5	2.89082171463668e-08\\
1	3.20738714953896e-08\\
1.5	3.61322669930914e-08\\
2	4.10526530457368e-08\\
2.5	4.68212686864319e-08\\
3	5.35197010251287e-08\\
3.5	6.11092717770498e-08\\
4	6.94352020447599e-08\\
4.5	7.83858518820565e-08\\
5	8.7724202666329e-08\\
5.5	9.69242960718571e-08\\
6	1.05299801934824e-07\\
6.5	1.11918723424064e-07\\
7	1.15291910926822e-07\\
7.5	1.13335284681031e-07\\
8	1.0331790458503e-07\\
8.5	8.14397817091791e-08\\
9	4.25102899812881e-08\\
9.5	-2.02024442105819e-08\\
10	-1.15478428996257e-07\\
10.5	-2.54785953446302e-07\\
11	-4.52640277571158e-07\\
11.5	-7.27094373782072e-07\\
12	-1.10044255106564e-06\\
12.5	-1.59949091282745e-06\\
13	-2.25536769390272e-06\\
13.5	-3.10304073881593e-06\\
14	-4.17972993827648e-06\\
14.5	-5.52145532360659e-06\\
15	-7.15757213307078e-06\\
15.5	-9.1024507527307e-06\\
16	-1.13431805531736e-05\\
16.5	-1.38235265968734e-05\\
17	-1.6425420803004e-05\\
17.5	-1.89505382079819e-05\\
18	-2.11081811299918e-05\\
18.5	-2.25198710610691e-05\\
19	-2.27528679499212e-05\\
19.5	-2.1393405262804e-05\\
20	-1.81617378701426e-05\\
20.5	-1.30505108602852e-05\\
21	-6.44067046387531e-06\\
21.5	8.69269769554482e-07\\
22	7.77551351921479e-06\\
22.5	1.31278315071426e-05\\
23	1.6061673469292e-05\\
23.5	1.62706571783522e-05\\
24	1.40958966262927e-05\\
24.5	1.03917362381245e-05\\
25	6.23482879105741e-06\\
25.5	2.60843222826823e-06\\
26	1.81298255075291e-07\\
26.5	-7.68340408130452e-07\\
27	-3.04692902436399e-07\\
27.5	1.28009724584266e-06\\
28	3.58951916586137e-06\\
28.5	6.21930344038568e-06\\
29	8.81603027529099e-06\\
29.5	1.11029280561725e-05\\
30	1.28834004919712e-05\\
30.5	1.40328965382436e-05\\
31	1.44876303228847e-05\\
31.5	1.42342152806285e-05\\
32	1.3301450922619e-05\\
32.5	1.17548612566555e-05\\
33	9.69298472897409e-06\\
33.5	7.24393414146263e-06\\
34	4.56112072879863e-06\\
34.5	1.8172748488688e-06\\
35	-8.03433231045397e-07\\
35.5	-3.11613558706594e-06\\
36	-4.94836317557149e-06\\
36.5	-6.15339752082838e-06\\
37	-6.62461396839457e-06\\
37.5	-6.30902613959431e-06\\
38	-5.21614987902602e-06\\
38.5	-3.42349106482228e-06\\
39	-1.07683190249505e-06\\
39.5	1.61838232695102e-06\\
40	4.41123827378195e-06\\
40.5	7.02513532365855e-06\\
41	9.18461616077363e-06\\
41.5	1.06432926460212e-05\\
42	1.12106176707967e-05\\
42.5	1.07769790806877e-05\\
43	9.33051293770142e-06\\
43.5	6.96361981866741e-06\\
44	3.86981637269523e-06\\
44.5	3.27524061134792e-07\\
45	-3.32752603659955e-06\\
45.5	-6.73667575260825e-06\\
46	-9.55809039972909e-06\\
46.5	-1.15066003495145e-05\\
47	-1.2386884206281e-05\\
47.5	-1.21174572034554e-05\\
48	-1.07416521065298e-05\\
48.5	-8.42231880968849e-06\\
49	-5.42225247537958e-06\\
49.5	-2.0726050891718e-06\\
50	1.26792037140432e-06\\
50.5	4.25626895537497e-06\\
51	6.60466550268405e-06\\
51.5	8.11264654914346e-06\\
52	8.6863676343171e-06\\
52.5	8.34268161982262e-06\\
53	7.19936378580961e-06\\
53.5	5.45305582520113e-06\\
54	3.34828639299224e-06\\
54.5	1.14367017790135e-06\\
55	-9.19936607355548e-07\\
55.5	-2.64579664986225e-06\\
56	-3.89832222292121e-06\\
56.5	-4.60952571027e-06\\
57	-4.77641643800678e-06\\
57.5	-4.45083047735303e-06\\
58	-3.72421364006508e-06\\
58.5	-2.71135667865853e-06\\
59	-1.53557487292554e-06\\
59.5	-3.16443160662112e-07\\
60	8.38583517144613e-07\\
60.5	1.83880499779395e-06\\
61	2.61292983031022e-06\\
61.5	3.11066904164595e-06\\
62	3.30376324741092e-06\\
62.5	3.18750172972045e-06\\
63	2.78252393753237e-06\\
63.5	2.13535857681935e-06\\
64	1.31686804500742e-06\\
64.5	4.17960105364664e-07\\
65	-4.58545723872656e-07\\
65.5	-1.20960103277808e-06\\
66	-1.7452941697571e-06\\
66.5	-2.00219080051928e-06\\
67	-1.95457020189647e-06\\
67.5	-1.62062482544323e-06\\
68	-1.06185611181314e-06\\
68.5	-3.75596074874228e-07\\
69	3.18971490948058e-07\\
69.5	8.99681032745032e-07\\
70	1.26160094130371e-06\\
70.5	1.3350204961125e-06\\
71	1.09842475522177e-06\\
71.5	5.83651427945144e-07\\
72	-1.28383009197714e-07\\
72.5	-9.20468128236606e-07\\
73	-1.65825576662446e-06\\
73.5	-2.21304878065671e-06\\
74	-2.48400742776602e-06\\
74.5	-2.41536201015984e-06\\
75	-2.00590509123422e-06\\
75.5	-1.30914011124184e-06\\
76	-4.23489175486095e-07\\
76.5	5.25521236577076e-07\\
77	1.40734890136217e-06\\
77.5	2.10904310338277e-06\\
78	2.554309711804e-06\\
78.5	2.71489591679329e-06\\
79	2.61222263954019e-06\\
79.5	2.30951201328708e-06\\
80	1.8957698271045e-06\\
80.5	1.46474527560522e-06\\
81	1.0938596328665e-06\\
81.5	8.27675123308248e-07\\
82	6.6910351678911e-07\\
82.5	5.80711754508091e-07\\
83	4.96318881492954e-07\\
83.5	3.39941155774481e-07\\
84	4.77632880354177e-08\\
84.5	-4.11513642960854e-07\\
85	-1.02317794842332e-06\\
85.5	-1.72387292011209e-06\\
86	-2.41062386459971e-06\\
86.5	-2.95992990714014e-06\\
87	-3.25350890085139e-06\\
87.5	-3.20486210468953e-06\\
88	-2.78013496948683e-06\\
88.5	-2.00867207948176e-06\\
89	-9.80700184181622e-07\\
89.5	1.6816369468201e-07\\
90	1.28213837474633e-06\\
90.5	2.21407563134699e-06\\
91	2.85159354825345e-06\\
91.5	3.13496058051313e-06\\
92	3.06263771863803e-06\\
92.5	2.68411775236567e-06\\
93	2.08331024095204e-06\\
93.5	1.35713981400043e-06\\
94	5.9490820115514e-07\\
94.5	-1.35773981420147e-07\\
95	-7.93174233168556e-07\\
95.5	-1.35614440368505e-06\\
96	-1.81230640052021e-06\\
96.5	-2.14501505094977e-06\\
97	-2.32574851929398e-06\\
97.5	-2.31604841020966e-06\\
98	-2.07849666232041e-06\\
98.5	-1.59325760546171e-06\\
99	-8.7500048485692e-07\\
99.5	1.58973404832709e-08\\
100	9.72431351338418e-07\\
100.5	1.85238991597311e-06\\
101	2.50250565170453e-06\\
101.5	2.79226563053199e-06\\
102	2.64857447230621e-06\\
102.5	2.07949329308523e-06\\
103	1.1778439986564e-06\\
103.5	1.02170530772346e-07\\
104	-9.58682443776088e-07\\
104.5	-1.82501414089809e-06\\
105	-2.35897824103908e-06\\
105.5	-2.48619471692202e-06\\
106	-2.20575645516306e-06\\
106.5	-1.5880481329448e-06\\
107	-7.5929114005393e-07\\
107.5	1.26472364715889e-07\\
108	9.21958879044404e-07\\
108.5	1.51590183312851e-06\\
109	1.84592688219292e-06\\
109.5	1.89368557284093e-06\\
110	1.67148897912049e-06\\
110.5	1.21257056157737e-06\\
111	5.71270936644741e-07\\
111.5	-1.70496606908002e-07\\
112	-9.01225076470272e-07\\
112.5	-1.49153340246362e-06\\
113	-1.82026578751092e-06\\
113.5	-1.80684917906913e-06\\
114	-1.43791882844899e-06\\
114.5	-7.77705865025879e-07\\
115	4.24300137803738e-08\\
115.5	8.52159927507631e-07\\
116	1.47956696250479e-06\\
116.5	1.78855106979054e-06\\
117	1.70902975392808e-06\\
117.5	1.25484212289875e-06\\
118	5.25498298152193e-07\\
118.5	-3.11303703511187e-07\\
119	-1.05675001104163e-06\\
119.5	-1.53127851008801e-06\\
120	-1.62414610739409e-06\\
120.5	-1.32521407285314e-06\\
121	-7.26923022000177e-07\\
121.5	5.0570457476606e-09\\
122	6.82369459555606e-07\\
122.5	1.14753407927922e-06\\
123	1.31204409655102e-06\\
123.5	1.16850373082709e-06\\
124	7.78189739179465e-07\\
124.5	2.44699664741151e-07\\
125	-3.1339190285844e-07\\
125.5	-7.82927920708951e-07\\
126	-1.06969933154381e-06\\
126.5	-1.11018006954809e-06\\
127	-8.85142278345712e-07\\
127.5	-4.32821307224169e-07\\
128	1.46580444603437e-07\\
128.5	7.0454200221246e-07\\
129	1.07991420444345e-06\\
129.5	1.14956202235328e-06\\
130	8.76213832258218e-07\\
130.5	3.32252880802946e-07\\
131	-3.14985323863609e-07\\
131.5	-8.55938259958288e-07\\
132	-1.11230654634342e-06\\
132.5	-1.00296182653242e-06\\
133	-5.74630882584708e-07\\
133.5	1.65404441967569e-08\\
134	5.65218398319726e-07\\
134.5	8.93237989772158e-07\\
135	9.10451720903766e-07\\
135.5	6.37662274080404e-07\\
136	1.87798544474143e-07\\
136.5	-2.81282022209341e-07\\
137	-6.21411278888079e-07\\
137.5	-7.38131925447337e-07\\
138	-6.1401389049478e-07\\
138.5	-3.0801954634682e-07\\
139	6.7813629633172e-08\\
139.5	3.88092140774637e-07\\
140	5.57150014339123e-07\\
140.5	5.38439830471507e-07\\
141	3.58337335010447e-07\\
141.5	8.64632440004666e-08\\
142	-1.92661674305478e-07\\
142.5	-4.01539175493072e-07\\
143	-4.82058050617241e-07\\
143.5	-4.06181543586818e-07\\
144	-1.90064006770455e-07\\
144.5	9.73840906613643e-08\\
145	3.48675186193177e-07\\
145.5	4.6039156868894e-07\\
146	3.85455671198917e-07\\
146.5	1.61444606381357e-07\\
147	-1.06978495263924e-07\\
147.5	-3.02753557692786e-07\\
148	-3.54515593045044e-07\\
148.5	-2.6431889770262e-07\\
149	-9.37983965322707e-08\\
149.5	7.61133444834175e-08\\
150	1.84876300466688e-07\\
};
\addlegendentry{MC$_4(0.06)$}

\addplot [color=mycolor1, dashed, line width=1.5pt]
  table[row sep=crcr]{%
-150	-4.39916524555882e-07\\
-149.5	-4.67707029494451e-07\\
-149	-2.94993440374802e-07\\
-148.5	-2.04150407157075e-08\\
-148	2.23521763203195e-07\\
-147.5	3.38032583842564e-07\\
-147	2.99679699098156e-07\\
-146.5	1.57731675390417e-07\\
-146	-4.16063585007134e-09\\
-145.5	-1.16623989114307e-07\\
-145	-1.57029726683816e-07\\
-144.5	-1.49658770828426e-07\\
-144	-1.35260841678233e-07\\
-143.5	-1.31437492107757e-07\\
-143	-1.14551864732608e-07\\
-142.5	-4.04782166905039e-08\\
-142	1.09367933066682e-07\\
-141.5	2.93543875649604e-07\\
-141	4.18060975345133e-07\\
-140.5	3.85230878670758e-07\\
-140	1.54739663680148e-07\\
-139.5	-2.17874436483508e-07\\
-139	-5.86650231279755e-07\\
-138.5	-7.67954389382062e-07\\
-138	-6.26734912147567e-07\\
-137.5	-1.62294060628208e-07\\
-137	4.56612224067083e-07\\
-136.5	9.47737007193443e-07\\
-136	1.05417425786841e-06\\
-135.5	6.94134744734676e-07\\
-135	2.29088352076859e-08\\
-134.5	-6.4623229287385e-07\\
-134	-1.0195529492535e-06\\
-133.5	-9.81526535320356e-07\\
-133	-6.28251593614702e-07\\
-132.5	-1.63427439977312e-07\\
-132	2.51173597417473e-07\\
-131.5	5.70764164566441e-07\\
-131	8.20030956217551e-07\\
-130.5	9.85008134448487e-07\\
-130	9.59067582319155e-07\\
-129.5	6.0743540928958e-07\\
-129	-9.4396139760804e-08\\
-128.5	-9.65584203950969e-07\\
-128	-1.65970189986418e-06\\
-127.5	-1.83049697633296e-06\\
-127	-1.32291350087544e-06\\
-126.5	-2.72835783784226e-07\\
-126	9.40381061007636e-07\\
-125.5	1.85988146432936e-06\\
-125	2.15066731715325e-06\\
-124.5	1.7381471431085e-06\\
-124	8.23211103456902e-07\\
-123.5	-2.30483358584305e-07\\
-123	-1.07640589931666e-06\\
-122.5	-1.53309521324749e-06\\
-122	-1.61963462343028e-06\\
-121.5	-1.47049013266218e-06\\
-121	-1.20013814207501e-06\\
-120.5	-8.15688453345901e-07\\
-120	-2.34944223660348e-07\\
-119.5	6.07456282350347e-07\\
-119	1.63946455538648e-06\\
-118.5	2.59887781389855e-06\\
-118	3.09526550366656e-06\\
-117.5	2.77388220906017e-06\\
-117	1.5097901230089e-06\\
-116.5	-4.69462624543457e-07\\
-116	-2.61118029906648e-06\\
-115.5	-4.22127130594082e-06\\
-115	-4.72903005375694e-06\\
-114.5	-3.91515875940396e-06\\
-114	-2.00340570019644e-06\\
-113.5	4.21595377051477e-07\\
-113	2.63460497908918e-06\\
-112.5	4.03663018686918e-06\\
-112	4.35752030177625e-06\\
-111.5	3.71163191152419e-06\\
-111	2.49179209376143e-06\\
-110.5	1.16187821696153e-06\\
-110	5.03168909700986e-08\\
-109.5	-7.63138778718179e-07\\
-109	-1.42445733125619e-06\\
-108.5	-2.16539532378327e-06\\
-108	-3.11707674158756e-06\\
-107.5	-4.17033990007353e-06\\
-107	-4.95797384075087e-06\\
-106.5	-4.97845134465546e-06\\
-106	-3.81392724016933e-06\\
-105.5	-1.35101191723708e-06\\
-105	2.09063624052588e-06\\
-104.5	5.78335021934e-06\\
-104	8.7756531429352e-06\\
-103.5	1.01802906829263e-05\\
-103	9.46668722019379e-06\\
-102.5	6.65193273087562e-06\\
-102	2.32043551694422e-06\\
-101.5	-2.53617338048257e-06\\
-101	-6.80639184551829e-06\\
-100.5	-9.57978992686592e-06\\
-100	-1.039693275557e-05\\
-99.5	-9.34782272240012e-06\\
-99	-6.99191258012775e-06\\
-98.5	-4.13504145550014e-06\\
-98	-1.54464956845001e-06\\
-97.5	2.99725319913103e-07\\
-97	1.33889854451944e-06\\
-96.5	1.90795619072394e-06\\
-96	2.5646731227467e-06\\
-95.5	3.83572089881325e-06\\
-95	5.97056682270001e-06\\
-94.5	8.7880475333558e-06\\
-94	1.16697600237217e-05\\
-93.5	1.37084123744122e-05\\
-93	1.39699974192304e-05\\
-92.5	1.17909949996194e-05\\
-92	7.01965178166552e-06\\
-91.5	1.27125911110045e-07\\
-91	-7.84834036767936e-06\\
-90.5	-1.55147773053274e-05\\
-90	-2.14437281534983e-05\\
-89.5	-2.45077123480472e-05\\
-89	-2.41487135633973e-05\\
-88.5	-2.05083887122769e-05\\
-88	-1.43876000921462e-05\\
-87.5	-7.04637698126633e-06\\
-87	1.0399036389951e-07\\
-86.5	5.83782487410223e-06\\
-86	9.39662544574926e-06\\
-85.5	1.06477041413644e-05\\
-85	1.00840972416427e-05\\
-84.5	8.66886571663455e-06\\
-84	7.56481625472298e-06\\
-83.5	7.81789157068319e-06\\
-83	1.00714871321792e-05\\
-82.5	1.43792499351592e-05\\
-82	2.0159929022523e-05\\
-81.5	2.63047681107249e-05\\
-81	3.14125553478024e-05\\
-80.5	3.40988491482736e-05\\
-80	3.33108842954582e-05\\
-79.5	2.85796992311803e-05\\
-79	2.01549151656725e-05\\
-78.5	8.9930212549898e-06\\
-78	-3.39860056785633e-06\\
-77.5	-1.52302601050909e-05\\
-77	-2.47578722391522e-05\\
-76.5	-3.0605958975831e-05\\
-76	-3.2023230126258e-05\\
-75.5	-2.90226890367008e-05\\
-75	-2.23814594951056e-05\\
-74.5	-1.35020059486182e-05\\
-74	-4.16172837602799e-06\\
-73.5	3.80212939714928e-06\\
-73	8.81496628285786e-06\\
-72.5	9.83825206833508e-06\\
-72	6.54789766833621e-06\\
-71.5	-6.1066355605946e-07\\
-71	-1.04966789059334e-05\\
-70.5	-2.14613964342363e-05\\
-70	-3.16146447231313e-05\\
-69.5	-3.91257710001249e-05\\
-69	-4.25106471410677e-05\\
-68.5	-4.08612380714251e-05\\
-68	-3.39857383629499e-05\\
-67.5	-2.24424550782606e-05\\
-67	-7.46689202291266e-06\\
-66.5	9.1932334373607e-06\\
-66	2.55085289108904e-05\\
-65.5	3.94307875722822e-05\\
-65	4.91531456944824e-05\\
-64.5	5.33352204185185e-05\\
-64	5.12676412148952e-05\\
-63.5	4.29596126269947e-05\\
-63	2.91430617632555e-05\\
-62.5	1.11965632579459e-05\\
-62	-8.99930923854152e-06\\
-61.5	-2.92574099063223e-05\\
-61	-4.73088716467031e-05\\
-60.5	-6.10346952922422e-05\\
-60	-6.86822343674332e-05\\
-59.5	-6.90474742219567e-05\\
-59	-6.16077959095766e-05\\
-58.5	-4.65942745397831e-05\\
-58	-2.4997409169512e-05\\
-57.5	1.49488338639401e-06\\
-57	3.06234342649183e-05\\
-56.5	5.97398232376895e-05\\
-56	8.60312282960602e-05\\
-55.5	0.000106770998126753\\
-55	0.000119574422202918\\
-54.5	0.000122636889310936\\
-54	0.000114931214914438\\
-53.5	9.63428428518996e-05\\
-53	6.77259720940885e-05\\
-52.5	3.08704571978335e-05\\
-52	-1.16218676917317e-05\\
-51.5	-5.65430722890799e-05\\
-51	-0.000100349380140978\\
-50.5	-0.000139469924318264\\
-50	-0.00017062290117147\\
-49.5	-0.000191106419117797\\
-49	-0.000199033640164612\\
-48.5	-0.00019348787384788\\
-48	-0.000174582403768941\\
-47.5	-0.000143420710524321\\
-47	-0.000101963993680221\\
-46.5	-5.28230953019213e-05\\
-46	1.00026902334106e-06\\
-45.5	5.63935694073778e-05\\
-45	0.000110404574237226\\
-44.5	0.000160476068255769\\
-44	0.000204622588895991\\
-43.5	0.000241535521844366\\
-43	0.000270611778012403\\
-42.5	0.000291909955550901\\
-42	0.000306045274423237\\
-41.5	0.000314039901333283\\
-41	0.000317148163345791\\
-40.5	0.00031667647921592\\
-40	0.000313815900644152\\
-39.5	0.000309501505131622\\
-39	0.000304308159454602\\
-38.5	0.000298387095351587\\
-38	0.000291442992374689\\
-37.5	0.000282747360462911\\
-37	0.000271181332249524\\
-36.5	0.000255299740031964\\
-36	0.000233408601873234\\
-35.5	0.000203649810908429\\
-35	0.000164089787899535\\
-34.5	0.000112812917249513\\
-34	4.80255046653409e-05\\
-33.5	-3.18185405068436e-05\\
-33	-0.000127853636033341\\
-32.5	-0.000240578959959152\\
-32	-0.000369552672220247\\
-31.5	-0.000512985655664676\\
-31	-0.000667232022868933\\
-30.5	-0.000826207050445255\\
-30	-0.000980821757456701\\
-29.5	-0.00111861027643781\\
-29	-0.00122383166614184\\
-28.5	-0.00127841547775651\\
-28	-0.00126411507925667\\
-27.5	-0.00116602694556284\\
-27	-0.000977139402821708\\
-26.5	-0.000702832380157337\\
-26	-0.000363556422718847\\
-25.5	6.18310634009633e-06\\
-25	0.000363435891116426\\
-24.5	0.000665959383463699\\
-24	0.000881910934665164\\
-23.5	0.000996465155539472\\
-23	0.00101307599331868\\
-22.5	0.000949590440663173\\
-22	0.000831482151216978\\
-21.5	0.000685032346207981\\
-21	0.000532448588674794\\
-20.5	0.000389556226513932\\
-20	0.000265641705756978\\
-19.5	0.000164583371717972\\
-19	8.64538142114346e-05\\
-18.5	2.90503068991199e-05\\
-18	-1.09044528937355e-05\\
-17.5	-3.69468844008888e-05\\
-17	-5.23911551301853e-05\\
-16.5	-6.00938994717894e-05\\
-16	-6.23791892800617e-05\\
-15.5	-6.10559932334546e-05\\
-15	-5.74811232321188e-05\\
-14.5	-5.26382639898879e-05\\
-14	-4.72165124061081e-05\\
-13.5	-4.16803089617335e-05\\
-13	-3.63277510458575e-05\\
-12.5	-3.13370944683214e-05\\
-12	-2.68026215225143e-05\\
-11.5	-2.27615719970626e-05\\
-11	-1.92138959764277e-05\\
-10.5	-1.61364247618354e-05\\
-10	-1.34928091465422e-05\\
-9.5	-1.12403160349537e-05\\
-9	-9.33433538246697e-06\\
-8.5	-7.73124986750188e-06\\
-8	-6.39015493371635e-06\\
-7.5	-5.27378913672241e-06\\
-7	-4.34893619194362e-06\\
-6.5	-3.58648506917323e-06\\
-6	-2.96128030085729e-06\\
-5.5	-2.45185311107865e-06\\
-5	-2.04009563334133e-06\\
-4.5	-1.71091902280234e-06\\
-4	-1.45192180797359e-06\\
-3.5	-1.25308477649772e-06\\
-3	-1.10650150277035e-06\\
-2.5	-1.00614951638805e-06\\
-2	-9.47703845047203e-07\\
-1.5	-9.2839324207359e-07\\
-1	-9.46898341815284e-07\\
-0.5	-1.00329052316395e-06\\
0	-1.09901029721464e-06\\
0.5	-1.23688377324678e-06\\
1	-1.42117606103808e-06\\
1.5	-1.65768001467693e-06\\
2	-1.95383859648065e-06\\
2.5	-2.31889804479268e-06\\
3	-2.76408775495526e-06\\
3.5	-3.30282024513462e-06\\
4	-3.95090110935447e-06\\
4.5	-4.72673351704131e-06\\
5	-5.6514943351908e-06\\
5.5	-6.74924847813934e-06\\
6	-8.04695356453533e-06\\
6.5	-9.57428746397064e-06\\
7	-1.13632049449127e-05\\
7.5	-1.34470951189481e-05\\
8	-1.58593661455418e-05\\
8.5	-1.86312261390765e-05\\
9	-2.17883569058053e-05\\
9.5	-2.53460892847453e-05\\
10	-2.93025857691451e-05\\
10.5	-3.36294221453181e-05\\
11	-3.82588459417899e-05\\
11.5	-4.30668977500666e-05\\
12	-4.78515525735267e-05\\
12.5	-5.23051420249585e-05\\
13	-5.59806661586809e-05\\
13.5	-5.82523673852038e-05\\
14	-5.82723637015219e-05\\
14.5	-5.49275493390254e-05\\
15	-4.68047476949243e-05\\
15.5	-3.21775850285685e-05\\
16	-9.03583180087569e-06\\
16.5	2.48135518800124e-05\\
17	7.15366583770444e-05\\
17.5	0.000132914268947529\\
18	0.000209804194106475\\
18.5	0.000301424982371942\\
19	0.000404525402562556\\
19.5	0.000512603497984376\\
20	0.000615458181291079\\
20.5	0.000699447232333561\\
21	0.000748806527227891\\
21.5	0.000748169996942549\\
22	0.000685995211099272\\
22.5	0.00055806483630147\\
23	0.000369873681191651\\
23.5	0.000136826038402621\\
24	-0.000118146906454342\\
24.5	-0.000368952750871676\\
25	-0.000591316725763458\\
25.5	-0.000766860845740569\\
26	-0.000885281826824735\\
26.5	-0.000944443481042112\\
27	-0.000948884734526478\\
27.5	-0.000907548395157567\\
28	-0.000831470732371484\\
28.5	-0.000731905499140442\\
29	-0.000619065135707669\\
29.5	-0.000501450538558587\\
30	-0.000385633914768889\\
30.5	-0.000276334808491145\\
31	-0.000176651798034159\\
31.5	-8.83519125722231e-05\\
32	-1.21585355710961e-05\\
32.5	5.19913826430227e-05\\
33	0.000104730088631728\\
33.5	0.000147079829425342\\
34	0.000180304829512605\\
34.5	0.000205786704186486\\
35	0.000224917022659928\\
35.5	0.00023900327571149\\
36	0.000249187304541375\\
36.5	0.000256377755028655\\
37	0.000261199962528559\\
37.5	0.000263967607083829\\
38	0.000264680374349002\\
38.5	0.00026305068839301\\
39	0.000258560441356437\\
39.5	0.000250545745121643\\
40	0.000238304401568061\\
40.5	0.00022121745979648\\
41	0.000198873386182372\\
41.5	0.000171181498239933\\
42	0.000138460823678351\\
42.5	0.000101491716896586\\
43	6.1520467762658e-05\\
43.5	2.02116082081931e-05\\
44	-2.04517505033204e-05\\
44.5	-5.8313060413873e-05\\
45	-9.12205500020665e-05\\
45.5	-0.000117221220904993\\
46	-0.000134750242595939\\
46.5	-0.000142794579150131\\
47	-0.000141011470096888\\
47.5	-0.000129786660508197\\
48	-0.000110223593734074\\
48.5	-8.40623954782949e-05\\
49	-5.35354457683766e-05\\
49.5	-2.11736306184548e-05\\
50	1.04169682402406e-05\\
50.5	3.87848798629014e-05\\
51	6.18459337696514e-05\\
51.5	7.80574575279795e-05\\
52	8.65372665935822e-05\\
52.5	8.71171215272301e-05\\
53	8.0327582324672e-05\\
53.5	6.73183836975618e-05\\
54	4.97247366580146e-05\\
54.5	2.94947086274171e-05\\
55	8.69569265946434e-06\\
55.5	-1.06811479352843e-05\\
56	-2.6900520969216e-05\\
56.5	-3.86312354964825e-05\\
57	-4.50510658298842e-05\\
57.5	-4.59005469386395e-05\\
58	-4.14825873000582e-05\\
58.5	-3.26094013333342e-05\\
59	-2.05026768287628e-05\\
59.5	-6.65685323937478e-06\\
60	7.32136308649087e-06\\
60.5	1.98818000565428e-05\\
61	2.96913618040183e-05\\
61.5	3.57700469033731e-05\\
62	3.75880488532429e-05\\
62.5	3.51129819719037e-05\\
63	2.8801140433632e-05\\
63.5	1.95326365589425e-05\\
64	8.49670592898706e-06\\
64.5	-2.96029811976863e-06\\
65	-1.35064344233304e-05\\
65.5	-2.19939744086292e-05\\
66	-2.76052671877186e-05\\
66.5	-2.99544090874649e-05\\
67	-2.91293849643266e-05\\
67.5	-2.56669567880921e-05\\
68	-2.04619716889399e-05\\
68.5	-1.46226574000933e-05\\
69	-9.29239220629898e-06\\
69.5	-5.46478950624042e-06\\
70	-3.82135121296753e-06\\
70.5	-4.61863699849047e-06\\
71	-7.64487423224856e-06\\
71.5	-1.22550460985479e-05\\
72	-1.74803494820405e-05\\
72.5	-2.21947739574837e-05\\
73	-2.53109220177688e-05\\
73.5	-2.59711206249613e-05\\
74	-2.36997308847001e-05\\
74.5	-1.84888257262065e-05\\
75	-1.08011563611173e-05\\
75.5	-1.48943880579032e-06\\
76	8.35323182542795e-06\\
76.5	1.7583626303767e-05\\
77	2.52029169242528e-05\\
77.5	3.0518074119161e-05\\
78	3.32391498251317e-05\\
78.5	3.3494386524663e-05\\
79	3.17620302105683e-05\\
79.5	2.87352822697339e-05\\
80	2.51511540522939e-05\\
80.5	2.16217872773185e-05\\
81	1.85062912403904e-05\\
81.5	1.58521359753021e-05\\
82	1.34191358167917e-05\\
82.5	1.0779611273055e-05\\
83	7.47003341235985e-06\\
83.5	3.15647567659917e-06\\
84	-2.22823666723821e-06\\
84.5	-8.41199455439538e-06\\
85	-1.47967922499928e-05\\
85.5	-2.05586140209707e-05\\
86	-2.48128497645568e-05\\
86.5	-2.68071879084605e-05\\
87	-2.6095832291838e-05\\
87.5	-2.26522336983241e-05\\
88	-1.68919553603609e-05\\
88.5	-9.59926857737752e-06\\
89	-1.77486357140925e-06\\
89.5	5.55831956943735e-06\\
90	1.15438725046685e-05\\
90.5	1.56305114416422e-05\\
91	1.76303947091933e-05\\
91.5	1.76895404356089e-05\\
92	1.61867143666246e-05\\
92.5	1.35961942680106e-05\\
93	1.03579963586283e-05\\
93.5	6.79365006236017e-06\\
94	3.08837506827995e-06\\
94.5	-6.6288695124436e-07\\
95	-4.37039866808639e-06\\
95.5	-7.86419475435086e-06\\
96	-1.08475950300654e-05\\
96.5	-1.29155907045079e-05\\
97	-1.36405269095144e-05\\
97.5	-1.27018387538284e-05\\
98	-1.00179392749727e-05\\
98.5	-5.83309948728551e-06\\
99	-7.22532409973423e-07\\
99.5	4.49705663224012e-06\\
100	8.9336843644296e-06\\
100.5	1.18193719900346e-05\\
101	1.26866844498928e-05\\
101.5	1.14720930095812e-05\\
102	8.51717504223591e-06\\
102.5	4.46995252297256e-06\\
103	1.19115901199656e-07\\
103.5	-3.78758263378635e-06\\
104	-6.68904595553047e-06\\
104.5	-8.28586744015426e-06\\
105	-8.53717828129676e-06\\
105.5	-7.59870942702337e-06\\
106	-5.73806582908492e-06\\
106.5	-3.26396111408796e-06\\
107	-4.91238757695391e-07\\
107.5	2.2602422587721e-06\\
108	4.65518542390661e-06\\
108.5	6.34973406827767e-06\\
109	7.03477580043529e-06\\
109.5	6.5139401339362e-06\\
110	4.79107917982265e-06\\
110.5	2.1267557304689e-06\\
111	-9.73794060407274e-07\\
111.5	-3.85583959705037e-06\\
112	-5.87471607765771e-06\\
112.5	-6.57212511589851e-06\\
113	-5.80981074802942e-06\\
113.5	-3.81253780094955e-06\\
114	-1.10319183191612e-06\\
114.5	1.64874522470397e-06\\
115	3.81350184750848e-06\\
115.5	4.95993349165036e-06\\
116	4.93956200233197e-06\\
116.5	3.88124998442508e-06\\
117	2.11633683197358e-06\\
117.5	7.30115327646333e-08\\
118	-1.82275426188885e-06\\
118.5	-3.2179088582254e-06\\
119	-3.87255299279628e-06\\
119.5	-3.6817291697779e-06\\
120	-2.69126281816539e-06\\
120.5	-1.10672036396305e-06\\
121	7.21175314517097e-07\\
121.5	2.35399766468776e-06\\
122	3.3689780522348e-06\\
122.5	3.48423801703709e-06\\
123	2.66044673231696e-06\\
123.5	1.13266156761843e-06\\
124	-6.50239042133829e-07\\
124.5	-2.16388554255411e-06\\
125	-2.97721880047433e-06\\
125.5	-2.8902967286284e-06\\
126	-1.9894314222491e-06\\
126.5	-6.00172201111364e-07\\
127	8.38025506679754e-07\\
127.5	1.9191895971648e-06\\
128	2.38243248202242e-06\\
128.5	2.15835875578811e-06\\
129	1.35674793829957e-06\\
129.5	2.20821246369841e-07\\
130	-9.31668005383628e-07\\
130.5	-1.77642344516851e-06\\
131	-2.06261087454687e-06\\
131.5	-1.69507348248695e-06\\
132	-7.86652803173718e-07\\
132.5	3.58121009567862e-07\\
133	1.34136851635569e-06\\
133.5	1.82044897822588e-06\\
134	1.64208523297942e-06\\
134.5	9.00284843344825e-07\\
135	-1.07348496482894e-07\\
135.5	-1.00284118013216e-06\\
136	-1.47479477547961e-06\\
136.5	-1.38816119717457e-06\\
137	-8.18533962770142e-07\\
137.5	-7.34702089839797e-09\\
138	7.37187383244479e-07\\
138.5	1.15462481116873e-06\\
139	1.12213976038659e-06\\
139.5	6.8880876295496e-07\\
140	4.42340203233954e-08\\
140.5	-5.61872661382446e-07\\
141	-9.14066868774261e-07\\
141.5	-9.04236342941935e-07\\
142	-5.61655692755673e-07\\
142.5	-3.43789021402013e-08\\
143	4.69902975659241e-07\\
143.5	7.6280134080216e-07\\
144	7.44811896542912e-07\\
144.5	4.39221865207019e-07\\
145	-2.12129721582713e-08\\
145.5	-4.48210174074449e-07\\
146	-6.72442782256163e-07\\
146.5	-6.11275850965409e-07\\
147	-3.02219107420106e-07\\
147.5	1.14523493435817e-07\\
148	4.5863968326196e-07\\
148.5	5.88932593068017e-07\\
149	4.62686250005051e-07\\
149.5	1.49726430501472e-07\\
150	-2.0300294308877e-07\\
};
\addlegendentry{FD$_4$}

\end{axis}
\end{tikzpicture}%

%% file: FrascaCacciaHydon2020v2.bbl
\begin{thebibliography}{99}
\setlength{\itemsep}{0.0cm}\setlength{\parsep}{0cm}
\bibitem{PC} A.\,Ankiewicz, A.\,P.\,Bassom, P.\,A.\,Clarkson and E.\,Dowie, Conservation laws and integral relations for the Boussinesq equation, Stud. Appl. Math. 139 (2017), 104-128.

\bibitem{BK} A.\,Borhanifar and M.\,M.\,Kabir. New periodic and soliton solutions by application of exp-function method for nonlinear evolution equations, J. Comput. Appl. Math. 229 (2009), 158-167.

\bibitem{BR} T.\,J.\,Bridges and S.\,Reich, Multi-symplectic integrators: numerical schemes for Hamiltonian PDEs that conserve symplecticity, Phys. Lett. A 284 (2001), 184-193.

\bibitem{CIZ} M.\,P.\,Calvo, A.\,Iserles and A.\,Zanna, Numerical solution of isospectral flows, Math. Comp. 66 (1997), 1461-1486.

\bibitem{Cell} E.\,Celledoni, V.\,Grimm, R.\,I.\,McLachlan, D.\,I.\,McLaren, D.\,O'Neal, B.\,Owren and G.\,R.\,W.\,Quispel, Preserving energy resp. dissipation in numerical PDEs using the ``average vector field'' method, J. Comput. Phys. 231 (2012), 6770-6789.

\bibitem{Cell2} E.\,Celledoni, R.\,I.\,McLachlan, B.\,Owren and G.\,R.\,W.\,Quispel, Energy-preserving integrators and the structure of B-series, Found. Comput. Math. 10 (2010), 673-693.

\bibitem{Chen} M.\,Chen, L.\,Kong and Y.\,Hong, Efficient structure-preserving schemes for good Boussinesq equation, Math. Methods Appl. Sci. 41 (2018), 1743-1752.

\bibitem{Coop} G.\,J.\,Cooper, Stability of Runge--Kutta methods for trajectory problems, IMA J. Numer. Anal. 7 (1987), 1-13.

\bibitem{FCHydon} G.\,Frasca-Caccia and P.\,E.\,Hydon, Simple bespoke preservation of two conservation laws, IMA J. Numer. Anal. 40 (2020), 1294-1329.

\bibitem{FCHmkdv} G.\,Frasca-Caccia and P.\,E.\,Hydon, Locally conservative finite difference schemes for the modified KdV equation, J. Comput. Dyn. 6 (2019), 307-323.

\bibitem{FCHnls} G.\,Frasca-Caccia and P.\,E.\,Hydon, Numerical preservation of multiple local conservation laws. \url{arXiv:1903.12278v2}.

\bibitem{hydonbook} P.\,E.\,Hydon, Difference Equations by Differential Equation Methods, Cambridge University Press, Cambridge, 2014.

\bibitem{Ismail} M.\,S.\,Ismail and F.\,Mosally, A fourth order finite difference method for the Good Boussinesq equation, Abstr. Appl. Anal. DOI:10.1155/2014/323260.

\bibitem{kuper} B.\,A.\,Kuperschmidt, Discrete Lax equations and differential-difference calculus, Société mathématique de France, 1985.

\bibitem{MMM} V.\,S.\,Manoranjan, A.\,R.\,Mitchell and J.\,Ll.\,Morris, Numerical solutions of the Good Boussinesq equation, SIAM J. Sci. and Stat. Comput. 5 (1984), 946-957.


\bibitem{Matsuo} T.\,Matsuo, New conservative schemes with discrete variational derivatives for nonlinear wave equations, J. Comput. Appl. Math. 203 (2007), 32-56.

\bibitem{QMcL} R.\,I. McLachlan and G.\,R.\,W.\,Quispel, Discrete gradient methods have an energy conservation law, Discrete Contin. Dyn. Syst. 34 (2014), 1099-1104.

\bibitem{olverbook} P.\,J.\,Olver, Application of Lie Groups to Differential Equations 2nd edn., Springer, New York, 1993. 

\bibitem{McLaren} G.\,R.\,W. Quispel and D.\,I.\,McLaren, A new class of energy-preserving numerical integration methods, J. Phys. A 41 (2008), 045206.

\bibitem{SSC} J.\,M.\,Sanz-Serna and M.\,P.\,Calvo, Numerical Hamiltonian Problems, Chapman \& Hall, London, 1994.

\bibitem{Sun} Y.\,Sun, Quadratic invariants and multi-symplecticity of partitioned Runge--Kutta methods for Hamiltonian PDEs, Numer. Math. 106 (2007), 691-715.

\bibitem{Zeng} W.\,P.\,Zeng, L.\,Y.\,Huang and M.\,Z.\,Qin, The multi-symplectic algorithm for ``Good'' Boussinesq equation, Appl. Math. Mech. (English Ed.) 23 (2002), 835-841.
\end{thebibliography}
